\documentclass[EJP]{ejpecp}

\usepackage{enumerate}  

\SHORTTITLE{Excited random walk in a Markovian environment}

\TITLE{Excited random walk in a Markovian environment}

\AUTHORS{
  Nicholas~F.~Travers\footnote{Indiana University, Bloomington.
  \EMAIL{travers.nicholas@gmail.com}}}

\KEYWORDS{excited random walk; cookie random walk;
self-interacting random walk; Markovian environment;
random environment} 

\AMSSUBJ{60K35; 60K37; 60J80} 

\SUBMITTED{May 17, 2017}
\ACCEPTED{February 28, 2018}

\VOLUME{23}
\YEAR{2018}
\PAPERNUM{43}
\DOI{10.1214/18-EJP155}

\ABSTRACT{One dimensional excited random walk has been extensively studied
for bounded, i.i.d. cookie environments. In this case, many important
properties of the walk including transience or recurrence, positivity or
non-positivity of the speed, and the limiting distribution of the position of
the walker are all characterized by a single parameter $\delta$, the total
expected drift per site. In the more general case of stationary ergodic
environments, things are not so well understood. If all cookies are positive
then the same threshold for transience vs. recurrence holds, even if the
cookie stacks are unbounded. However, it is unknown if the threshold for
transience vs. recurrence extends to the case when cookies may be negative
(even for bounded stacks), and moreover there are simple counterexamples to
show that the threshold for positivity of the speed does not. It is thus
natural to study the behavior of the model in the case of Markovian
environments, which are intermediate between the i.i.d. and stationary
ergodic cases. We show here that many of the important results from the
i.i.d. setting, including the thresholds for transience and positivity of the
speed, as well as the limiting distribution of the position of the walker,
extend to a large class of Markovian environments. No assumptions are made
about the positivity of the cookies.}

\theoremstyle{plain}   
\theoremstyle{plain}   
\theoremstyle{plain} 	
\theoremstyle{plain} 	
\theoremstyle{plain} 	
\theoremstyle{plain} 	
\theoremstyle{plain}	
\theoremstyle{plain}	
\theoremstyle{plain}	
\theoremstyle{plain}	
\theoremstyle{plain}   
\theoremstyle{definition}
\newtheorem{Exa}{Example}

\def\norm#1{\|#1\|}

\def\ceil#1{\left\lceil#1\right\rceil}
\def\floor#1{\left\lfloor#1\right\rfloor}
\def\implies{\Longrightarrow}
\def\iff{\Longleftrightarrow}
\def\indicator{\mathds{1}}

\def\Var{\mbox{Var}}
\def\Cov{\mbox{Cov}}

\def\EB{\textbf{E}}

\def\PB{\textbf{P}}

\def\E{\mathbb{E}}

\def\N{\mathbb{N}}

\def\P{\mathbb{P}}

\def\R{\mathbb{R}}

\def\Z{\mathbb{Z}}

\def\FM{\mathcal{F}}
\def\GM{\mathcal{G}}
\def\HM{\mathcal{H}}
\def\IM{\mathcal{I}}

\def\KM{\mathcal{K}}

\def\MM{\mathcal{M}}

\def\RM{\mathcal{R}}
\def\SM{\mathcal{S}}

\def\YM{\mathcal{Y}}
\def\ZM{\mathcal{Z}}

\def\Ct{\widetilde{C}}

\def\Jt{\widetilde{J}}
\def\Kt{\widetilde{K}}

\def\Rt{\widetilde{R}}
\def\St{\widetilde{S}}

\def\Vt{\widetilde{V}}
\def\Wt{\widetilde{W}}
\def\Xt{\widetilde{X}}
\def\Yt{\widetilde{Y}}
\def\Zt{\widetilde{Z}}

\def\mt{\widetilde{m}}

\def\st{\widetilde{s}}

\def\xt{\widetilde{x}}

\def\Uh{\widehat{U}}
\def\Vh{\widehat{V}}

\def\Yh{\widehat{Y}}

\def\Zb{\overline{Z}}

\def\epsilont{\widetilde{\epsilon}}

\def\deltat{\widetilde{\delta}}

\def\zetat{\widetilde{\zeta}}

\begin{document}


\section{Introduction and statement of results}
\label{sec:Intro} Let $\Omega = [0,1]^{\Z \times \N}$. An \emph{excited
random walk} (ERW) started from position $x_0$ in a \emph{cookie environment}
$\omega = (\omega(x,i))_{x \in \Z, i \in \N} \in \Omega$ is an integer-valued
stochastic process $(X_n)_{n \geq 0}$ with probability measure
$P_{x_0}^{\omega}$ given by
\begin{align*}
& P_{x_0}^{\omega}(X_0 = x_0) = 1, \\
& P_{x_0}^{\omega}(X_{n+1} = X_n + 1|X_0, \ldots, X_n) = \omega(X_n, I_n), \\
& P_{x_0}^{\omega}(X_{n+1} = X_n - 1|X_0, \ldots, X_n) = 1 - \omega(X_n, I_n),
\end{align*}
where $I_n = |\{0 \leq m \leq n: X_m = X_n\}|$. The name cookie environment
comes from the following informal interpretation first given in
\cite{Zerner2005}. At each site $x \in \Z$ we initially place an infinite
stack of ``cookies''. The strength of the $i$-th cookie at site $x$ is
$\omega(x,i)$. Upon the $i$-th visit to $x$ the walker consumes the $i$-th
cookie there, which biases it to jump right with probability $\omega(x,i)$
and left with probability $1 - \omega(x,i)$.

A multi-dimensional form of this excited random walk model was first
introduced in \cite{Benjamini2003}, but with only a single cookie per site.
Subsequently the model was extended in \cite{Zerner2005} to the case of
infinite cookie stacks and random environments, for the 1-dimensional
setting, and has since been a topic of substantial interest
\cite{Kosygina2008, Basdevant2008, Basdevant2008b, Kosygina2011,
Kosygina2014, Amir2016, Kozma2016, Kosygina2017}. For a good and fairly
recent survey the reader is referred to \cite{Kosygina2013}.

Henceforth, we will be concerned only with the case of 1-dimensional ERW when
the environment $\omega$ is itself random. The setup is analogous to the case
of classical random walk in random environment. We first sample $\omega \in
\Omega$ according to some probability measure $\PB$ on $\Omega$. Then,
conditional on selecting the particular environment $\omega$, the random walk
proceeds as described above. For an initial position $x \in \Z$,
$P_x^{\omega}$ is called the \emph{quenched measure} of the walk in the
environment $\omega$, and the \emph{averaged} or \emph{annealed} measure
$P_x$ is defined by
\begin{align}
\label{eq:DefAnnealedProbMeasure}
P_x(\cdot) \equiv  \int P_x^{\omega}(\cdot) \PB(d\omega) = \EB[P_x^{\omega}(\cdot)].
\end{align}

Of course, to say anything meaningful about the behavior of the walk under
the averaged measure $P_x$ one must put some assumptions on the probability
measure $\PB$ over cookie environments. The following assumptions have often
been used in previous works: \begin{enumerate}[(POS)]
\item[(IID)] The sequence of cookie stacks $(\omega(x,\cdot))_{x \in \Z}$
    is i.i.d. under $\PB$.
\item[(SE)] The sequence of cookie stacks $(\omega(x,\cdot))_{x \in \Z}$ is
    stationary and ergodic under $\PB$. \
\item[(BD)] The cookie stacks are of uniform bounded height. That is, there
    exists a deterministic $M \in \N$ such that $\omega(x,i) = 1/2$ a.s.,
    for all $i > M$ and $x \in \Z$.
\item[(POS)]  The cookies are positive (or right biased): $\omega(x,i) \geq
    1/2$ a.s., for all $x \in \Z$ and $i \in \N$.
\item[(ELL)] The cookie environment is elliptic: $\omega(x,i) \in (0,1)$
    a.s., for all $x \in \Z$ and $i \in \N$.
 \end{enumerate}

If the measure $\PB$ satisfies either (IID) or (SE) and either (POS) or (BD)
then the total expected drift per site
\begin{align}
\label{eq:OriginalDeltaDef}
\delta \equiv \sum_{i = 1}^{\infty} \Big( 2 \EB[\omega(x,i)] - 1 \Big)
\end{align}
is well defined (possibly $+\infty$) and independent of $x$, and it turns out
that many key properties of the walk depend on this parameter $\delta$. The
(ELL) condition is only a technical assumption that is necessary to avoid
certain trivialities, and some weaker forms have been used instead in certain
instances.

\subsection{Summary of known results}

Here we list some important known results about the ERW model relevant to our
work. Some of these results hold (and were initially stated) with a somewhat
weaker form of ellipticity than the (ELL) condition given above, but these
differences will not be important for us. The following terminology will be
used in our statements, as well as in the statements of the new results that
follow in Section \ref{subsec:StatementNewResults}.
\begin{itemize}
\item We say that the random walk $(X_n)$ is \emph{recurrent} if $P_0(X_n =
    0, \mbox{ i.o.}) = 1$, \emph{right transient} if $P_0(X_n \rightarrow
    +\infty) = 1$, and \emph{left transient} if $P_0(X_n \rightarrow
    -\infty) = 1$.
\item We say that the random walk $(X_n)$ \emph{satisfies a law of large
    numbers} with \emph{velocity} $v \in \R$ if $\lim_{n \to \infty} X_n/n
    = v$, $P_0$ a.s. If $v \not= 0$, we say the random walk is
    \emph{ballistic}.
\end{itemize}

\begin{theorem}(Zero-One Law for Directional Transience, Theorem 1.2 of \cite{Amir2016}) \\
\label{thm:ZeroOneLawSE} Assume $\PB$ is (SE) and (ELL). Then $P_0(X_n
\rightarrow +\infty) \in \{0,1\}$ and $P_0(X_n \rightarrow -\infty) \in
\{0,1\}$.
\end{theorem}

\begin{theorem}(Law of Large Numbers, Theorem 4.1 of \cite{Kosygina2013}) \\
\label{thm:LLNSE} If $\PB$ is (SE) and the conclusion of Theorem
\ref{thm:ZeroOneLawSE} holds, then $(X_n)$ satisfies a law of large numbers
with some velocity $v \in [-1,1]$. In particular, if $\PB$ is (SE) and (ELL),
then $(X_n)$ satisfies a law of large numbers.
\end{theorem}

\begin{theorem}(Transience vs. Recurrence Threshold)
\label{thm:TransienceRecurrenceKnown}
\item[(i)] (Theorem 1 of \cite{Kosygina2008}). Assume $\PB$ satisfies (IID),
    (BD), and (ELL) and let $\delta$ be as in \eqref{eq:OriginalDeltaDef}.
    Then $(X_n)$ is recurrent if $\delta \in [-1,1]$, right transient if
    $\delta > 1$, and left transient if $\delta < -1$.
\item[(ii)] (Theorem 12 of \cite{Zerner2005}) Assume $\PB$ satisfies (SE),
    (POS), and (ELL) and let $\delta$ be as in \eqref{eq:OriginalDeltaDef}.
    Then $(X_n)$ is recurrent if $\delta \in [0,1]$ and right transient if
    $\delta > 1$.
\end{theorem}

\begin{theorem}(Ballisticity Threshold, Theorem 2 of \cite{Kosygina2008})\\
\label{thm:BallisticityKnown} Assume $\PB$ satisfies (IID), (BD), and (ELL)
and let $\delta$ be as in \eqref{eq:OriginalDeltaDef}. Also, let $v$ be the
velocity from Theorem \ref{thm:LLNSE}. Then $v = 0$ if $\delta \in [-2,2]$,
$v > 0$ if $\delta > 2$, and $v < 0$ if $\delta < -2$.
\end{theorem}

\begin{theorem} (Limit Laws, Theorem 6.5 of \cite{Kosygina2013})
\label{thm:LimitLawsKnown} Assume that $\PB$ satisfies (IID), (BD), and (ELL)
and let $\delta$ be as in \eqref{eq:OriginalDeltaDef}. Also, let $v$ be the
velocity from Theorem \ref{thm:LLNSE}. For $\alpha \in (0,2]$ and $b > 0$,
let $Z_{\alpha, b}$ be a random variable with totally asymmetric stable law
of index $\alpha$, defined by its characteristic function
\begin{align}
\label{eq:CharacteristicFunctionStableLaw}
\phi_{\alpha,b}(t) \equiv E[ e^{itZ_{\alpha, b}}] = \begin{cases}
\exp\big[ -b|t|^{\alpha} \big(1 - i \tan(\frac{\pi \alpha}{2}) \mathrm{sgn}(t) \big) \big], ~~\alpha \not= 1,\\
\exp\big[ -b|t| \big(1 + \frac{2i}{\pi} \log|t| \mathrm{sgn}(t) \big) \big], ~~~~~\alpha = 1.
\end{cases}
\end{align}
(Note that $Z_{2, b}$ is simply a normal random variable with mean 0 and
variance $2 b$.)

\begin{itemize}
\item[(i)] If $\delta \in (1,2)$ then there is some $b > 0$ such that
\begin{align}
\label{eq:LimitLawDeltaIn12}
\frac{X_n}{n^{\delta/2}} \rightarrow (Z_{\delta/2, b})^{-\delta/2} \mbox{ in distribution, as } n \rightarrow \infty.
\end{align}

\item[(ii)] If $\delta = 2$ then there exist constants $a, b > 0$ and a
    sequence $\Gamma(n) \sim an/\log(n)$ such that
\begin{align}
\label{eq:LimitLawDeltaEqual2}
\frac{X_n - \Gamma(n)}{a^2 n/\log^2(n)} \rightarrow -Z_{1,b} \mbox{ in distribution, as } n \rightarrow \infty.
\end{align}

\item[(iii)] If $\delta \in (2,4)$ then there is some $b > 0$ such that
\begin{align}
\label{eq:LimitLawDeltaIn24}
\frac{X_n - v n}{n^{2/\delta}} \rightarrow - Z_{\delta/2, b} \mbox{ in distribution, as } n \rightarrow \infty.
\end{align}

\item[(iv)] If $\delta = 4$ then there is some $b > 0$ such that
\begin{align}
\label{eq:LimitLawDeltaEqual4}
\frac{X_n - v n}{\sqrt{n \log(n)}} \rightarrow Z_{2, b} \mbox{ in distribution, as } n \rightarrow \infty.
\end{align}

\item[(v)] If $\delta > 4$ then there is some $b > 0$ such that
\begin{align}
\label{eq:LimitLawDeltaGreater4}
\frac{X_n - v n}{\sqrt{n}} \rightarrow Z_{2, b} \mbox{ in distribution, as } n \rightarrow \infty.
\end{align}
\end{itemize}
\end{theorem}

\begin{remark}
Analogous results to Theorem \ref{thm:LimitLawsKnown} hold in the case of
negative $\delta$ by symmetry.
\end{remark}

\begin{remark}
Reference \cite{Kosygina2013} is a review paper and many of the results there
were proved earlier in other places (either partially or completely). Case
(v) of Theorem \ref{thm:LimitLawsKnown} was originally proven in
\cite{Kosygina2008}. Cases (iii) and (iv) were originally proven in
\cite{Kosygina2011}. Case (i) under some stronger hypothesis was first shown
in \cite{Basdevant2008b}. The extension to the (IID), (BD), (ELL) case
follows from the same methods used in \cite{Kosygina2011}, as noted in that
work.
\end{remark}

\begin{remark}If the cookie stacks are unbounded and negative cookies are allowed then $\delta$ is not necessarily well defined.
However, recent work in \cite{Kozma2016} and \cite{Kosygina2017} has
considered extensions of the parameter $\delta$ to certain unbounded
environments with both positive and negative cookies. Specifically, in
\cite{Kozma2016} the authors consider deterministic, periodic cookie stacks,
which are the same at each site $x$. Then, in \cite{Kosygina2017} this model
is extended to the case where the cookie stack $(\omega(x,i))_{i \in \N}$ at
each site $x$ is a finite state Markov chain, started from some distribution
$\eta$ that is the same for all $x$. In this latter case of Markovian cookie
stacks, analogs of Theorems \ref{thm:TransienceRecurrenceKnown},
\ref{thm:BallisticityKnown}, and \ref{thm:LimitLawsKnown} are all proved for
the random walk $(X_n)$, in terms of some parameters $\delta$ and $\deltat$,
which are generalized versions of the $\delta$ given in
\eqref{eq:OriginalDeltaDef} and its negative. This work extends many of the
old results from (IID), (BD), and (ELL) environments by removing the (BD)
assumption while maintaining the (IID) assumption. Our Assumption (A)
presented in the following section will go in a somewhat different direction.
We will maintain the (BD) assumption, but weaken the (IID) condition.
\end{remark}

\subsection{Statement of new results}
\label{subsec:StatementNewResults} In light of Theorems
\ref{thm:TransienceRecurrenceKnown}-(i), \ref{thm:BallisticityKnown}, and
\ref{thm:LimitLawsKnown} we see that the behavior of the random walk $(X_n)$
is fairly well understood in the case of (IID), (BD), and (ELL) environments.
On the other hand, much less is known if the (IID) assumption is weakened to
(SE). A zero-one law for directional transience and law of large numbers
still hold, but it is generally not well understood when the walk will be
transient or ballistic. In the case that the cookies are all positive,
Theorem \ref{thm:TransienceRecurrenceKnown}-(ii) implies that the same
threshold for transience as the (IID) case holds, without the boundedness
assumption on the cookie stacks. However, even in the case of bounded cookie
stacks, it is still unknown whether the same transience/recurrence threshold
is always valid for (SE) environments with both positive and negative
cookies. Furthermore, there are simple counterexamples (given in Section
\ref{subsec:Counterexamples}) which indicate that the ballisticity threshold
of Theorem \ref{thm:BallisticityKnown} is not, in general, valid for (SE)
environments, even if (BD), (ELL), and (POS) are all assumed.

It is thus reasonable to wonder if there are classes of non-(IID)
environments for which an explicit characterization of the behavior of the
random walk $(X_n)$ is possible, similar to the (IID) case. A natural first
step in this direction is to consider Markovian environments, and we will
show that in fact Theorems \ref{thm:TransienceRecurrenceKnown}-(i),
\ref{thm:BallisticityKnown}, and \ref{thm:LimitLawsKnown} all extend to a
large class of Markovian environments.

\begin{definition}
Let $\SM$ be a countable (either finite or countably infinite) set, and let
$(S_n)$ be a discrete time Markov chain on state space $\SM$ with transition
matrix $\KM = \{\KM(s,s')\}_{s, s' \in \SM}$. The Markov chain $(S_n)$ is
said to be \emph{ergodic} if it is irreducible, aperiodic, and positive
recurrent. In this case (see \cite[Theorem 21.14]{Levin2009}), there exists a
unique stationary distribution $\pi$ on $\SM$ satisfying $\pi = \pi \KM$, and
\begin{align}
\label{eq:ErgodicMCDef}
\lim_{n \to \infty} \norm{ \KM^n(s, \cdot) - \pi}_{TV} = 0, \mbox{ for each } s \in \SM,
\end{align}
where $\norm{\mu - \nu}_{TV} \equiv \frac{1}{2} \norm{\mu - \nu}_1$ is the
total variational norm between two probability distributions $\mu$ and $\nu$
on $\SM$. The Markov chain $(S_n)$ is said to be \emph{uniformly ergodic} if
the rate of convergence in \eqref{eq:ErgodicMCDef} is uniform in the initial
state. That is, if
\begin{align}
\label{eq:UniformlyErgodicMCDef}
\lim_{n \to \infty} \Big[ \sup_{s \in \SM} \norm{ \KM^n(s, \cdot) - \pi}_{TV} \Big] = 0.
\end{align}
\end{definition}
Of course, all irreducible, aperiodic Markov chains on a finite state space
are uniformly ergodic, and many natural positive recurrent Markov chains on a
countably infinite state space are also uniformly ergodic. The positive
recurrent assumption is clearly necessary to give some asymptotic form of
stationarity for the Markov chain $(S_n)$, which, in attempt to partially
extend from (IID) environments to (SE) environments, is what we will want to
have.

For given $M \in \N$, we denote by $\SM^*_M$ the set of all elliptic cookie
stacks of height $M$:
\begin{align*}
\SM^*_M = \{s = (s(i))_{i \in \N}: s(i) \in (0,1) \mbox{ for } i = 1, \ldots, M \mbox{ and } s(i) = 1/2 \mbox{ for } i > M\}.
\end{align*}
If $\SM \subset \SM^*_M$ and $(S_k)_{k \in \Z}$ is a stochastic process
taking values in $\SM$, then we can define a bounded, elliptic cookie
environment $\omega = (\omega(k,i))_{k \in \Z, i \in \N}$ by
\begin{align}
\label{eq:DefineEnvironmentBySk}
\omega(k,i) = S_k(i) ~,~ k \in \Z \mbox{ and } i \in \N.
\end{align}
In the theorems below we will always make the following assumption on our
cookie environments.
\medskip

\noindent
\textbf{Assumption (A)}\\
The probability measure $\PB$ on cookie environments $\omega \in \Omega$ is
the probability measure obtained when $\omega$ is defined by
\eqref{eq:DefineEnvironmentBySk} and the process $(S_k)_{k \in \Z}$ is as
follows: $M \in \N$ is a positive integer, $\SM \subset \SM^*_M$ is a
countable set, and $(S_k)_{k \in \Z}$ and $(R_k)_{k \in \Z}$ are both
uniformly ergodic Markov chains on the state space $\SM$, where $R_k \equiv
S_{-k}$, $k \in \Z$.

\smallskip

\begin{remark}
If $(S_k)$ is an ergodic Markov chain on a countable state space $\SM$, then
the reversed process $(R_k)$ defined by $R_k = S_{-k}$, $k \in \Z$, is also
always an ergodic Markov chain. However, if $(S_k)$ is a uniformly ergodic
Markov chain, then the reversed process $(R_k)$ is not necessarily uniformly
ergodic. We assume explicitly that both the stack sequence $(S_k)$ and its
reversal $(R_k)$ are uniformly ergodic Markov chains, so that our Assumption
(A) is symmetric with respect to spatial directions of the model. This
condition is satisfied in many natural cases, e.g. when the Markov chain
$(S_k)$ is uniformly ergodic and reversible, or when the state space $\SM$ is
finite and $(S_k)$ is irreducible and aperiodic.

With our current methods of proof, this symmetric assumption of
bi-directional uniform ergodicity is necessary for a complete extension of
the results from the (IID) setting given in Theorems
\ref{thm:TransienceRecurrenceKnown}-(i), \ref{thm:BallisticityKnown}, and
\ref{thm:LimitLawsKnown}. If one assumes instead only that $(S_k)$ is
uniformly ergodic (or only that $(R_k)$ is uniformly ergodic) then the
transience/recurrence characterization of Theorem
\ref{thm:TransienceRecurrenceKnown}-(i) still extends fully (with only minor
modifications of the proof given in this paper). However, the results of
Theorem \ref{thm:BallisticityKnown} on ballisticity and Theorem
\ref{thm:LimitLawsKnown} on limiting distributions do not quite fully extend
(only one-sided versions are available, for $\delta > 0$ when $(R_k)$ is
uniformly ergodic, or for $\delta < 0$ when $(S_k)$ is uniformly ergodic).
\end{remark}

In the case Assumption $(A)$ is satisfied we will denote the transition
matrix for the Markov chain $(S_k)_{k \in \Z}$ by $\KM = \{\KM(s,s')\}_{s, s'
\in \SM}$ and the marginal distribution of $S_0$ by $\phi = (\phi(s))_{s \in
\SM}$. Together the pair $(\KM, \phi)$ completely determines the law of the
process $(S_k)$, and, hence, the probability measure $\PB$. If $\phi = \pi$
is the stationary distribution of the Markov chain, then $(S_k)_{k \in \Z}$
is a stationary and ergodic stochastic process. Thus, the (SE) assumption is
satisfied for the environment $\omega$, and the $\delta$ defined in
\eqref{eq:OriginalDeltaDef} is, indeed, well\vadjust{\goodbreak} defined. If
$\phi \not= \pi$, then the process $(S_k)_{k \in \Z}$ is no longer
stationary, so the definition \eqref{eq:OriginalDeltaDef} is no longer
directly applicable, because it is not the same for each site $x$.
Nevertheless, asymptotically, for $|k|$ large, $S_k$ has distribution close
to $\pi$, whatever the distribution $\phi$ on $S_0$ is. We, thus, extend the
definition of $\delta$ as follows when Assumption (A) is satisfied:
\begin{align}
\label{eq:DefDeltaMarkov}
\delta \equiv \sum_{s \in \SM} \pi(s) \cdot \delta(s) ~~~\mbox{ where }~~~ \delta(s) \equiv \sum_{i = 1}^{\infty} \big(2 s(i) - 1\big) = \sum_{i = 1}^{M} \big(2 s(i) - 1\big).
\end{align}
Our first theorem gives a threshold for transience vs. recurrence of the
random walk $(X_n)$, analogous to Theorem
\ref{thm:TransienceRecurrenceKnown}.
\begin{theorem}
\label{thm:TransienceRecurrence} Assume that Assumption (A) is satisfied for
the probability measure $\PB$ on cookie environments, and let $\delta$ be as
in \eqref{eq:DefDeltaMarkov}. Then the ERW $(X_n)_{n \geq 0}$ is recurrent if
$\delta \in [-1,1]$, right transient if $\delta > 1$, and left transient if
$\delta < -1$.
\end{theorem}
Our next theorem gives a threshold for ballisticity of the random walk,
analogous to Theorem \ref{thm:BallisticityKnown}.
\begin{theorem}
\label{thm:Ballisticity} Assume that Assumption (A) is satisfied for the
probability measure $\PB$ on cookie environments, and let $\delta$ be as in
\eqref{eq:DefDeltaMarkov}. Then there exists a deterministic $v \in [-1,1]$
such that $X_n/n \rightarrow v$, $P_0$ a.s. Moreover, $v = 0$ if $\delta \in
[-2,2]$, $v > 0$ if $\delta > 2$, and $v < 0$ if $\delta < -2$.
\end{theorem}
Our final theorem characterizes the limiting distribution of $(X_n)$,
analogous to Theorem \ref{thm:LimitLawsKnown}.
\begin{theorem}
\label{thm:LimitLaws} Assume that Assumption (A) is satisfied for the
probability measure $\PB$ on cookie environments, and let $\delta$ be as in
\eqref{eq:DefDeltaMarkov}. Also, let $v$ be the velocity of the random walk
$(X_n)$ as in Theorem \ref{thm:Ballisticity}. Then (i)-(v) of Theorem
\ref{thm:LimitLawsKnown} all hold.
\end{theorem}

\begin{remark}
Just as for Theorem \ref{thm:LimitLawsKnown} in the (IID) case, analogous
results to Theorem \ref{thm:LimitLaws} also hold in the case of negative
$\delta$. This follows from the fact that Assumption (A) is preserved when
spatial directions of the model are interchanged (see Remark
\ref{rem:InvarianceAssumptionAWhenInterchangeSpatialDirections} below).
\end{remark}

\begin{remark}
In fact, Assumption (A) in Theorems
\ref{thm:TransienceRecurrence}-\ref{thm:LimitLaws} can be weakened to the
somewhat more general \emph{hidden Markov} Assumption (B) given below. The
proofs are essentially unchanged; it is only for convenience of notation and
language that our proofs are written using Assumption (A) instead of (B).
\end{remark}

\noindent
\textbf{Assumption (B)}\\
The probability measure $\PB$ on cookie environments $\omega \in \Omega$ is
the probability measure obtained when $\omega$ is defined by
\eqref{eq:DefineEnvironmentBySk} and the process $(S_k)_{k \in \Z}$ is as
follows: $(Z_k)_{k \in \Z}$ is a uniformly ergodic Markov chain on a
countable state space $\ZM$ such that the reversed Markov chain $(Z_{-k})_{k
\in \Z}$ is also uniformly ergodic, $M \in \N$ is a positive integer, $\SM
\subset \SM^*_M$ is a countable set, and $S_k = f(Z_k)$ where $f: \ZM
\rightarrow \SM$ is some observation function ($f$ not necessarily 1-1).

\subsection{Counterexamples}
\label{subsec:Counterexamples}

In this section we present three closely related counterexamples, which
indicate that the threshold for positivity of the speed given in Theorem
\ref{thm:BallisticityKnown} (and also as a consequence the limiting
distributions presented in Theorem \ref{thm:LimitLawsKnown} when $\delta >
2$), do not extend as far as one might hope from the (IID), (BD), (ELL) case.
The first example shows that the threshold for ballisticity does not in
general hold for (POS) and (SE) environments. The second, which is a
modification of the first, shows that the same conclusion is true if one adds
(BD) and (ELL) to the (POS) and (SE) assumptions. These examples are known,
see e.g. \cite[Example 5.7]{Kosygina2013} (and also \cite[page
290]{Mountford2006} for an earlier and somewhat related example). The final
example, which we present here for the first time, is a modification of the
second. It shows that if one considers Markovian environments as in
Assumption (A) with bounded and elliptic cookie stacks, but assumes only that
the Markov chain $(S_k)_{k \in \Z}$ is ergodic (rather than uniformly
ergodic), then again the threshold of Theorem \ref{thm:BallisticityKnown} for
ballisticity is not in general valid\footnotemark{}. Thus, our uniformly
ergodic assumption on the Markov chains $(S_k)$ and $(R_k)$, or at least some
assumption beyond ergodicity, is necessary for the (IID) results to translate
completely. \footnotetext{In fact, although this is not explicit in their
descriptions, the first two examples are of the hidden Markov type as in
Assumption (B), but where the underlying Markov chain $(Z_k)_{k \geq 0}$ is
not uniformly ergodic.}

The following definition and lemma on monotonicity of the speed will be
needed for our examples.

\begin{definition} If $\omega_1, \omega_2 \in \Omega$ are two cookie environments, we say that $\omega_1$ \emph{dominates} $\omega_2$ if
$\omega_1(x,i) \geq \omega_2(x,i)$, for all $x \in \Z$ and $i \in \N$. If
$\PB_1$ and $\PB_2$ are two probability measures on $\Omega$, we say that
$\PB_1$ \emph{dominates} $\PB_2$ if there exists some joint probability
measure $\overline{\PB}$ on $\Omega \times \Omega$ with marginals $\PB_1$ and
$\PB_2$ such that $\overline{\PB}\big(\{(\omega_1, \omega_2) \in \Omega^2 :
\omega_1 \mbox{ dominates } \omega_2 \} \big) = 1$.
\end{definition}

\begin{lemma}[Special case of Proposition 4.2 from \cite{Kosygina2013}]
\label{lem:MonotonicityOfSpeed} Let $\PB_1$ and $\PB_2$ be two (SE) and (ELL)
probability measures on $\Omega$ such that $\PB_1$ \emph{dominates} $\PB_2$,
and let $v_1$ and $v_2$ be the corresponding velocities of the associated
ERWs (as in Theorem \ref{thm:LLNSE}). Then $v_1 \geq v_2$.
\end{lemma}

\begin{Exa}
\label{exmp1}
Define cookie stacks $s_0$ and $s_1$ by $s_0(i) = 1$ for all $i
\in \N$, $s_1(i) = 1/2$ for all $i \in \N$. That is, $s_0$ is an infinite
stack of completely right biased cookies, and $s_1$ is an infinite stack of
``placebo cookies'' which induce no bias on the random walker. Let $\SM =
\{s_0, s_1\}$, and let $(S_k)_{k \in \Z}$ be a stationary and ergodic process
taking values in $\SM$ such that the intervals between consecutive
occurrences of $s_0$ in the process $(S_k)$ are i.i.d. Define the random
environment $\omega \in \Omega$ by \eqref{eq:DefineEnvironmentBySk}. Also,
define $\tau_1 = \inf\{k > 0: S_k = s_0\}$ and $\tau_{i+1} = \inf\{k >
\tau_i: S_k = s_0\}$, for $i \geq 1$, and let $T_i = \inf\{n > 0: X_n =
\tau_i\}$. That is, $T_i$ is the time at which the walker first reaches the
$i$-th position to the right of $0$ which has an $s_0$ stack. Assume that the
distribution of the random variable $(\tau_2 - \tau_1)$, i.e. the
distribution of the distance between consecutive occurrences of stack $s_0$
in the process $(S_k)$, is such that $E(\tau_2 - \tau_1) < \infty$, but
$E\left((\tau_2 - \tau_1)^2\right) = \infty$. Then the following all hold.
\begin{enumerate}
\item $\delta = \infty > 2$.
\item $E(T_2 - T_1) = \infty$, since $E(T_2 - T_1|\tau_2 - \tau_1 = \ell) =
    \ell^2$.
\item $v \equiv \lim_{n \to \infty} X_n/n = \lim_{i \to \infty} \tau_i/T_i
    = \dfrac{E(\tau_2 - \tau_1)}{E(T_2 - T_1)} = 0$.
\end{enumerate}
\end{Exa}

\begin{Exa}
\label{exmp2} Let $M \in \N$ and $p \in (1/2,1)$, and modify Example
\ref{exmp1} so that the stack $s_0$ is defined as follows: $s_0(i) = p$ for
$i = 1, \ldots, M$ and $s_0(i) = 1/2$ for $i > M$. Then the natural coupling
between environments shows that the probability measure $\PB_1$ on cookie
environments from Example \ref{exmp1} dominates the new probability measure
$\PB_2$. Thus, by Lemma \ref{lem:MonotonicityOfSpeed}, $v_2 \leq v_1 = 0$,
where $v_1$ and $v_2$ are the associated velocities of the random walks. This
construction works for any $M \in \N$, and so we may choose $M$ sufficiently
large that $\delta = M(2p-1)/E(\tau_2 - \tau_1) > 2$ in the new modified
case. Then we have a (SE), (POS), (BD), (ELL) probability measure on cookie
environments with velocity 0, but $\delta > 2$.
\end{Exa}

\begin{Exa}
\label{exmp3} Let $s_0$, $s_1$ and the process $(S_k)_{k \in \Z}$ be as in
Example \ref{exmp2}. Assume that $M$ is sufficiently large that $\frac{M(2p
-1)}{E(\tau_2 - \tau_1)} > 3$, and that the distribution of $(\tau_2 -
\tau_1)$ has support on all of $\N$. Let $\tilde{\SM} = \{\tilde{s}_j : j
\geq 0 \}$ where the stacks $\tilde{s}_j$ are defined by $\tilde{s}_0 = s_0$
and, for $j \geq 1$,
\begin{align*}
\tilde{s}_j(1) = 1/2 - \frac{1}{3j} ~~\mbox{ and }~~ \tilde{s}_j(i) = 1/2,~ i > 1.
\end{align*}
Now, define a new process $(\St_k)_{k \in \Z}$ from the process $(S_k)_{k \in
\Z}$ by the following projection:
\begin{align*}
& \mbox{$\St_k = \tilde{s}_j$, where $j= k - \ell$ and $\ell = \sup\{m \leq k : S_m = s_0\}$.}
\end{align*}
Thus, $\St_k$ is equal to $\tilde{s}_0$ if $S_k = s_0$, and otherwise $\St_k$
equals $\st_j$ where $j$ is the number of time steps since the last
occurrence of $s_0$ for the process $(S_m)_{m \in \Z}$.  With this
construction $(\St_k)_{k \in \Z}$ is a Markov chain with transition
probabilities
\begin{align*}
& P(\St_{k+1} = \tilde{s}_0 | \St_k = \tilde{s}_j) = P(\tau_2 - \tau_1 = j + 1 | \tau_2 - \tau_1 > j), \\
& P(\St_{k+1} = \tilde{s}_{j+1} | \St_k = \tilde{s}_j) = P( \tau_2 - \tau_1 > j + 1 | \tau_2 - \tau_1 >j).
\end{align*}
Moreover, the state space $\tilde{\SM}$ of this Markov chain has only
bounded, elliptic stacks, and the Markov chain $(\St_k)$ itself is stationary
(since $(S_k)$ is), irreducible and aperiodic (since $(\tau_2 - \tau_1)$ has
support on all of $\N$), and positive recurrent (since $E(\tau_2 - \tau_1) <
\infty$). Finally, for the probability measure $\PB_3$ on environments
$\omega$ constructed from the $(\St_k)$ process and associated parameter
$\delta$ the following hold:
\begin{enumerate}
\item $\PB_3$ is dominated by the probability measure $\PB_2$ from Example
    \ref{exmp2}. Hence, by Lemma \ref{lem:MonotonicityOfSpeed},  $v_3 \leq
    v_2 = 0$, where $v_3$ and $v_2$ are associated velocities.
\item Let $\delta(\tilde{s}_j) = \sum_{i = 1}^{\infty} (2 \tilde{s}_j(i) -
    1)$ be the net drift in stack $\tilde{s}_j$, and let $\pi = (\pi_j)_{j
    \geq 0}$ be the stationary distribution of the Markov chain
    $(\widetilde{S}_k)$. Then $\delta(\tilde{s}_0) = M(2p - 1)$ and $-2/3 =
    \delta(\tilde{s}_1) < \delta(\tilde{s}_2) < \delta(\tilde{s}_3) <
    \ldots$, so
\begin{align*}
&\delta = \sum_{j=0}^{\infty} \pi_j \delta(\tilde{s}_j)
> \pi_0 \delta(\tilde{s}_0) + \sum_{j=1}^{\infty} \pi_j \delta(\tilde{s}_1) \\
&=  \frac{1}{E(\tau_2 - \tau_1)} M(2p - 1) + \left[1 - \frac{1}{E(\tau_2 - \tau_1)}\right] (-2/3) > 2.
\end{align*}
\end{enumerate}
In summary, the stack sequence $(\widetilde{S}_k)_{k \in \Z}$ (hence also the
reversed sequence $(\widetilde{R}_k)_{k \in \Z})$ are stationary and ergodic
Markov chains, and $\delta > 2$, but the velocity $v_3 \leq 0$. So, the
ballisticity threshold of Theorem \ref{thm:Ballisticity} does not extend to
the case when the sequence of cookies stacks is simply an ergodic Markov
chain rather than a uniformly ergodic one. Indeed, to the best of our
knowledge, it is not even known for this example whether the walk is
transient or recurrent since the environment does not satisfy the (POS)
condition (and, thus, Theorem \ref{thm:TransienceRecurrenceKnown}-(ii) is not
applicable).
\end{Exa}

\subsection{Outline of paper}
\label{subsec:OutlineOfPaper}

An outline of the remainder of the paper is as follows. In Section
\ref{subsec:Notation} we introduce some basic notation and conventions that
will be used throughout. In Section \ref{sec:BranchingProcesses} we review a
well known connection between ERW and certain branching processes, called the
\emph{forward branching process} and \emph{backward branching process}. We
also introduce some related processes, which are easier to analyze, and give
some concentration estimates and expectation and variance calculations for
these related processes. In Section
\ref{sec:ProofTheoremTransienceRecurrence} we prove Theorem
\ref{thm:TransienceRecurrence}. The proof is based on the connection between
the ERW and forward branching process, and follows the general approach used
in \cite{Kozma2016}. In Section
\ref{sec:ProofTheoremsBallisticityAndLimitLaws} we prove Theorems
\ref{thm:Ballisticity} and \ref{thm:LimitLaws}. The proofs are based on a
connection between the ERW and backward branching process and follow the
general approach used in \cite{Kosygina2011}, \cite{Kosygina2017}. Central to
these arguments is a diffusion approximation limit for the backward branching
process introduced in \cite{Kosygina2011}. The proofs of several technical
results are deferred to the appendices.

\subsection{Notation}
\label{subsec:Notation} The positive integers are denoted by $\N$ (as above),
and the non-negative integers by $\N_0$. The infimum of the empty set is
defined to be $\infty$, and $\sum_{i = j}^k z_i \equiv 0$, for any $j > k$
and sequence $(z_i)$. For a stochastic process $Z = (Z_n)_{n \geq 0}$,
$\tau_x^Z \equiv \inf\{n > 0: Z_n = x\}$. The same notation is also used for
a continuous time process $(Z(t))_{t \geq 0}$ with continuous sample paths.
For sequences of real numbers $(a_n), (b_n)$ we write $a_n \sim b_n$ if
$\lim_{n \to \infty} a_n/b_n = 1$. Similarly, $f(x) \sim g(x)$ means $\lim_{x
\to \infty} f(x)/g(x) = 1$, for real valued functions $f$ and $g$. Constants
of the form $c_i$ are assumed to carry over between the various propositions
and lemmas throughout. Other constants $C, c, C_i, K_i$, ... etc. are
particular only to the specific lemma or proposition where they are
introduced.

Unless otherwise specified it is assumed in the remainder of the paper that
Assumption (A) holds for the probability measure $\PB$ on cookie
environments. The stack height $M$, state set $\SM \subset \SM_M^*$, and
transition matrix $\KM$ for the Markov chain $(S_k)_{k \in \Z}$ are all
assumed to be fixed. The marginal distribution of $S_0$ (according to $\PB$)
is denoted by $\phi$, and the stationary distribution of the Markov chain
$(S_k)$ is denoted by $\pi$. Also, $\delta$ is given by
\eqref{eq:DefDeltaMarkov}. By a slight abuse of notation, we will use $\PB$
as the probability measure for the Markov chain $(S_k)$ itself, as well as
the environment $\omega$ derived from it according to
\eqref{eq:DefineEnvironmentBySk}. The probability measures $\PB_s$, $s \in
\SM$, and $\PB_{\pi}$ are the modified probability measures for the Markov
chain $(S_k)$ (or equivalently for the environment $\omega$) when $S_0 = s$
or $S_0 \sim \pi$:
\begin{align*}
\PB_s(\cdot) \equiv \PB(\cdot |S_0 = s) ~~\mbox{ and } ~~ \PB_{\pi}(\cdot) \equiv \sum_{s \in \SM} \pi(s) \PB_s(\cdot).
\end{align*}
Expectations with respect to $\PB_s$ and $\PB_{\pi}$ are denoted by $\EB_s$
and $\EB_{\pi}$, respectively, and the corresponding averaged measures for
the random walk $(X_n)$ started from position $x$ are
\begin{align*}
P_{x,s}(\cdot) \equiv \EB_s[ P_x^{\omega}(\cdot)] ~~ \mbox{ and } ~~ P_{x, \pi}(\cdot) \equiv \EB_{\pi}[ P_x^{\omega}(\cdot)].
\end{align*}
The probability measure $\P$ and corresponding expectation operator $\E$ will
be used generically for auxiliary random variables living on outside
probability spaces, separate from those of the environment $\omega$ and
random walk $(X_n)$.

\section{Branching processes}
\label{sec:BranchingProcesses}

In this section we introduce our main tool in the analysis of the ERW, which
is a connection with two related branching processes known as the forward
branching process and backward branching process. The definition of the
forward branching process is given in Section \ref{subsec:ForwardBP}, and the
definition of the backward branching process in Section
\ref{subsec:BackwardBP}. Some related branching processes which are easier to
analyze are introduced in Section \ref{subsec:RelatedProcesses}. Various
concentration estimates and expectation and variance calculations for some of
the related branching processes are given in Section
\ref{subsec:ExpectationVarianceConcentrationEstimates}.

\subsection{The forward branching process}
\label{subsec:ForwardBP}

The construction of both the forward and backward branching processes is
based on the \emph{coin tossing construction} of the ERW introduced in
\cite{Kosygina2008}. For a fixed environment $\omega \in \Omega$, we
initially flip an infinite sequence of coins at each site $k$, where the
$i$-th coin at site $k$ has probability $\omega(k,i)$ of landing heads. The
walker begins its walk at some given site $x$, and if it ever reaches site
$k$ for the $i$-th time, then it jumps right if the $i$-th coin toss at site
$k$ was heads and left otherwise. More formally, let $\xi_i^k$, $k \in \Z$
and $i \in \N$, be independent random variables such that $\xi_i^k$ has
Bernoulli distribution with parameter $p = \omega(k,i)$. Then the random walk
$(X_n)_{n \geq 0}$ started from position $x$ in the given environment
$\omega$ can be constructed from the $\xi_i^k$'s as follows:
\begin{align}
\label{eq:CoinTossingConstructionOfXn}
X_0 = x~\mbox{ and }~ X_{n+1} = (2 \xi_{I_n}^{X_n} - 1) + X_n
\end{align}
where $I_n = |\{0 \leq m \leq n:X_m = X_n\}|$. We will say that $\xi_i^k$ is
a \emph{success} if $\xi_i^k = 1$ (i.e. heads) and a failure if $\xi_i^k = 0$
(i.e. tails). The \emph{forward branching process} $(U_k)_{k \geq 0}$ started
from level $u_0 \in \N_0$ is defined by
\begin{align}
\label{eq:DefForwardBP}
U_0 = u_0 ~\mbox{ and }~ U_{k+1} = \inf \Big\{m : \sum_{i = 1}^m \indicator{\{\xi_i^{k+1} = 0\}} = U_k \Big\} - U_k.
\end{align}
That is, $U_{k+1}$ is the number of successes in the sequence
$(\xi_i^{k+1})_{i \in \N}$ before the $U_k$-th failure\footnotemark{}.
\footnotetext{In the case $U_k = \infty$ (when \eqref{eq:DefForwardBP} is no
longer directly meaningful) we will extend this interpretation, so that
$U_{k+1}$ is then defined to be the total number of successes in the sequence
$(\xi_i^{k+1})_{i \in \N}$.} If we define $G_i^k$ to be the number of
successes in the sequence $(\xi_j^{k})_{j \in \N}$ between the $(i-1)$-th and
$i$-th failures then we have, for each $k \geq 0$,
\begin{align}
\label{eq:DefForwardBPInTermsOfGeometrics}
U_{k+1} = \sum_{i = 1}^{U_k} G_i^{k+1}, \mbox{ where } (G_i^k)_{i > M, k \geq 1} \mbox{ are i.i.d. Geo(1/2) random variables.}
\end{align}
Thus, the process $(U_k)_{k \geq 0}$ may be seen as a type of branching
process with a time dependent migration term. More precisely, the $k$-th step
of the process may be interpreted as a combination of the following 3 things:
\begin{itemize}
\item First, $U_k \wedge M$ individuals emigrate out of the population
    before reproducing.
\item Then, all remaining individuals (if any) have a Geo(1/2) number of
    offspring independently.
\item Finally, $\sum_{i=1}^{U_k \wedge M} G_i^{k+1}$ individuals immigrate
    into the population after reproduction.
\end{itemize}

Now, this construction for the process $U = (U_k)_{k \geq 0}$ has been for a
fixed environment $\omega$, but one can also consider the same process when
the environment $\omega$ is first chosen randomly according to some
probability measure. We will denote by $P_{u_0}^{U,\omega}$ the probability
measure for the process $U$ started from level $u_0$ in a fixed environment
$\omega$, as constructed above, and by $P_{u_0,s}^U$ the probability measure
for the joint process $(U_k,S_k)_{k \geq 0}$ when $U_0 = u_0$ and $S_0 = s$.
That is, we first sample $(S_k)_{k \in \Z}$ according to $\PB_s$ to get an
environment $\omega = (\omega(k,i))_{k \in \Z, i \in \N} = (S_k(i))_{k \in
\Z, i \in \N}$, and then we sample $(U_k)_{k \geq 0}$ according to
$P_{u_0}^{U,\omega}$. This two step procedure gives a joint measure
on\footnotemark{} $(U_k,S_k)_{k \geq 0}$ which is the measure $P_{u_0,s}^U$.
The measures $P_{u_0,\pi}^U$ and $P_{u_0}^U$ are the averaged measures when
$S_0$ is distributed according to $\pi$ or $\phi$, respectively:
\begin{align*}
P_{u_0,\pi}^{U}(\cdot) \equiv \sum_{s \in \SM} \pi(s) P_{u_0,s}^U(\cdot) ~~~~\mbox{ and }~~~~ P_{u_0}^{U}(\cdot) \equiv \sum_{s \in \SM} \phi(s) P_{u_0,s}^U(\cdot).
\end{align*}
\footnotetext{Note that the process $(U_k)_{k \geq 0}$  depends only on
$S_k(i) = \omega(k,i)$, for $i,k \geq 1$. So, we do not need to completely
specify $\omega$ to construct $(U_k)_{k \geq 0}$. It is sufficient to
consider $(S_k)_{k \geq 0}$.} Under any of these measures $P_{u_0,s}^U$,
$P_{u_0,\pi}^{U}$, and $P_{u_0}^{U}$ the joint process $(U_k, S_k)_{k \geq
0}$ is a time-homogeneous Markov chain with transition probabilities
$p^U_{(u,r)(u',r')} \equiv \mbox{Prob}(U_{k+1} = u', S_{k+1} = r'|U_k = u,
S_k = r)$ given by
\begin{align}
\label{eq:UkSkMarkovChainTransitionProbs}
p^U_{(u,r)(u',r')} = \KM(r,r') P^{U,\omega_{r'}}_u(U_1 = u'),
\end{align}
where $\KM$ is the transition matrix for the Markov chain $(S_k)$ and
$\omega_{r'}$ is the deterministic environment with stack $r'$ at each site:
$\omega_{r'}(x,i) = r'(i)$, for all $x \in \Z$ and $i \in \N$.

The main interest in the forward branching process is its connection to a
related process $(U_k')_{k \geq 1}$ defined by
\begin{align*}
U_k' = |\{0 \leq n < \tau_0^X: X_n = k, X_{n+1} = k+1\}|.
\end{align*}
Clearly, survival of the process $(U_k')$, i.e. occurrence of the event
$\{U_k' > 0, \forall k > 0\}$, is closely related to right transience of the
random walk $(X_n)$. The following lemma is standard, but we will provide a
proof for the convenience of the reader.

\begin{lemma}
\label{lem:UkUkPrimeSurvival} Assume the process $(U_k)_{k \geq 0}$ is
started from level $u_0 = 1$ and that $(X_n)_{n \geq 0}$ is started from
position $X_0 = 1$. Then, for any realization of the random variables
$(\xi_i^k)_{k \in \Z, i \in \N}$, $U_k' \leq U_k$ for all $k \geq 1$.
Moreover, for any realization of the random variables $(\xi_i^k)_{k \in \Z, i
\in \N}$ such that $\tau_0^X < \infty$, $U_k' = U_k$ for all $k \geq 1$.
\end{lemma}

\noindent
\textbf{Note:} The lemma does not specify anything about the
probability measure on the environment $\omega$. The relation between $U_k$
and $U_k'$ is a deterministic function of the values of $(\xi_i^k)_{k \in \Z,
i \in \N}$, from which both processes $(U_k)$ and $(U_k')$ are constructed.

\begin{proof}
First fix a realization $(\xi_i^k)_{k \in \Z, i \in \N}$ such that $\tau_0^X
< \infty$. By definition $U_1'$ is the number of right jumps from site 1
before time $\tau_0^X$, which (if $\tau_0^X < \infty$) is simply the number
of right jumps from site $1$ before the first left jump from this site, or
equivalently the number of successes in the sequence $(\xi_i^1)_{i \in \N}$
before the first failure.  Since we assume $u_0 = 1$, the latter quantity is
exactly $U_1$, so we have $U_1' = U_1$. Now suppose that $U_k = U_k' = m \geq
0$, for some $k \geq 1$. Then, by the definition of the $(U_j')$ process, the
random walk $(X_n)$ must jump right from site $k$ exactly $m$ times prior to
time $\tau_0^X$. Thus, the walk must jump left from site $k+1$ exactly $m$
times prior to time $\tau_0^X$. Thus, the number of right jumps from site
$k+1$ prior to time $\tau_0^X$ is exactly the number of right jumps from site
$k+1$ before there are $m$ left jumps from it, or equivalently the number of
successes in the sequence $(\xi_i^{k+1})_{i \in \N}$ before the $m$-th
failure. Since $U_k = m$, this shows that $U_{k+1}' = U_{k+1}$. It follows,
by induction, that $U_k' = U_k$ for all $k \geq 1$.

Now, fix a realization $(\xi_i^k)_{k \in \Z, i \in \N}$ such that $\tau_0^X =
\infty$. Then, for each $k \geq 1$, $U_k'$ is simply the total number of
right jumps from site $k$ by the random walk $(X_n)$. Because the walk never
jumps left from site $1$, the total number of right jumps from site 1 is at
most the number of successes in the sequence $(\xi_i^1)_{i \in \N}$ before
the first failure. So, we have $U_1' \leq U_1$. Now suppose that $U_k = \ell$
and $U_k' = m$ for some $0 \leq m \leq \ell \leq \infty$ and $k \geq 1$. If
$m = 0$, then the walk never jumps right from site $k$, so it never reaches
site $k+1$, so $U'_{k+1}$ = 0, so $U'_{k+1} \leq U_{k+1}$. If $1 \leq m <
\infty$, then the walk jumps right from site $k$ exactly $m$ times total, so
either it jumps left from site $k+1$ exactly $(m-1)$ times or it jumps left
from site $k+1$ exactly $m$ times and never returns to site $k+1$ after the
$m$-th left jump. Either way, the total number of right jumps from site $k+1$
can be at most the number of successes in $(\xi_i^{k+1})_{i \in \N}$ before
the $m$-th failure, which is at most the number of successes in
$(\xi_i^{k+1})_{i \in \N}$ before the $\ell$-th failure. Hence, again,
$U_{k+1}' \leq U_{k+1}$. Finally, if $m = \infty$ then the walk must jump
right from site $k$ infinitely many times, so it must jump left from site
$k+1$ infinitely many times, so it must visit site $k+1$ infinitely often, so
the total number of right jumps from site $k+1$ is simply the total number of
successes in the sequence $(\xi_i^{k+1})_{i \in \N}$, which is equal to
$U_{k+1}$ (since $\ell = \infty$, with $m = \infty$ and $\ell \geq m$). Thus,
in all possible cases $U_{k+1}' \leq U_{k+1}$, so it follows, by induction,
that $U_k' \leq U_k$ for all $k \geq 1$.
\end{proof}

In Appendix \ref{app:ProofLemmaFiniteModificationArgument} we will prove the
following basic fact using a finite modification argument.

\begin{lemma}
\label{lem:FiniteModificationArgument} Define $A^+ = \{X_n > 0, \forall n > 0
\mbox{ and } \lim_{n \to \infty} X_n = +\infty \}$. Then, for each $x \in \N$
and $s \in \SM$, $P_{x,s}(A^+) > 0$ if and only if $P_{x,s}( X_n \rightarrow
+\infty) > 0$.
\end{lemma}

Using this fact along with Lemma \ref{lem:UkUkPrimeSurvival} and Theorem
\ref{thm:ZeroOneLawSE}, we now establish an explicit criteria relating
transience/recurrence of the random walk $(X_n)$ to the forward branching
process $(U_k)$. This criteria will be used to prove Theorem
\ref{thm:TransienceRecurrence} in Section
\ref{sec:ProofTheoremTransienceRecurrence}. For the statement of the lemma
recall that $P_{u_0,s}^U(\cdot)$ is the joint probability measure for
$(U_k,S_k)_{k \geq 0}$ when $U_0 = u_0$ and $S_0 = s$.

\begin{lemma}
\label{lem:TransienceRecurrenceInTermsOfSurvivalUk} The following hold:
\begin{align*}
& \mbox{ If $\exists s \in \SM$ such that } P_{1,s}^{U}(U_k > 0, \forall k > 0) > 0, \mbox{ then } P_0(X_n \rightarrow +\infty) = 1. \\
& \mbox{ If $\exists s \in \SM$ such that } P_{1,s}^{U}(U_k > 0, \forall k > 0) = 0, \mbox{ then } P_0(X_n \rightarrow +\infty) = 0.
\end{align*}
\end{lemma}

\begin{proof}
We will say that ``$U_k$ survives'' if $U_k > 0$ for all $k > 0$, and
similarly for $U_k'$. Also, we extend the probability measure $P_{x,s}$ for
the random walk $(X_n)$ to the process $(U_k')$ derived from it. By Theorem
\ref{thm:ZeroOneLawSE}, $P_{1,\pi}(X_n \rightarrow + \infty) = P_{0,\pi}(X_n
\rightarrow + \infty) \in \{0,1\}$. Since $\pi(s) > 0$ for all $s \in \SM$,
this implies that either
\begin{align*}
& \mbox{ (a) $P_{1,s}(X_n \rightarrow +\infty) = P_{0,s}(X_n \rightarrow +\infty) = 1$, for all $s \in \SM$, or  } \\
& \mbox{ (b) $P_{1,s}(X_n \rightarrow +\infty) = P_{0,s}(X_n \rightarrow +\infty) = 0$, for all $s \in \SM$. }
\end{align*}
We consider these two cases separately. \\

\noindent
\emph{Case (a):} \\
In this case, it follows from Lemma \ref{lem:FiniteModificationArgument} that
$P_{1,s}(A^+) > 0$, for each $s \in \SM$. Hence, $P_{1,s}(U_k' \mbox{
survives}) > 0$, for each $s \in \SM$. By Lemma \ref{lem:UkUkPrimeSurvival}
this implies
$P_{1,s}^{U}(U_k \mbox{ survives}) > 0$, for each $s \in \SM$. \\

\noindent
\emph{Case (b):} \\
Since $\omega(x,i) = 1/2$ for each $i > M$ and $x \in \Z$, $\PB_s$ a.s., we
have $P_{1,s}(\liminf_{n \to \infty} X_n = x) = 0$, for each $x \in \Z$.
Thus, if $P_{1,s}(X_n \rightarrow +\infty) = 0$ for each $s \in \SM$, then
$P_{1,s}(\liminf_{n \to \infty} X_n = -\infty) = 1$ for each $s \in \SM$, and
in particular, $P_{1,s}(\tau_0^X < \infty) = 1$ for each $s \in \SM$. By
Lemma \ref{lem:UkUkPrimeSurvival} and the definition of the $(U_k')$ process
this
implies $P_{1,s}^{U}(U_k \mbox{ survives}) = 0$ for each $s \in \SM$. \\

\noindent Thus, we have established the following dichotomy: Either
\begin{align*}
& \mbox{(a') $P_{1,s}^{U}(U_k \mbox{ survives}) > 0$ for each $s \in \SM$ and (a) holds, or} \\
& \mbox{(b') $P_{1,s}^{U}(U_k \mbox{ survives}) = 0$ for each $s \in \SM$ and (b) holds.}
\end{align*}
Since $P_0(X_n \rightarrow +\infty) = 1$ if (a) holds, and $P_0(X_n
\rightarrow +\infty) = 0$ if (b) holds, this establishes the lemma.
\end{proof}

\subsection{The backward branching process}
\label{subsec:BackwardBP}

Let the random variables $(\xi_i^k)_{k \in \Z, i \in \N}$ be as above in
Section \ref{subsec:ForwardBP}. We continue to assume the random walk $(X_n)$
is constructed from these random variables using
\eqref{eq:CoinTossingConstructionOfXn}. Also, we recall that $(R_k)_{k \in
\Z}$ is the spatial reversal of the stack sequence $(S_k)_{k \in \Z}$:
\begin{align}
\label{eq:DefRk}
R_k = S_{-k} ~,~ k \in \Z.
\end{align}
The \emph{backward branching process} $(V_k)_{k \geq 0}$ started from level
$v_0 \in \N_0$ is defined by
\begin{align}
\label{eq:DefBackwardBP}
V_0 = v_0 ~\mbox{ and }~ V_{k+1} = \inf \Big\{ m : \sum_{i=1}^m \indicator\{\xi_i^{-(k+1)} = 1\} = V_k + 1 \Big\} - (V_k + 1).
\end{align}
That is, $V_{k+1}$ is the number of failures in the sequence
$(\xi_i^{-(k+1)})_{i \in \N}$ (i.e. at stack $R_{k+1}$) before there are $V_k
+ 1$ successes. If we let $H_i^k$ be the number of failures in the sequence
$(\xi_j^{-k})_{j \in \N}$ between the $(i-1)$-th and $i$-th successes then
\begin{align}
\label{eq:DefBackwardBPInTermsOfGeometrics}
V_{k+1} = \sum_{i = 1}^{V_k + 1} H_i^{k+1}, \mbox{ where } (H_i^k)_{i > M, k \geq 1} \mbox{ are i.i.d. Geo(1/2) random variables. }
\end{align}
Thus, by similar reasoning as for the forward branching process $(U_k)_{k
\geq 0}$, this process $(V_k)_{k \geq 0}$ may also be seen as a type of
branching process with a time dependent migration term.

In fact, the definition of the backward branching process $(V_k)$ is almost
exactly symmetric to the definition of the forward branching process $(U_k)$
with successes replaced by failures and $S_k$ replaced by $R_k$, but there is
one notable difference: In the backward process we count failures until $V_k
+ 1$ successes, whereas in the forward process we only count successes until
$U_k$ failures. This ``$+1$'' is important, because it means that $0$ is not
an absorbing state for the process $(V_k)$, as it is for the process $(U_k)$.

Our interest in the backward branching process stems from the following lemma
about down crossings. The analog in the case of (IID) environments is well
known.

\begin{lemma}
\label{lem:DowncrossingBackwardBPSameDistribution} Assume that $\delta > 1$
and $X_0 = V_0 = 0$. For $n \in \N$ and $k \leq n$, let
\begin{align*}
D_{n,k} = |\{0 \leq m < \tau^X_n : X_m = k, X_{m+1} = k - 1\}|
\end{align*}
be the number of down crossings of the edge $(k,k-1)$ by the random walk
$(X_m)$ up to time $\tau_n^X$. Then $(V_0, V_1, \ldots ,V_n)$ and $(D_{n,n},
D_{n,n-1}, \ldots, D_{n,0})$ have the same distribution if the environment
$\omega$ is chosen according to the stationary measure $\PB_{\pi}$.
\end{lemma}

\begin{proof}
Since we assume that $\delta > 1$, it follows from Theorem
\ref{thm:TransienceRecurrence} that $\tau^X_n$ is $P_{0,\pi}$ a.s.
finite\footnotemark{} for each $n$. Fix $n \in \N$ and any realization of the
random variables $(\xi_i^k)_{k \in \Z, i \in \N}$ such that $\tau^X_n$ is
finite. Then, define $\widetilde{\xi}_i^k = \xi_i^{k+n}$, for $k \in \Z$ and
$i \in \N$. Let $D_{n,k}$, $0 \leq k \leq n$, be as in the statement of the
lemma when the random walk $(X_m)$ is generated according to the specific
fixed values of $(\xi_i^k)_{k \in \Z, i \in \N}$, and let $(\Vt_k)_{k \geq
0}$ be defined as in \eqref{eq:DefBackwardBP} with $v_0 = 0$, but
$\xi_i^{-(k+1)}$ replaced with $\widetilde{\xi}_i^{-(k+1)}$. We claim that,
in this case, $D_{n,n-k} = \Vt_k$, for each $0 \leq k \leq n$.
\footnotetext{Theorem \ref{thm:TransienceRecurrence} will not be proved till
later in Section \ref{sec:ProofTheoremTransienceRecurrence}, but the proof
uses only the forward branching process described above, and is thus
independent of the development in this section. The theorem is stated when
$\omega$ is chosen according to $\PB \equiv \PB_{\phi}$, rather than
$\PB_{\pi}$, but $\phi$ is allowed to be any initial distribution on $\SM$ in
the theorem. So, in particular, the conclusion is valid when $\phi = \pi$.}

The proof is by induction on $k$. For $k = 0$ we have $\Vt_0 = 0$, by
assumption, and $D_{n,n} = 0$, since the walk cannot down cross the edge
$(n,n-1)$ before first hitting site $n$. Now assume $D_{n, n - k} = \Vt_k =
\ell$ for some $0 \leq k < n$ and $\ell \geq 0$. Then the walk $(X_m)$ must
jump left from site $n-k$ exactly $\ell$ times prior to time $\tau_n^X$.
Thus, the walk must jump right from site $n - (k+1)$ exactly $\ell+1$ times
prior to $\tau_n^X$. Thus, the number of left jumps from site $n - (k+1)$
prior to time $\tau_n^X$ is exactly the number of left jumps from site $n -
(k+1)$ before the $(\ell+1)-th$ right jump. So, we have:
\begin{align*}
D_{n, n - (k+1)}
& = \mbox{ \# left jumps from site $n - (k+1)$ prior to time $\tau_n^X$ } \\
& = \mbox{ \# left jumps from site $n - (k+1)$ before $(\ell+1)$-th right jump } \\
& = \mbox{ \# failures in $(\xi_i^{n-(k+1)})_{i \in \N}$ before $(\ell+1)$-th success } \\
& = \mbox{ \# failures in $(\widetilde{\xi}_i^{-(k+1)})_{i \in \N}$ before $(\ell+1)$-th success } \\
& = \mbox{ \# failures in $(\widetilde{\xi}_i^{-(k+1)})_{i \in \N}$ before $(\Vt_k+1)$-th success } \\
& = \Vt_{k+1}.
\end{align*}

This completes the proof that $D_{n,n-k} = \Vt_k$, for each $0 \leq k \leq
n$, using the specific fixed values of $(\xi_i^k)_{k \in \Z, i \in \N}$ and
associated values of $\widetilde{\xi}_i^k = \xi_i^{k+n}$. The lemma now
follows since $(S_k)_{k \in \Z}$ is stationary under $\PB_{\pi}$, so the
stochastic process $(\xi_i^k)_{k \in \Z, i \in \N}$ has the same distribution
as the process $(\xi_i^{k+n})_{k \in \Z, i \in \N}$.
\end{proof}

The importance of the down crossings is their relation to hitting times for
the random walk $(X_n)$. If $X_0 = 0$, then for each $n \in \N$
\begin{align}
\label{eq:RelationHittingTimesAndDownCrossings}
\tau_n^X = n + 2 \sum_{k \leq n} D_{n,k} = n + 2 \sum_{k = 0}^n D_{n,k} + 2 \sum_{k < 0} D_{n,k}.
\end{align}
If $\delta > 1$, so that the random walk is right transient, then $\lim_{n
\to \infty} \sum_{k < 0} D_{n,k}$ is a.s. finite, and thus the asymptotic
distribution of $\tau_n^X$ is determined by the asymptotic distribution of
$\sum_{k = 0}^n D_{n,k}$. By Lemma
\ref{lem:DowncrossingBackwardBPSameDistribution}, the latter sum has the same
distribution as $\sum_{k=0}^n V_k$ (assuming that $\phi = \pi$). The general
proof strategy for Theorems  \ref{thm:Ballisticity} and \ref{thm:LimitLaws}
is to analyze asymptotic properties of $\sum_{k=0}^n V_k$, relate these to
asymptotic properties of the hitting times $\tau_n^X$, and then relate those
to asymptotic properties of the random walk $(X_n)$ itself. This basic
approach has been employed many times before in the study of excited random
walks, e.g. \cite{Kosygina2008}, \cite{Basdevant2008b}, \cite{Kosygina2011},
\cite{Kosygina2017}. (See also
 \cite{Kesten1975}, \cite{Toth1995}, \cite{Toth1996}, and \cite{Pinsky2010} for similar uses of branching processes
in analyzing one dimensional self-interacting random walks and random walk in
random environment.)

In the sequel we will use the following notation for the backward branching
process $V = (V_k)_{k \geq 0}$, similar to that for the forward branching
process $(U_k)_{k \geq 0}$. $P_{v_0}^{V, \omega}$ is the probability measure
for $(V_k)_{k \geq 0}$ started from level $v_0$ in a fixed environment
$\omega$, and, for $s \in \SM$, $P_{v_0,s}^V$ is the probability measure for
the joint process $(V_k,R_k)_{k \geq 0}$ when $V_0 = v_0$ and $R_0 = s$. The
measures $P_{v_0, \pi}^{V}$ and $P_{v_0}^{V}$ are defined by
\begin{align*}
P_{v_0,\pi}^{V}(\cdot) \equiv \sum_{s \in \SM} \pi(s) P_{v_0,s}^V(\cdot) ~~~~\mbox{ and }~~~~ P_{v_0}^{V}(\cdot) \equiv \sum_{s \in \SM} \phi(s) P_{v_0,s}^V(\cdot).
\end{align*}
Under any of these measures $P_{v_0,s}^V$, $P_{v_0,\pi}^{V}$, and
$P_{v_0}^{V}$ the joint process $(V_k, R_k)_{k \geq 0}$ is a time-homogeneous
Markov chain with transition probabilities $p^V_{(v,r)(v',r')} \equiv
\mbox{Prob}(V_{k+1} = v', R_{k+1} = r'|V_k = v, R_k = r)$ given by
\begin{align}
\label{eq:VkRkMarkovChainTransitionProbs}
p^V_{(v,r)(v',r')} = \widetilde{\KM}(r,r')P^{V,\omega_{r'}}_v(V_1 = v'),
\end{align}
where $\widetilde{\KM}$ is the transition matrix for the Markov chain $(R_k)$
given by $\widetilde{\KM}(r,r') = \KM(r',r) \cdot \frac{\pi(r')}{\pi(r)}$ and
$\omega_{r'}$ is the deterministic environment with stack $r'$ at each site:
$\omega_{r'}(x,i) = r'(i)$, for all $x \in \Z$ and $i \in \N$.

\subsection{Related processes}
\label{subsec:RelatedProcesses}

The branching processes $(U_k)$ and $(V_k)$ are difficult to analyze directly
because their transition probabilities depend on the underlying environment
$\omega$, and therefore these processes are not Markovian when $\omega$ is
chosen randomly according to $\PB$ (or $\PB_{\pi}$ or $\PB_s$), and are not
time-homogeneous in a fixed environment $\omega$. In this section we
introduce some simpler related processes, which are both Markovian and
time-homogeneous and, thus, easier to analyze.

\subsubsection{The processes $(\Uh_k)_{k \geq 0}$ and $(\Vh_k)_{k \geq 0}$}
\label{subsubsec:UhkVhk}

Throughout this section and the remainder of the paper $s \in \SM$ is an
arbitrary but fixed stack. We define stopping times $(\tau_k)_{k \geq 0}$ and
$(\tau'_k)_{k \geq 0}$ by
\begin{align}
\label{eq:DefTaukTaukprime}
& \tau_0 = \inf\{j \geq 0: R_j = s\} \mbox{ and } \tau_{k+1} = \inf\{j > \tau_k: R_j = s\} ~,~k \geq 0;\\
& \tau'_0 = \inf\{j \geq 0: S_j = s\} \mbox{ and } \tau'_{k+1} = \inf\{j > \tau'_k: S_j = s\} ~,~k \geq 0.
\end{align}
Then we define processes $(\Vh_k)_{k \geq 0}$ and $(\Uh_k)_{k \geq 0}$ by
\begin{align}
\label{eq:DefUhkVhk}
\Vh_k = V_{\tau_k} ~~\mbox{ and }~~ \Uh_k = U_{\tau'_k}.
\end{align}
In other words, the processes $(\Vh_k)$ and $(\Uh_k)$ are constructed from
$(V_k)$ and $(U_k)$ by observing the latter only at times $j$ when $R_j = s$
or $S_j = s$, respectively.

Since the process $(U_k,S_k)_{k \geq 0}$ is a time-homogeneous Markov chain
(under $P_x^U$, $P_{x,\pi}^U$, and $P_{x,r}^U$, $r \in \SM$) and the process
$(V_k,R_k)_{k \geq 0}$ is a time-homogeneous Markov chain (under $P_x^V$,
$P_{x,\pi}^V$, and $P_{x,r}^V$), the processes $(\Uh_k)_{k \geq 0}$ and
$(\Vh_k)_{k \geq 0}$ are also time homogeneous Markov chains (under these
same measures) with transition probabilities\footnotemark{}
\footnotetext{Note that we extend here the probability measures $P_x^U$,
$P_{x,\pi}^U$, $P_{x,r}^U$ for the joint process $(U_k,S_k)_{k \geq 0}$ to
the process $(\Uh_k)_{k \geq 0}$ derived from it. The initial value $x$ is
still for $U_0$. This is not, in general, the same as the initial value
$\Uh_0$, except under $P_{x,s}^U$ where $S_0 = s$ deterministically. Similar
remarks apply to the process $(\Vh_k)$ with respect to the probability
measures $P_x^V$, $P_{x,\pi}^V$, $P_{x,r}^V$.}
\begin{align}
\label{eq:TransitionProbsUhk}
&P_{x,r}^U(\Uh_{k+1} = y|\Uh_k = z) = P_{x,\pi}^U(\Uh_{k+1} = y|\Uh_k = z) = P_{x}^U(\Uh_{k+1} = y|\Uh_k = z) = P_{z,s}^U(U_{\tau_s^S} = y), \\
\label{eq:TransitionProbsVhk}
&P_{x,r}^V(\Vh_{k+1} = y|\Vh_k = z) = P_{x,\pi}^V(\Vh_{k+1} = y|\Vh_k = z) = P_{x}^V(\Vh_{k+1} = y|\Vh_k = z) = P_{z,s}^V(V_{\tau_s^R} = y)
\end{align}
for $x, y, z, k \in \N_0$, where (in accordance with our conventions in
Section \ref{subsec:Notation})
\begin{align*}
\tau_s^S \equiv \inf\{j > 0: S_j = s\} ~~\mbox{ and } ~~\tau_s^R \equiv \inf\{j > 0: R_j = s\}.
\end{align*}
In words, the probability of transitioning from $z$ to $y$ for the Markov
chain $(\Uh_k)$ is the probability the process $(U_k)$ transitions from level
$z$ to level $y$ during the time period that the process $(S_k)$ makes one
excursion from state $s$. Similarly, the probability of transitioning from
$z$ to $y$ for the Markov chain $(\Vh_k)$ is the probability the process
$(V_k)$ transitions from level $z$ to level $y$ during the time period that
the process $(R_k)$ makes one excursion from state $s$. Although the
transition probabilities for these Markov chains are complicated because they
depend on the random return times $\tau_s^S$ and $\tau_s^R$, we will see in
Section \ref{subsec:ExpectationVarianceConcentrationEstimates} that they can
be analyzed reasonably well. By contrast, trying to analyze the processes
$(U_k)$ and $(V_k)$ directly, under any of the above averaged measures,
appears difficult, because they are not Markovian.

\subsubsection{The dominating processes $(U_k^{\pm})_{k \geq 0}$ and $(V_k^{\pm})_{k \geq 0}$}
\label{subsubsec:DominatingProcessesUkpmVkpm}

To analyze the transition probabilities for the processes $(\Uh_k)$ and
$(\Vh_k)$ it will be helpful to introduce some additional auxiliary
processes, which dominate the processes $(U_k)$ and $(V_k)$, from both above
and below. Recall that the forward branching process $(U_k)_{k \geq 0}$ and
the backward branching process $(V_k)_{k \geq 0}$, started from level $x$,
are defined in terms of the random variables $(\xi_i^k)_{k \in \Z, i \in \N}$
by
\begin{align*}
& U_0 = x ~\mbox{ and }~ U_{k+1} = \inf \Big\{m : \sum_{i = 1}^m \indicator{\{\xi_i^{k+1} = 0\}} = U_k \Big\} - U_k, \\
& V_0 = x ~\mbox{ and }~ V_{k+1} = \inf \Big\{ m : \sum_{i=1}^m \indicator\{\xi_i^{-(k+1)} = 1\} = V_k + 1 \Big\} - (V_k + 1).
\end{align*}
That is, $U_{k+1}$ is the number of successes in the sequence
$(\xi_i^{k+1})_{i \in \N}$ prior to the $U_k$-th failure, and $V_{k+1}$ is
the number of failures in the sequence $(\xi_i^{-(k+1)})_{i \in \N}$ prior to
the $(V_k + 1)$-th success. Let us define modified processes  $(U_k^+)_{k
\geq 0}$, $(V_k^+)_{k \geq 0}$, $(U_k^-)_{k \geq 0}$, and  $(V_k^-)_{k \geq
0}$, all started from level $x$ as follows\footnotemark{}: \footnotetext{Note
that, by our conventions for empty sums, the infimum in the last two
equations is defined to be $M$ if $U_k^- \leq M$ or $V_k^- \leq M$. Thus, in
these cases the whole right hand side is 0.}
\begin{align*}
& U_0^+ \,{=}\, x ~\mbox{ and }~ U^+_{k+1} \,{=}\, \inf \Big\{m \,{\geq}\, M: \sum_{i = M+1}^m \indicator{\{\xi_i^{k+1} = 0\}} = U^+_k+1 \Big\} - (U^+_k+1), \\
& V_0^+ \,{=}\, x ~\mbox{ and }~ V^+_{k+1} \,{=}\, \inf \Big\{m \,{\geq}\, M: \sum_{i=M+1}^m \indicator\{\xi_i^{-(k+1)} = 1\} = V^+_k + 1 \Big\} - (V^+_k + 1), \\
& U_0^- \,{=}\, x ~\mbox{ and }~ U^-_{k+1} \,{=}\, \inf \Big\{m \,{\geq}\, M : \sum_{i = M+1}^m \indicator{\{\xi_i^{k+1} = 0\}} = (U_k^- - M)^+ \Big\} - (U_k^- - M)^+ - M, \\
& V_0^- \,{=}\, x ~\mbox{ and }~ V^-_{k+1} \,{=}\, \inf \Big\{m \,{\geq}\, M : \sum_{i = M+1}^m \indicator{\{\xi_i^{-(k+1)} \,{=}\, 1\}} \,{=}\, (V_k^- \,{-}\, M)^+ \Big\} \,{-}\, (V_k^- \,{-}\, M)^+ \,{-}\, M.
\end{align*}
In words, $U_{k+1}^+$ is the number of successes in the sequence
$(\xi_i^{k+1})_{i \in \N}$ before the $(U_k^+ + 1)$-th failure, when we
condition that $\xi_1^{k+1}, \ldots, \xi_M^{k+1}$ are all successes, and
$U_{k+1}^-$ is the number of successes in the sequence $(\xi_i^{k+1})_{i \in
\N}$ before the $U_k^-$-th failure, when we condition that $\xi_1^{k+1},
\ldots, \xi_M^{k+1}$ are all failures. The interpretations for $V_{k+1}^+$
and $V_{k+1}^-$ are the same with ``success'' replaced by ``failure'' and
``$~\xi_i^{k+1}~$'' replaced by ``$~\xi_i^{-(k+1)}~$''.

By construction we have
\begin{align}
\label{eq:MonotonicityUkVkProcesses}
U_k^- \leq U_k \leq U_k^+ ~,~ \mbox{for all $k$ }~~~&\mbox{ and }~~~V_k^- \leq V_k \leq V_k^+ ~,~\mbox{for all $k$ }
\end{align}
if all processes are started from the same level $x$.  Also, since
$(\xi_i^k)_{k \in \Z, i > M}$ are i.i.d. Ber(1/2) random variables, for any
values of the cookie stacks $(S_k)_{k \in \Z}$, each of the processes
$(U_k^+)$, $(U_k^-)$, $(V_k^+)$, $(V_k^-)$ is a time-homogeneous Markov chain
and
\begin{align}
\label{eq:UkVkIndepedentOfSk}
(U_k^-, U_k^+)_{k \geq 0} \perp (S_k)_{k \in \Z} ~~\mbox{ and }~~ (V_k^-, V_k^+)_{k \geq 0} \perp (S_k)_{k \in \Z}.
\end{align}
These statements hold for any initial values $U_0^-, U_0^+, V_0^-, V_0^+$ and
any marginal distribution $\rho$ on $S_0$ (including $\phi$, $\pi$, or a
point mass at $r \in \SM$). For the same reason (i.e. that $(\xi_i^k)_{k \in
\Z, i > M}$ are i.i.d. Ber(1/2)), we also have
\begin{align}
\label{eq:UkPlusMinusVkPlusMinusEqualInLaw}
(U_k^-, U_k^+)_{k \geq 0} \stackrel{law}{=} (V_k^-, V_k^+)_{k \geq 0}
\end{align}
when all processes are started from the same level $x$ (again for any
marginal distribution on $S_0$).

To analyze the processes $(U_k^{\pm})$ and $(V_k^{\pm})$ it will be helpful
to represent them in a form similar to
\eqref{eq:DefForwardBPInTermsOfGeometrics} and
\eqref{eq:DefBackwardBPInTermsOfGeometrics} for the processes $(U_k)$ and
$(V_k)$. Define $\GM_i^k$ to be the number of successes in the sequence
$(\xi_j^k)_{j > M}$ between the $(i-1)$-th and $i$-th failures, and define
$\HM_i^k$ to be the number of failures in the sequence $(\xi_j^{-k})_{j > M}$
between the $(i-1)$-th and $i$-th successes. Then
\begin{align}
\label{eq:UkPlusMinusGeometricRepresentation}
& U_{k+1}^+ = M + \sum_{i = 1}^{U_k^+ + 1} \GM_i^{k+1}   ~~~~\mbox{ and }~~~~ U_{k+1}^- = \sum_{i = 1}^{(U_k^- -M)^+} \GM_i^{k+1}
\end{align}
where $(\GM_i^k)_{k \in \Z, i \in \N}$ are i.i.d. Geo(1/2) random variables,
and
\begin{align}
\label{eq:VkPlusMinusGeometricRepresentation}
& V_{k+1}^+ = M + \sum_{i = 1}^{V_k^+ + 1} \HM_i^{k+1}   ~~~~\mbox{ and }~~~~ V_{k+1}^- = \sum_{i = 1}^{(V_k^- -M)^+} \HM_i^{k+1}
\end{align}
where $(\HM_i^k)_{k \in \Z, i \in \N}$ are i.i.d. Geo(1/2) random variables.

\subsection{Expectation, variance, and concentration estimates}
\label{subsec:ExpectationVarianceConcentrationEstimates}

In this section we use the dominating processes $(U_k^{\pm})$ and
$(V_k^{\pm})$ to analyze the transition probabilities
\eqref{eq:TransitionProbsUhk} and \eqref{eq:TransitionProbsVhk} for the
processes $(\Uh_k)$ and $(\Vh_k)$. We also prove, slightly more generally,
concentration estimates for the processes $(U_k)$ and $(V_k)$ up to the
random stopping times $\tau_s^S$ and $\tau_s^R$, when $(S_k)$ and $(R_k)$,
respectively, are started from an arbitrary initial state $r \in \SM$, rather
than $s$. Finally, we prove a type of ``overshoot lemma'' for the processes
$(\Uh_k)$ and $(\Vh_k)$ analogous to Lemma 5.1 of \cite{Kosygina2011}.

Throughout it is assumed, when not otherwise specified, that the processes
$(U_k)$, $(U_k^+)$, $(U_k^-)$ are all started from the same level $x$, and
the processes $(V_k)$, $(V_k^+)$, $(V_k^-)$ are all started from the same
level $x$. The probability measure $P_{x,r}^{U}$ will be used for all these
``$U$-processes'' started from level $x$ when $S_0 = r$, and the probability
measure $P_{x,r}^{V}$ will be used for all these ``$V$-processes'' started
from level $x$ when $R_0 = r$. The following general fact will be needed in
our analysis of the ``$U$-processes'' and ``$V$-processes'' below, as well as
in several other parts of the paper.
\begin{lemma}
\label{lem:GeneralLargeDeviationBoundSumsIID} Let $Z$ be a random variable
with mean $\mu$ and exponential tails, and let $Z_1, Z_2,\ldots$ be i.i.d.
random variables distributed as $Z$. Then for any $\epsilon_0 \in (0,
\infty)$ there exist constants $C_1(\epsilon_0), C_2(\epsilon_0) > 0$ such
that the empirical means $\Zb_n \equiv \frac{1}{n} \sum_{i=1}^n Z_i$ satisfy:
\begin{align}
\label{eq:iidSmallEpsBound}
\P( |\Zb_n - \mu| \geq \epsilon ) & \leq C_1 \exp(-C_2 \epsilon^2 n) ~,~ \mbox{for all } 0 < \epsilon \leq \epsilon_0 \mbox{ and } n \in \N. \\
\label{eq:iidLargeEpsBound}
\P( |\Zb_n - \mu| \geq \epsilon ) & \leq C_1 \exp(-C_2 \epsilon n) ~,~ \mbox{for all } \epsilon \geq \epsilon_0 \mbox{ and } n \in \N.
\end{align}
\end{lemma}

\begin{proof}
The exponential tails condition on the random variable $Z$ implies there
exist some positive constants $b,c$ such that for all $\lambda \in [-b,b]$,
$\E( e^{\lambda (Z-\mu)}) \leq e^{c \lambda^2}$ and $\E( e^{\lambda (\mu-Z)})
\leq e^{c \lambda^2}$. Thus, the lemma is a consequence of \cite[Theorem
III.15]{Petrov1975}.
\end{proof}

Using Lemma \ref{lem:GeneralLargeDeviationBoundSumsIID} along with
\eqref{eq:UkPlusMinusGeometricRepresentation} and
\eqref{eq:VkPlusMinusGeometricRepresentation} and a small bit of analysis one
may obtain the following concentration estimates for the differences
$(U_k^{\pm} - U_{k-1}^{\pm})$ and $(V_k^{\pm} - V_{k-1}^{\pm})$.
\begin{lemma}
\label{lem:ConcentrationEstimateUkVkPlusMinusSingleStepDifferences} For each
$\epsilon_0 \in (0,\infty)$, there exist constants $c_1(\epsilon_0)$,
$c_2(\epsilon_0) > 0$ such that the following hold for all $r \in \SM$, $x,y
\in \N_0$, and $k \in \N:$
\begin{align}
\label{eq:UkMinusSmallEps}
& P_{x,r}^U\left(|U_k^- - U_{k-1}^-| \geq \epsilon y \Big| U_{k-1}^- = y \right) \leq c_1 e^{-c_2 \epsilon^2 y} ~,~ \mbox{ for } 0 < \epsilon \leq \epsilon_0. \\
\label{eq:UkMinusBigEps}
& P_{x,r}^U\left(|U_k^- - U_{k-1}^-| \geq \epsilon y \Big| U_{k-1}^- = y \right) \leq c_1 e^{-c_2 \epsilon y} ~,~~ \mbox{ for } \epsilon \geq \epsilon_0. \\
\label{eq:UkPlusSmallEps}
& P_{x,r}^U \left(|U_k^+ - U_{k-1}^+| \geq \epsilon y \Big| U_{k-1}^+ = y \right) \leq c_1 e^{-c_2 \epsilon^2 y} ~,~ \mbox{ for } 0 < \epsilon \leq \epsilon_0. \\
\label{eq:UkPlusBigEps}
& P_{x,r}^U\left(|U_k^+ - U_{k-1}^+| \geq \epsilon y \Big| U_{k-1}^+ = y \right) \leq c_1 e^{-c_2 \epsilon y} ~,~~ \mbox{ for } \epsilon \geq \epsilon_0.
\end{align}
Moreover, the equivalent statements also hold for $(V_k^+)$ and $(V_k^-)$
with the same constants $c_1, c_2$.
\end{lemma}

\begin{remark}
Note that since $(U_k^+)$ and $(U_k^-)$ are each time-homogeneous Markov
chains independent of $(S_k)_{k \in \Z}$ the probabilities on the left hand
side of these equations do not depend on $x$, $r$, or $k$.  Similar
statements also apply for the processes $(V_k^+)$ and $(V_k^-)$.
\end{remark}

We wish now to extend these concentration estimates for the single time step
differences in the processes $(U_k^{\pm})$ and $(V_k^{\pm})$ to concentration
estimates for these processes up to the random stopping times $\tau_s^S$ and
$\tau_s^R$. This is where we will need the uniform ergodicity hypothesis on
the Markov chains $(S_k)$ and $(R_k)$. Due to Lemma
\ref{lem:ExponentialTailHittingTimesUniformlyErgodicMCs} below and uniform
ergodicity of $(S_k)$ and $(R_k)$ there exist some constants $c_3, c_4 > 0$
such that:
\begin{align}
\label{eq:tausStausRExponentialTail}
& \PB_r(\tau_s^S \geq t) \leq c_3 e^{-c_4 t} ~,~ \mbox{ for all } r \in \SM \mbox{ and } t \in [0,\infty). \nonumber \\
& \PB_r(\tau_s^R \geq t) \leq c_3 e^{-c_4 t} ~,~ \mbox{ for all } r \in \SM \mbox{ and } t \in [0,\infty).
\end{align}
Here, as described in Section \ref{subsec:Notation}, $\PB_r$ is the
probability measure for the process $(S_k)$ itself (equivalently the process
$(R_k)$) when $S_0 = R_0 = r$.

\begin{lemma}
\label{lem:ExponentialTailHittingTimesUniformlyErgodicMCs} Let $(Z_k)$ be a
uniformly ergodic Markov chain on a countable state space $\ZM$, and let
$\P_x(\cdot)$ be the probability measure for the Markov chain $(Z_k)$ started
from $Z_0 = x$. Then, for each $z \in \ZM$, there exist constants $C > 0$ and
$0 < \alpha < 1$ such that
\begin{align*}
\P_x(\tau^Z_z > n) \leq C \alpha^n ~,~\mbox{ for all } x \in \ZM, n \in \N_0.
\end{align*}
\end{lemma}

\begin{proof}
Fix $z \in \ZM$. Let $\rho = (\rho(x))_{x \in \ZM}$ denote the stationary
distribution of the Markov chain $(Z_k)$, and let $\MM = \{\MM(x,y)\}_{x,y
\in \ZM}$ be its transition matrix. Also, let $\epsilon = \rho(z)$. By
uniform ergodicity of $(Z_k)$, there is some $\ell \in \N$ such that $\norm{
\MM^{\ell}(y, \cdot) - \rho(\cdot)}_{TV} \leq \epsilon/2$, for all $y \in
\ZM$. This implies $\P_y(Z_{\ell} = z) = \MM^{\ell}(y,z) \geq \epsilon/2$,
for all $y \in \ZM$. Thus, starting from any initial state $x$ we have
\begin{align*}
\P_x(\tau^Z_z > n \cdot \ell) = \prod_{m = 0}^{n-1} \P_x(\tau^Z_z > (m+1)\ell | \tau^Z_z > m \ell) \leq (1 - \epsilon/2)^n = \alpha^{\ell n},
\end{align*}
where $\alpha \equiv (1 - \epsilon/2)^{1/\ell} \in (0, 1)$. It follows that
$\P_x(\tau^Z_z > n) \leq C \alpha^n$, for all $n \geq 0$ and $x \in \ZM$,
with $C \equiv 1/\alpha^{\ell}$.
\end{proof}

\begin{lemma}
\label{lem:ConcentrationEstimateForUkVkPlusMinus} There exist constants $c_5,
c_6 > 0$ such that the following hold for each $r \in \SM:$
\begin{align}
\label{eq:UkMinusSmallEpsBound}
& P_{x,r}^U\left(\max_{0 \leq k \leq \tau_s^S} |U_k^- - x| \geq \epsilon x \right) \leq c_5 (1 + \epsilon^{2/3} x^{1/3}) e^{-c_6 \epsilon^{2/3} x^{1/3}} ~,~\mbox{for all $x \in \N_0$ and $0 < \epsilon \leq 1$.} \\
\label{eq:UkMinusLargeEpsBound}
& P_{x,r}^U\left(\max_{0 \leq k \leq \tau_s^S} |U_k^- - x| \geq \epsilon x \right) \leq c_5 (1 + \epsilon^{1/3} x^{1/3}) e^{-c_6 \epsilon^{1/3} x^{1/3}} ~,~\mbox{for all $x \in \N_0$ and $\epsilon \geq 1$.} \\
\label{eq:UkPlusSmallEpsBound}
& P_{x,r}^U\left(\max_{0 \leq k \leq \tau_s^S} |U_k^+ - x| \geq \epsilon x\right) \leq c_5 (1 + \epsilon^{2/3} x^{1/3}) e^{-c_6 \epsilon^{2/3} x^{1/3}} ~,~\mbox{for all $x \in \N_0$ and $0 < \epsilon \leq 1$.} \\
\label{eq:UkPlusLargeEpsBound}
& P_{x,r}^U\left(\max_{0 \leq k \leq \tau_s^S} |U_k^+ - x| \geq \epsilon x \right) \leq c_5 (1 + \epsilon^{1/3} x^{1/3}) e^{-c_6 \epsilon^{1/3} x^{1/3}} ~,~\mbox{for all $x \in \N_0$ and $\epsilon \geq 1$.}
\end{align}
Moreover, the equivalent statements (with $\tau_s^S$ replaced by $\tau_s^R$)
also hold for the processes $(V_k^+)$ and $(V_k^-)$ with the same constants
$c_5, c_6$.
\end{lemma}

All statements are trivially true if $x = 0$ (taking any $c_5 \geq 1$ and any
$c_6 > 0$), so we will assume $x \geq 1$. We will prove
\eqref{eq:UkMinusSmallEpsBound} and \eqref{eq:UkMinusLargeEpsBound}.  The
proof of \eqref{eq:UkPlusSmallEpsBound} is identical to that of
\eqref{eq:UkMinusSmallEpsBound} with $U_k^-$ replaced by $U_k^+$ line by
line, and the proof of \eqref{eq:UkPlusLargeEpsBound} is almost identical to
that of \eqref{eq:UkMinusLargeEpsBound} with $U_k^-$ replaced by $U_k^+$ line
by line\footnotemark{}. The equivalent statements to
\eqref{eq:UkMinusSmallEpsBound}-\eqref{eq:UkPlusLargeEpsBound} for the
processes $(V_k^-)$ and $(V_k^+)$ are also proved exactly the same way; all
the proofs use is Lemma
\ref{lem:ConcentrationEstimateUkVkPlusMinusSingleStepDifferences} and
\eqref{eq:tausStausRExponentialTail}, and these estimates are the same for
$(V_k^{\pm})$ and $(U_k^{\pm})$ and for $\tau_s^R$ and $\tau_s^S$. The proofs
of \eqref{eq:UkMinusSmallEpsBound} and \eqref{eq:UkMinusLargeEpsBound} will
be given separately and the constants $c_5$, $c_6$ obtained in the two cases
will not be the same.  To find a single $c_5$ and $c_6$ that hold in both
cases simply take $c_5$ to be the maximum of the $c_5$'s from the two proofs
and $c_6$ to be the minimum of the $c_6$'s from the two proofs.

\footnotetext{In the derivation of \eqref{eq:IfUkMinus1Small} the case $z =
0$ must be considered separately for both proofs. For the process $(U_k^-)$
we have $U_k^- = 0$ (deterministically) if $U_{k-1}^- = z = 0$, as noted in
the proof of \eqref{eq:UkMinusLargeEpsBound} below. For the process $(U_k^+)$
this is not the case. However, if we start with $U_0^+ = x \geq 1$, as we
assume, then it is actually impossible that $U_k^+$ is 0 for any $k \geq 0$
(indeed, $U_k^+ \geq M$ for all $k \geq 1$). So we do not have this problem
to deal with. }

\begin{proof}[Proof of Lemma \ref{lem:ConcentrationEstimateForUkVkPlusMinus}, Equation \eqref{eq:UkMinusSmallEpsBound}, with $x \geq 1$.]
Fix $\epsilon \leq 1$ and denote $\MM =\break \max_{0 \leq k \leq \tau_s^S}
\big|U_k^- - x\big|$. Then
\begin{align}
\label{eq:MGreaterEpsilonxSplittingSmallEps}
P_{x,r}^U(\MM \geq \epsilon x)
& = P_{x,r}^U\Big(\tau_s^S \leq \frac{1}{2} \epsilon^{2/3} x^{1/3}\Big) P_{x,r}^U\Big(\MM \geq \epsilon x \Big| \tau_s^S \leq \frac{1}{2} \epsilon^{2/3} x^{1/3}\Big) \nonumber \\
& ~~~~+ P_{x,r}^U\Big(\tau_s^S > \frac{1}{2} \epsilon^{2/3} x^{1/3}\Big) P_{x,r}^U\Big(\MM \geq \epsilon x \Big| \tau_s^S > \frac{1}{2} \epsilon^{2/3} x^{1/3}\Big) \nonumber \\
& \leq \max_{0 \leq n \leq \frac{1}{2} \epsilon^{2/3} x^{1/3}} P_{x,r}^U\Big(\MM \geq \epsilon x \Big| \tau_s^S = n \Big) ~+~ P_{x,r}^U\Big(\tau_s^S > \frac{1}{2} \epsilon^{2/3} x^{1/3}\Big).
\end{align}
By \eqref{eq:tausStausRExponentialTail},
\begin{align}
\label{eq:TausSNotTooBigSmallEps}
P_{x,r}^U\Big(\tau_s^S > \frac{1}{2} \epsilon^{2/3} x^{1/3}\Big) = \PB_r\Big(\tau_s^S > \frac{1}{2} \epsilon^{2/3} x^{1/3}\Big) \leq c_3 e^{-c_4 \cdot \frac{1}{2} \epsilon^{2/3} x^{1/3} }.
\end{align}
So, we need only bound $P_{x,r}^U(\MM \geq \epsilon x| \tau_s^S = n)$, for $n
\leq \frac{1}{2} \epsilon^{2/3} x^{1/3}$. Fix such an $n$ and define events
$A_k$, $k \geq 1$, by $A_k = \left\{\left|U_k^- - U_{k-1}^-\right| \leq
\epsilon^{1/3} x^{2/3} \right\}$. Then, for $0 \leq k \leq n$, we have
\begin{align}
\label{eq:MaxUkPlusBoundSmallEps}
U_k^- \leq x + n \cdot \epsilon^{1/3} x^{2/3} \leq x + x/2 < 2x  ~,~ \mbox{ on the event } A_1 \cap \ldots \cap A_k
\end{align}
and
\begin{align}
\label{eq:MinUkPlusBoundSmallEps}
U_k^- \geq x - n \cdot \epsilon^{1/3} x^{2/3} \geq x - x/2 = x/2  ~,~ \mbox{ on the event } A_1 \cap \ldots \cap A_k.
\end{align}
Also, by Lemma
\ref{lem:ConcentrationEstimateUkVkPlusMinusSingleStepDifferences}, there
exist some constants $c_1, c_2 > 0$ such that
\begin{align}
\label{eq:UkPlusOneStepDiffBoundSmallEps}
P_{x,r}^U\Big(|U_k^- - U_{k-1}^-| > \epsilont z \Big| U_{k-1}^- = z\Big)
\leq c_1 \exp(-c_2 \cdot \epsilont^2 \cdot z) ~,~\mbox{ for each } k,z \mbox{ and } 0 < \epsilont \leq 2.
\end{align}
Since $\epsilon \leq 1$ and $x \geq 1$, $\epsilon^{1/3}x^{2/3}/z \leq 2$ for
all $z \geq x/2$. Thus, for each $x/2 \leq z \leq 2x$, we have
\begin{align}
\label{eq:UkMinus1Conditional}
& P_{x,r}^U(A_k^c | U_{k-1}^- = z)
=  P_{x,r}^U\Big( |U_k^- - U_{k-1}^-| > \Big(\frac{\epsilon^{1/3} x^{2/3}}{z} \Big) \cdot z  \Big| U_{k-1}^- = z \Big) \nonumber \\
& \leq c_1 \exp\Big( -c_2 \frac{(\epsilon^{1/3} x^{2/3})^2}{2x} \Big) = c_1 \exp\Big(-\frac{c_2}{2} \epsilon^{2/3} x^{1/3} \Big).
\end{align}
Combining \eqref{eq:MaxUkPlusBoundSmallEps},
\eqref{eq:MinUkPlusBoundSmallEps}, and \eqref{eq:UkMinus1Conditional} and
using the fact that $(U_k^-)$ is a Markov chain shows
\begin{align*}
P_{x,r}^U(A_k^c |A_1, \ldots, A_{k-1}) \leq c_1 \exp\left(-\frac{c_2}{2} \epsilon^{2/3} x^{1/3} \right) ~,~\mbox{ for } 1 \leq k \leq n.
\end{align*}
Hence,
\begin{align*}
P_{x,r}^U\left(\bigcup_{k=1}^n A_k^c \right)
\leq n \cdot c_1 \exp\left( -\frac{c_2}{2} \epsilon^{2/3} x^{1/3} \right)
\leq \frac{1}{2} \epsilon^{2/3} x^{1/3} \cdot c_1 \exp\left( -\frac{c_2}{2} \epsilon^{2/3} x^{1/3} \right).
\end{align*}
Now, when $\tau_s^S = n$, $\MM \leq n \cdot \epsilon^{1/3} x^{2/3} \leq
\frac{1}{2} \epsilon x$ on the event $\cap_{k=1}^n A_k$. So,
\begin{align}
\label{eq:MBoundTausSEqualnSmallEps}
P_{x,r}^U(\MM \geq \epsilon x|\tau_s^S=n) \leq \frac{1}{2} \epsilon^{2/3} x^{1/3} \cdot c_1 \exp\left( -\frac{c_2}{2} \epsilon^{2/3} x^{1/3} \right).
\end{align}
Since, \eqref{eq:MBoundTausSEqualnSmallEps} is valid for each $n \leq
\frac{1}{2} \epsilon^{2/3} x^{1/3}$ it follows from
\eqref{eq:MGreaterEpsilonxSplittingSmallEps} and
\eqref{eq:TausSNotTooBigSmallEps} that
\begin{align*}
P_{x,r}^U(\MM \geq \epsilon x) \leq c_5(1 + \epsilon^{2/3} x^{1/3}) \exp(-c_6 \epsilon^{2/3} x^{1/3})
\end{align*}
where $c_5 = \max\{\frac{c_1}{2},c_3\}$ and $c_6 = \min\{\frac{c_2}{2},
\frac{c_4}{2}\}$.
\end{proof}

\begin{remark}
If we define the deterministic time $t_{\epsilon,x} = \floor{\frac{1}{2}
\epsilon^{2/3}x^{1/3}}$, for $0 < \epsilon \leq 1$ and $x \in \N$, then the
same exact steps used in the derivation of
\eqref{eq:MBoundTausSEqualnSmallEps} for an arbitrary $n \leq t_{\epsilon,x}$
show that
\begin{align}
\label{eq:MBoundFixedSmallEnoughTime}
P_{x,r}^U\left(\max_{0 \leq k \leq t_{\epsilon,x}} |U_k^- - x| \geq \epsilon x \right) \leq c_1 t_{\epsilon,x} \exp \left( -c_2 t_{\epsilon,x} \right).
\end{align}
We isolate this observation, as it will be needed later in the proof of Lemma
\ref{lem:ExpectationUkVkStoppedAtTausS}.
\end{remark}

\begin{proof}[Proof of Lemma  \ref{lem:ConcentrationEstimateForUkVkPlusMinus}, Equation \eqref{eq:UkMinusLargeEpsBound}, with $x \geq 1$.]
As above, let $\MM = \max_{0 \leq k \leq \tau_s^S} \big|U_k^- - x\big|$. We
will assume that $\epsilon > 1$, as the case $\epsilon = 1$ follows from
\eqref{eq:UkMinusSmallEpsBound}. Then
\begin{align}
\label{eq:MGreaterEpsilonxSplittingLargeEps}
P_{x,r}^U(\MM \geq \epsilon x)
& = P_{x,r}^U\Big(\tau_s^S < \epsilon^{1/3} x^{1/3} \Big) P_{x,r}^U\Big(\MM \geq \epsilon x \Big| \tau_s^S < \epsilon^{1/3} x^{1/3} \Big) \nonumber \\
   & ~~~~~+ P_{x,r}^U\Big(\tau_s^S \geq \epsilon^{1/3} x^{1/3} \Big) P_{x,r}^U\Big(\MM \geq \epsilon x \Big| \tau_s^S \geq \epsilon^{1/3} x^{1/3} \Big) \nonumber \\
& \leq \max_{0 \leq n < \epsilon^{1/3} x^{1/3}} P_{x,r}^U\left(\MM \geq \epsilon x \Big| \tau_s^S = n \right) ~+~ P_{x,r}^U\left(\tau_s^S \geq \epsilon^{1/3} x^{1/3} \right).
\end{align}
By \eqref{eq:tausStausRExponentialTail},
\begin{align}
\label{eq:TausSNotTooBigLargeEps}
P_{x,r}^U(\tau_s^S \geq \epsilon^{1/3} x^{1/3}) = \PB_r(\tau_s^S \geq \epsilon^{1/3} x^{1/3}) \leq c_3 e^{-c_4 \epsilon^{1/3} x^{1/3} }.
\end{align}
So, we need only bound $P_{x,r}^U(\MM \geq \epsilon x| \tau_s^S = n)$, for $n
< \epsilon^{1/3} x^{1/3}$. Fix such an $n$ and define events $A_k$, $k \geq
1$, by $A_k = \left\{ \left|U_k^- - U_{k-1}^- \right| \leq \epsilon^{2/3}
x^{2/3} \right\}$. Then, for $0 \leq k \leq n$, we have
\begin{align}
\label{eq:MaxUkMinusBoundLargeEps}
U_k^- \leq x + n \cdot \epsilon^{2/3} x^{2/3} \leq x + \epsilon x \leq 2\epsilon x  ~,~ \mbox{ on the event } A_1 \cap \ldots \cap A_k.
\end{align}
Also, by Lemma
\ref{lem:ConcentrationEstimateUkVkPlusMinusSingleStepDifferences}, there
exist some constants $c_1, c_2 > 0$ such that:
\begin{align}
\label{eq:UkUkminus1DiffBound}
& P_{x,r}^U\Big(|U_k^- - U_{k-1}^-| > \epsilont z \Big| U_{k-1}^- = z \Big) \leq c_1 \exp(-c_2 \cdot \epsilont^2 \cdot z) ~,~\mbox{ for each } k,z \mbox{ and } 0 < \epsilont \leq 1. \nonumber \\
& P_{x,r}^U\Big(|U_k^- - U_{k-1}^-| > \epsilont z \Big| U_{k-1}^- = z\Big) \leq c_1 \exp(-c_2 \cdot \epsilont \cdot z) ~,~\mbox{ for each } k,z \mbox{ and } \epsilont \geq 1.
\end{align}
Thus, for each $\epsilon^{2/3} x^{2/3} \leq z \leq 2\epsilon x$ and $1 \leq k
\leq n$,
\begin{align}
\label{eq:IfUkMinus1Large}
& P_{x,r}^U(A_k^c | U_{k-1}^- = z)
=  P_{x,r}^U\Big( |U_k^- - U_{k-1}^-| > \Big(\frac{\epsilon^{2/3} x^{2/3}}{z}\Big) \cdot z  \Big| U_{k-1}^- = z \Big) \nonumber \\
& \leq c_1 \exp\Big( -c_2 \frac{(\epsilon^{2/3} x^{2/3})^2}{2\epsilon x} \Big) = c_1 \exp\Big(-\frac{c_2}{2} \epsilon^{1/3} x^{1/3} \Big)
\end{align}
and, for each $z < \epsilon^{2/3} x^{2/3}$ and $1 \leq k \leq n$,
\begin{align}
\label{eq:IfUkMinus1Small}
P_{x,r}^U(A_k^c | U_{k-1}^- = z)
=  P_{x,r}^U\Big( |U_k^- - U_{k-1}^-| > \Big(\frac{\epsilon^{2/3} x^{2/3}}{z}\Big) \cdot z  \Big|U_{k-1}^- = z \Big)
\leq c_1 \exp\Big( -c_2 \epsilon^{2/3} x^{2/3} \Big).
\end{align}
Note that the inequality \eqref{eq:IfUkMinus1Small} remains valid when $z =
0$ (even though the derivation above has an issue dividing by 0), since in
this case $U_k^- = U_{k-1}^- = 0$, deterministically. Now, since $\epsilon, x
\geq 1$, by assumption, and the process $(U_k^-)$ is a Markov chain, it
follows from \eqref{eq:MaxUkMinusBoundLargeEps}, \eqref{eq:IfUkMinus1Large},
and \eqref{eq:IfUkMinus1Small} that
\begin{align*}
P_{x,r}^U(A_k^c |A_1, \ldots, A_{k-1}) \leq c_1 \exp\left(-\frac{c_2}{2} \epsilon^{1/3} x^{1/3} \right) ~,~\mbox{ for } 1 \leq k \leq n.
\end{align*}
Thus,
\begin{align*}
P_{x,r}^U\left(\bigcup_{k=1}^n A_k^c \right)
\leq n \cdot c_1 \exp\left(-\frac{c_2}{2} \epsilon^{1/3} x^{1/3} \right)
< \epsilon^{1/3} x^{1/3} \cdot c_1 \exp\left(-\frac{c_2}{2} \epsilon^{1/3} x^{1/3} \right).
\end{align*}
Now, when $\tau_s^S = n$, $\MM \leq n \cdot \epsilon^{2/3} x^{2/3} <
\epsilon^{1/3} x^{1/3} \cdot \epsilon^{2/3} x^{2/3} = \epsilon x $ on the
event $\cap_{k=1}^n A_k$. So,
\begin{align}
\label{eq:MBoundTausSequalnLargeEps}
P_{x,r}^U(\MM \geq \epsilon x|\tau_s^S=n) < \epsilon^{1/3} x^{1/3} \cdot c_1 \exp\left(-\frac{c_2}{2} \epsilon^{1/3} x^{1/3} \right).
\end{align}
Since \eqref{eq:MBoundTausSequalnLargeEps} is valid for each $n <
\epsilon^{1/3} x^{1/3}$ it follows from
\eqref{eq:MGreaterEpsilonxSplittingLargeEps} and
\eqref{eq:TausSNotTooBigLargeEps} that
\begin{align*}
P_{x,r}^U(\MM \geq \epsilon x) \leq c_5(1 + \epsilon^{1/3} x^{1/3}) \exp(-c_6 \epsilon^{1/3} x^{1/3})
\end{align*}
where $c_5 = \max\{c_1, c_3\}$ and $c_6 = \min\{\frac{c_2}{2},c_4\}$.
\end{proof}

Using \eqref{eq:MonotonicityUkVkProcesses} the concentration estimates for
the processes $(U_k^{\pm})$ and $(V_k^{\pm})$ proven above in Lemma
\ref{lem:ConcentrationEstimateForUkVkPlusMinus} yield the following
concentration estimates for the processes $(U_k)$ and $(V_k)$.

\begin{lemma}
\label{lem:ConcentrationEstimateUkVk} Let $c_7 = 2 c_5$. Then the following
hold for each $r \in \SM:$
\begin{align}
\label{eq:UkSmallEpsBound}
& P_{x,r}^U \left(\max_{0 \leq k \leq \tau_s^S} |U_k - x| \geq \epsilon x\right) \leq c_7 (1 + \epsilon^{2/3} x^{1/3}) e^{-c_6 \epsilon^{2/3} x^{1/3}} ~,~\mbox{for all $x \in \N_0$ and $0 < \epsilon \leq 1$.} \\
\label{eq:UkLargeEpsBound}
& P_{x,r}^U \left(\max_{0 \leq k \leq \tau_s^S} |U_k - x| \geq \epsilon x\right) \leq c_7 (1 + \epsilon^{1/3} x^{1/3}) e^{-c_6 \epsilon^{1/3} x^{1/3}} ~,~\mbox{for all $x \in \N_0$ and $\epsilon \geq 1$.} \\
\label{eq:VkSmallEpsBound}
& P_{x,r}^V \left(\max_{0 \leq k \leq \tau_s^R} |V_k - x| \geq \epsilon x\right) \leq c_7 (1 + \epsilon^{2/3} x^{1/3}) e^{-c_6 \epsilon^{2/3} x^{1/3}} ~,~\mbox{for all $x \in \N_0$ and $0 < \epsilon \leq 1$.} \\
\label{eq:VkLargeEpsBound}
& P_{x,r}^V \left(\max_{0 \leq k \leq \tau_s^R} |V_k - x| \geq \epsilon x\right) \leq c_7 (1 + \epsilon^{1/3} x^{1/3}) e^{-c_6 \epsilon^{1/3} x^{1/3}} ~,~\mbox{for all $x \in \N_0$ and $\epsilon \geq 1$.}
\end{align}
\end{lemma}

The next two lemmas give estimates for the expectation and variance of
$U_{\tau_s^S}$ and $V_{\tau_s^R}$ in the case that $S_0 = R_0 = s$. For these
lemmas, and the remainder of this section, $E_{x,s}^U$, $\Var_{x,s}^U$, and
$\Cov_{x,s}^U$ are used to denote, respectively, expectation, variance, and
covariance under the probability measure $P_{x,s}^U$. Similarly, $E_{x,s}^V$,
$\Var_{x,s}^V$, and $\Cov_{x,s}^V$ are used to denote expectation, variance,
and covariance with respect to $P_{x,s}^V$. Also, we define
\begin{align*}
\mu_s \equiv \EB_s(\tau_s^S) = \EB_s(\tau_s^R)
\end{align*}
to be the mean return time to state $s$ for the Markov chain $(S_k)$, or
equivalently for the Markov chain $(R_k)$. Note that $E_{x,s}^U(\tau_s^S) =
E_{x,s}^V(\tau_s^R) = \mu_s$, for any $x$.

\begin{lemma}
\label{lem:ExpectationUkVkStoppedAtTausS} As $x \rightarrow \infty$,
\begin{align*}
E_{x,s}^U(U_{\tau_s^S}) = x ~+~ \delta \cdot \mu_s ~+~ O(e^{-x^{1/4}}) ~~\mbox{ and }~~
E_{x,s}^V(V_{\tau_s^R}) = x ~+~ (1-\delta) \cdot \mu_s ~+~ O(e^{-x^{1/4}}).
\end{align*}
\end{lemma}

\begin{proof}
We will prove the statement about the expectation of $U_{\tau_s^S}$. The
proof of the analogous claim for the expectation of $V_{\tau_s^R}$ is very
similar.

Fix a realization $(s_k)_{k \in \Z}$ of the random variables $(S_k)_{k \in
\Z}$ with $s_0 = s$, and let $\omega = (\omega(k,i))_{k \in \Z, i \in \N}$ be
the corresponding cookie environment defined by $\omega(k,i) = s_k(i)$. Also,
let $t_s^S = \inf\{k > 0: s_k = s\}$ be the corresponding realization of the
random variable $\tau_s^S$. We will consider first the process $(U_k)_{k \geq
0}$ started from level $x$ in this fixed environment $\omega$.  For $k \in
\Z$, we define $\delta_k = \sum_{i=1}^M (2\omega(k,i) - 1)$ to be the net
drift induced by consuming all cookies in stack $s_k$. Also, we denote by
$E_x^{U, \omega}$ expectation with respect to the probability measure
$P_x^{U, \omega}$ for the process $(U_k)_{k \geq 0}$ in this fixed
environment $\omega$.

We decompose $E_x^{U, \omega}(U_{t_s^S})$ as
\begin{align}
\label{eq:ExOmegaUtsSDecompostion}
E_x^{U, \omega}(U_{t_s^S}) = E_x^{U, \omega}(U_0) + \sum_{k=1}^{t_s^S} E_x^{U, \omega}(U_k - U_{k-1}) = x + \sum_{k=1}^{t_s^S} E_x^{U, \omega}(U_k - U_{k-1}).
\end{align}
Straightforward calculations show that
\begin{align*}
& ~E_x^{U, \omega}(U_k - U_{k-1}| U_{k-1} = m) = \delta_k ~,~\mbox{for each $k \geq 1$ and $m \geq M$, and} \\
& \left|E_x^{U, \omega}(U_k - U_{k-1}| U_{k-1} = m)\right| \leq M ~,~ \mbox{for each $k \geq 1$ and $m \geq 0$}.
\end{align*}
Thus,
\begin{align*}
\left|E_x^{U, \omega}(U_k - U_{k-1}) - \delta_k \right|
& = \Big| P_x^{U, \omega}(U_{k-1} \geq M) \cdot E_x^{U, \omega}(U_k - U_{k-1}|U_{k-1} \geq M)  \\
&~~~~~~~+~ P_x^{U, \omega}(U_{k-1} < M) \cdot E_x^{U, \omega}(U_k - U_{k-1}|U_{k-1} < M) - \delta_k \Big| \\
& \leq \left| P_x^{U, \omega}(U_{k-1} \geq M) \cdot \delta_k ~-~\delta_k \right| + P_x^{U, \omega}(U_{k-1} < M) \cdot M \\
& \leq 2M \cdot P_x^{U, \omega} (U_{k-1} < M).
\end{align*}
Plugging into \eqref{eq:ExOmegaUtsSDecompostion} gives
\begin{align}
\label{eq:ExUtsOmegaDiffUpperBound}
\Big| E_x^{U, \omega}(U_{t_s^S}) - x - \sum_{k = 1}^{t_s^S} \delta_k \Big|
\leq 2M \sum_{k=1}^{t_s^S} P_x^{U, \omega}(U_{k-1} < M).
\end{align}

So far our analysis has been for a fixed environment $\omega$. To prove the
lemma we will need to take expectations with respect to the probability
measure $\PB_s$ on environments. Recall that, for $r \in \SM$, $\delta(r)
\equiv \sum_{i=1}^M (2 r(i) - 1)$ is the net drift induced by all cookies in
stack $r$. Taking expectation of the random variable $g(\omega) =
\sum_{k=1}^{t_s^S(\omega)} \delta_k(\omega)$ with respect to $\PB_s$ gives
\begin{align}
\label{eq:SumExpectationExchangeSeries}
& \EB_s\left[ \sum_{k=1}^{t_s^S} \delta_k \right]
= \EB_s \left[ \sum_{k=1}^{t_s^S}  \sum_{r \in \SM} \delta(r) \cdot \indicator\{s_k = r\} \right]
= \EB_s \left[ \sum_{r \in \SM} \sum_{k=1}^{t_s^S} \delta(r) \cdot \indicator\{s_k = r\} \right] \nonumber \\
& \stackrel{(*)}{=} \sum_{r \in \SM} \EB_s \left[ \sum_{k=1}^{t_s^S} \delta(r) \indicator\{s_k = r\} \right]
=  \sum_{r \in \SM} \delta(r) \cdot \big[ \pi(r) \cdot \mu_s \big]
= \delta \cdot \mu_s.
\end{align}
In step (*) we have used Fubini's Theorem to interchange the sum with the
expectation; this is applicable since $\EB_s \left( \sum_{r \in \SM} \left|
\sum_{k=1}^{t_s^S} \delta(r) \cdot \indicator\{s_k = r\} \right| \right) \leq
\EB_s(M \cdot t_s^S) = M \mu_s < \infty$. Combining
\eqref{eq:ExUtsOmegaDiffUpperBound} and
\eqref{eq:SumExpectationExchangeSeries}, and using the fact that $\EB_s[
E_x^{U, \omega}(U_{t_s^S}) ] = E_{x,s}^U(U_{\tau_s^S})$, gives
\begin{align}
\label{eq:ExUtsAveragedDiffUpperBoundPreliminary}
\Big| E_{x,s}^U(U_{\tau_s^S}) \,{-}\, x \,{-}\, \delta \cdot \mu_s \Big|
\,{=}\, \left| \EB_s \left( E_x^{U, \omega}(U_{t_s^S}) \,{-}\, x \,{-}\,  \sum_{k=1}^{t_s^S} \delta_k \right) \right|
\,{\leq}\, 2M \,{\cdot}\, \EB_s \left( \sum_{k=1}^{t_s^S} P_x^{U, \omega}(U_{k-1} \,{<}\, M) \right).
\end{align}
Denote $p_{x,k} = P_{x,s}^U(U_k^- < M)$ and $q_n = \PB_s(\tau_s^S = n)$. By
construction of the process $(U_k^-)_{k \geq 0}$, $P_x^{U, \omega}(U_{k-1} <
M) \leq P_x^{U, \omega}(U_{k-1}^- < M) = p_{x,k-1}$,  for $\PB_s$ a.e.
$\omega$. Thus, it follows from
\eqref{eq:ExUtsAveragedDiffUpperBoundPreliminary} that
\begin{align}
\label{eq:ExUtsAveragedDiffUpperBound}
\Big| E_{x,s}^U(U_{\tau_s^S}) - x - \delta \cdot \mu_s \Big|
\leq 2M \sum_{n=1}^{\infty} q_n \sum_{k=1}^n p_{x,k-1}.
\end{align}
We define $n_0 = \floor{({1/2})^{5/3} x^{1/3}}$. Splitting the sum at $n_0$,
the right hand side of \eqref{eq:ExUtsAveragedDiffUpperBound} may be bounded
as follows:
\begin{align}
\label{eq:ExUtsAveragedDiffUpperBoundSplitAtn0}
2M \sum_{n=1}^{\infty} q_n \sum_{k=1}^n p_{x,k-1}
& \leq 2M \cdot \left[ \sum_{n=1}^{n_0} q_n \sum_{k=1}^{n_0} p_{x,k-1} ~+~ \sum_{n > n_0} q_n \sum_{k=1}^n p_{x,k-1} \right] \nonumber \\
& \leq 2M \cdot \left[ \sum_{k=1}^{n_0} p_{x,k-1} ~+~ \sum_{n > n_0} q_n \cdot n\right].
\end{align}
By \eqref{eq:tausStausRExponentialTail},
\begin{align}
\label{eq:BoundSumtausSTailExpectation}
\sum_{n > n_0} q_n \cdot n \leq \sum_{n > n_0} c_3 e^{-c_4 n}  \cdot n = O(e^{-x^{1/4}}).
\end{align}
Also, using \eqref{eq:MBoundFixedSmallEnoughTime} with $\epsilon = 1/2$ shows
that, for all $x > 2M$ and $0 \leq k \leq n_0$,
\begin{align*}
p_{x,k}
\leq P_{x,s}^U\left(|U_k^- - x| > \frac{1}{2} x\right)
\leq P_{x,s}^U\left(\max_{0 \leq j \leq n_0} |U_j^- - x| > \frac{1}{2} x\right)
\leq c_1 n_0 e^{-c_2 n_0},
\end{align*}
for some constants $c_1, c_2 > 0$. Hence,
\begin{align}
\label{eq:BoundSumpxk}
\sum_{k=1}^{n_0} p_{x,k-1} \leq n_0 \cdot c_1 n_0 e^{-c_2 n_0} = O(e^{-x^{1/4}}).
\end{align}
Combining \eqref{eq:ExUtsAveragedDiffUpperBound}-\eqref{eq:BoundSumpxk} shows
that $\Big| E_{x,s}^U(U_{\tau_s^S}) - x - \delta \cdot \mu_s \Big| =
O(e^{-x^{1/4}})$, which proves the lemma.
\end{proof}

\begin{lemma} As $x \rightarrow \infty$,
\label{lem:VarianceUkVkStoppedAtTaus}
\begin{align*}
\Var_{x,s}^U(U_{\tau_s^S}) = 2x \cdot \mu_s + O(x^{1/2}) ~~\mbox{ and }~~
\Var_{x,s}^V(V_{\tau_s^R}) = 2x \cdot \mu_s + O(x^{1/2}).
\end{align*}
\end{lemma}

\begin{proof}
We will prove the statement about the variance of $U_{\tau_s^S}$. The proof
of the analogous statement for the variance of $V_{\tau_s^R}$ is again very
similar.
A central element of the proof is the following claim. \\

\noindent \emph{Claim:} There exists a non-negative random variable $\Delta$
with finite variance (defined on some outside probability space, separate
from the $(U_k^+)$ and $(U_k^-)$ processes) such that
\begin{align}
\label{eq:UkPlusUkMinusDiffStochLessDelta}
U_{\tau_s^S}^+ - U_{\tau_s^S}^- \stackrel{stoch}{\leq} \Delta ~,~\mbox{ under $P_{x,s}^U$, for any $x \in \N_0$.}
\end{align}
Note that although the distributions of $U_{\tau_s^S}^+$ and $U_{\tau_s^S}^-$
do depend
on the initial value $U_0^+  = U_0^- = x$, the random variable $\Delta$ does not. \\

The proof of the claim will be given after the main proof of the lemma. The
basic idea for the proof of the lemma is to approximate the process $(U_k)_{k
\geq 0}$ by the process $(U_k^*)_{k \geq 0}$ defined by
\begin{align*}
U_0^* = U_0 = x ~~\mbox{ and }~~ U_{k+1}^* = \sum_{i=1}^{U_k^*} \GM_i^{k+1} ,~k \geq 0
\end{align*}
where the random variables $\GM_i^k$ are as in
$\eqref{eq:UkPlusMinusGeometricRepresentation}$. This process $(U_k^*)_{k
\geq 0}$ is a standard Galton-Watson branching process with Geo(1/2)
offspring distribution, independent of $(S_k)$, and satisfies
\begin{align}
\label{eq:MonotonicityUkVkProcesses2}
U_k^- \leq U_k^* \leq U_k^+ ~,~\mbox{ for all } k \geq 0.
\end{align}
We express $U_{\tau_s^S}$ as
\begin{align}
\label{eq:UTausSDecomposition}
U_{\tau_s^S} = U_{\tau_s^S}^* + \widetilde{\Delta},~\mbox{ where }~~\widetilde{\Delta} \equiv U_{\tau_s^S} - U_{\tau_s^S}^*.
\end{align}
By \eqref{eq:MonotonicityUkVkProcesses} and
\eqref{eq:MonotonicityUkVkProcesses2}, along with the claim
\eqref{eq:UkPlusUkMinusDiffStochLessDelta}, $|\widetilde{\Delta}|
\stackrel{stoch}{\leq} \Delta$, for a random variable $\Delta$ with finite
variance (not depending on $x$). The first term $U_{\tau_s^S}^*$ can be
analyzed exactly. Since $(U_k^*)_{k \geq 0}$ is a standard Geo(1/2)
Galton-Watson branching processes started from level $x$, we have
\begin{align*}
E_{x,s}^U (U_k^*) = x ~\mbox{ and }~\Var_{x,s}^U(U_k^*) = 2xk
\end{align*}
for any fixed $k > 0$. Thus, since the stopping time $\tau_s^S$ is
independent of the process $(U_k^*)$ and a.s. finite,
\begin{align}
\label{eq:VarUTausS0}
\Var_{x,s}^U(U_{\tau_s^S}^*) & = E_{x,s}^U[ \Var_{x,s}^U(U_{\tau_s^S}^*|\tau_s^S)] + \Var_{x,s}^U[E_{x,s}^U(U_{\tau_s^S}^*|\tau_s^S)] \\
& = E_{x,s}^U(2x \cdot \tau_s^S) + \Var_{x,s}^U(x) = 2x \mu_s.
\end{align}
Expanding \eqref{eq:UTausSDecomposition} gives,
\begin{align}
\label{eq:VarxsUTausSDecomposition}
\Var_{x,s}^U(U_{\tau_s^S})
= \Var_{x,s}^U(U_{\tau_s^S}^*) + \Var_{x,s}^U(\widetilde{\Delta}) + 2 \Cov_{x,s}^U(U_{\tau_s^S}^*,\widetilde{\Delta}).
\end{align}
By the calculation above, the first term on the right hand side of
\eqref{eq:VarxsUTausSDecomposition} is exactly equal to $2x \mu_s$. The
second two terms may be bounded as follows:
\begin{align}
\label{eq:BoundVarDeltaTilde}
& |\Var_{x,s}^U(\widetilde{\Delta})| \leq E_{x,s}^U(\widetilde{\Delta}^2) \leq \E(\Delta^2) \equiv C < \infty. \\
\label{eq:BoundCovUTausS0DeltaTilde}
& |2 \Cov_{x,s}^U(U_{\tau_s^S}^*,\widetilde{\Delta})| \leq 2 \left[\Var_{x,s}^U( U_{\tau_s^S}^*)\right]^{1/2} \left[\Var_{x,s}^U(\widetilde{\Delta}) \right]^{1/2}
\leq 2 \cdot (2x \mu_s)^{1/2} \cdot C^{1/2}.
\end{align}
Combining \eqref{eq:VarUTausS0}-\eqref{eq:BoundCovUTausS0DeltaTilde} shows
that $\Var_{x,s}^U(U_{\tau_s^S}) = 2x \mu_s + O(x^{1/2})$.
Thus, it remains only to prove the claim \eqref{eq:UkPlusUkMinusDiffStochLessDelta} \\

\noindent \emph{Proof of Claim:} Fix any $x \in \N_0$ and assume throughout
that $U_0^+ = U_0^- = x$. Let $(B_k)_{k \geq 0}$ be a standard Galton-Watson
branching processes with Geo(1/2) offspring distribution, started from $B_0 =
1$. Also, let $(\beta_{k,i})_{k \geq 0, i \in \N}$ and
$(\tilde{\beta}_{k,i})_{k \geq 0, i \in \N}$ all be independent random
variables such that $\beta_{k,i} \stackrel{law}{=} \tilde{\beta}_{k,i}
\stackrel{law}{=} B_k$.  Finally, let $T$ be a random time with the same
distribution as $\tau_s^S$ (under $P_{x,s}^U$) which is defined on the same
probability space as the $\beta$ and $\tilde{\beta}$ random variables, but
independently of them. We will denote the probability measure for this
probability space by $\P$, and the corresponding expectation operator by
$\E$. We claim that
\begin{align}
\label{eq:StochasticDominationUnPlusUnMinusDiff}
U_n^+ - U_n^- \stackrel{stoch}{\leq} \Delta_n \equiv \sum_{k=1}^n \sum_{i=1}^{M+1} \beta_{k,i} + \sum_{k=0}^{n-1} \sum_{i=1}^M \tilde{\beta}_{k,i},\mbox{ for all } n \in \N.
\end{align}
Since $(U_n^+, U_n^-)_{n \in \N} \perp \tau_s^S$, $(\Delta_n)_{n \in \N}
\perp T$, and $\tau_s^S \stackrel{law}{=} T$ it follows from this that
\begin{align*}
U_{\tau_s^S}^+ - U_{\tau_s^S}^- \stackrel{stoch}{\leq} \Delta \equiv \sum_{n=1}^{\infty} \Delta_n \cdot \indicator\{T=n\}.
\end{align*}
Direct computations using $\E(\beta_{k,i}) = \E(\tilde{\beta}_{k,i}) = 1$ and
$\Var(\beta_{k,i}) = \Var(\tilde{\beta}_{k,i}) = 2k$, along with
independence, give
\begin{align*}
\E(\Delta_n) = (2M+1)n ~~,~~ \Var(\Delta_n) = 2Mn^2 + n^2 + n ~~,~~ \E(\Delta_n^2) = (2M+1)(2M+2)n^2 + n.
\end{align*}
Thus, since $T$ has an exponential tail and the random variables $\Delta_n$
are independent of $T$,
\begin{align*}
\E(\Delta^2) = \sum_{n=1}^{\infty} \P(T=n) \E(\Delta^2|T=n) = \sum_{n=1}^{\infty} \P(T=n) \cdot  \E(\Delta_n^2) < \infty.
\end{align*}
So, it remains only to show \eqref{eq:StochasticDominationUnPlusUnMinusDiff}.

To this end, recall again that the processes  $(U_k^+)_{k \geq 0}$ and
$(U_k^-)_{k \geq 0}$ may be represented in terms of the independent Geo(1/2)
random variables $\GM_i^k$ according to
\eqref{eq:UkPlusMinusGeometricRepresentation}. Let us define a new process
$(Z_k)_{k \geq 0}$ as follows:
\begin{align*}
Z_0 = x ~~\mbox{ and }~~ Z_{k+1} = M + \sum_{i=1}^{\ell_k+1} \GM_i^{k+1} ~~~ \mbox{ where }~~~ \ell_k = (U_k^- - M)^+ ~+~ M ~+~ (Z_k - U_k^-).
\end{align*}
We claim that $Z_k \geq U_k^+$, for all $k \geq 0$. The proof is by induction
on $k$. For $k = 0$, we have $U_0^+ = Z_0 = x$, by assumption. Now, assume
that $Z_k \geq U_k^+$. By the definition $\ell_k \geq Z_k$, so it follows
that $\ell_k \geq U_k^+$, and therefore $Z_{k+1} = M + \sum_{i=1}^{\ell_k +
1} \GM_i^{k+1} \geq M + \sum_{i=1}^{U_k^+ + 1} \GM_i^{k+1} = U_{k+1}^+.$ This
shows that $Z_k \geq U_k^+$, for all $k$. So, $U_k^+ - U_k^- \leq Z_k - U_k^-
\equiv W_k$, for all $k$. So, to establish
\eqref{eq:StochasticDominationUnPlusUnMinusDiff} it will suffice to show that
\begin{align}
\label{eq:WnDeltanSameDist}
W_n \stackrel{law}{=} \Delta_n ~,~\mbox{ for all $n$}.
\end{align}
To prove \eqref{eq:WnDeltanSameDist} observe that $W_0 = 0$ and, for all $k
\geq 0$,
\begin{align*}
W_{k+1}
~=~ \left[ M + \sum_{i=1}^{\ell_k + 1} \GM_i^{k+1} \right ] - \left[ \sum_{i=1}^{(U_k^- - M)^+} \GM_i^{k+1} \right]
~=~ M ~+~ \sum_{i = (U_k^- - M)^+ ~+~ 1}^{(U_k^- - M)^+ ~+~ W_k ~+~ (M+1)} \GM_i^{k+1}.
\end{align*}
Thus, the process $(W_k)_{k \geq 0}$ has the same distribution as the process
$(\Wt_k)_{k \geq 0}$ defined by
\begin{align}
\label{eq:DefWtildeProcess}
\Wt_0 = 0 ~~\mbox{ and }~~ \Wt_{k+1} = M ~+ \sum_{i=1}^{\Wt_k + (M + 1)} \widetilde{\GM}_i^{k+1}
\end{align}
where $(\widetilde{\GM}_i^k)_{i,k \in \N}$ are i.i.d. Geo(1/2) random
variables (living on some probability space). To see this note that, for each
$\vec{u} = (u_0, \ldots, u_n) \in \N_0^{n+1}$, occurrence of the event
$E(\vec{u}) \equiv \{U_0^- = u_0, \ldots, U_n^- = u_n\}$ is a deterministic
function of the random variables $\GM_i^k$ such that $1 \leq k \leq n$ and,
for each such $k$, $1 \leq i \leq (u_{k-1} - M)^+$. Thus, the conditional law
of the random variables $\GM_i^k$ such that $1 \leq k \leq n$ and $i >
(u_{k-1} - M)^+$ on the event $E(\vec{u})$ is still i.i.d. Geo(1/2). Thus,
conditional on the event $E(\vec{u})$, $(W_k)_{k=0}^n$ has the same law as
$(\Wt_k)_{k = 0}^n$. And since that holds for all $n \in \N$ and $\vec{u} =
(u_0, \ldots, u_n)$ with $P_{x,s}^U(E(\vec{u})) > 0$, it follows that the
entire process $(W_k)_{k \geq 0}$ has the same (unconditional) law as
$(\Wt_k)_{k \geq 0}$.

Now, the process $(\Wt_k)$ defined by \eqref{eq:DefWtildeProcess} may be
understood as a type of branching process with migration where, at each step
$k$, in addition to all the ``regular children'' created by the branching
mechanism we also introduce $M+1$ immigrants prior to the reproduction stage
(the $M+1$ in the upper limit of the sum) along with $M$ additional
immigrants after the reproduction stage (the $M$ in front of the sum). It
follows that, for each $n$,
\begin{align}
\label{eq:WnTildeDeltanSameDist}
\Wt_n \stackrel{law}{=} \Delta_n \equiv \sum_{k=1}^n \sum_{i=1}^{M+1} \beta_{k,i} + \sum_{k=0}^{n-1} \sum_{i=1}^{M} \tilde{\beta}_{k,i}.
\end{align}
The $\tilde{\beta}_{k,i}$ sum corresponds to the descendants of the $M$
immigrants coming in after the reproduction stage in each generation that are
alive at time $n$, and the $\beta_{k,i}$ sum corresponds to the descendants
of the $(M+1)$ immigrants coming in prior to the reproduction stage in each
generation that are alive at time $n$. Since $W_n \stackrel{law}{=}
\widetilde{W}_n$ for all $n$, \eqref{eq:WnTildeDeltanSameDist} implies
\eqref{eq:WnDeltanSameDist}.
\end{proof}

Our final result of this section is an ``overshoot lemma'' which gives
concentration estimates for $(\Vh_k)$ when exiting certain intervals.
Analogous statements also hold for the process $(\Uh_k)$, but we will not
need these.

\begin{lemma}
\label{lem:ModifiedLemma51FromLimitLawsERW} Let $\tau^{\Vh}_{x^+} = \inf\{k >
0: \Vh_k \geq x\}$ and $\tau^{\Vh}_{x^-} = \inf\{k > 0: \Vh_k \leq x\}$. Then
there exist constants $c_8, c_9 > 0$ and $N \in \N$ such that for all $x \geq
N$ the following hold:
\begin{itemize}
\item[(i)] $\sup_{0 \leq z < x} P_{z,s}^V (\Vh_{\tau_{x^+}^{\Vh}} \,{>}\, x
    \,{+}\, y| \tau_{x^+}^{\Vh} \,{<}\, \tau_0^{\Vh}) \,{\leq}\,
    \begin{cases}
c_8(1 \,{+}\, y^{2/3}x^{-1/3}) e^{-c_9 y^{2/3} x^{-1/3}},~ \mbox{ for } 0 \leq y \leq x \\
c_8(1 + y^{1/3}) e^{-c_9 y^{1/3}},~ \mbox{ for } y \geq x.
\end{cases}$
\item[(ii)] $\sup_{x < z < 4x} P_{z,s}^V (\Vh_{\tau_{x^-}^{\Vh} \wedge
    \tau_{(4x)^+}^{\Vh}} < x - y)
\leq c_8(1 + y^{2/3}x^{-1/3}) e^{-c_9 y^{2/3} x^{-1/3}} ~,~ \mbox{ for } 0 \leq y \leq x$. \\
\end{itemize}
\end{lemma}
Before proceeding to the main proof we isolate an important piece as its own
lemma. For the statement of this lemma recall that $P_{v_0}^{V,\omega}$ is
the probability measure for the backward branching process $(V_k)_{k \geq 0}$
defined by \eqref{eq:DefBackwardBP} in the particular environment $\omega \in
\Omega$, started from $V_0 = v_0$.

\begin{lemma}
\label{lem:NotToBigJumpOnExitingIntervalOneStepVk} There exist constants $C,c
> 0$ such that for any environment $\omega$ satisfying $\omega(k,i) = 1/2$
for all $k \in \Z$ and $i > M$,
\begin{align}
\label{eq:NotToBigJumpOnExitingIntervalOneStepVk}
P^{V,\omega}_z(V_1 > x + y|V_1 \geq x) \leq C(e^{-cy^2/x} + e^{-cy})
\end{align}
for each $y \geq 6M$, $x \geq 2M +1 $, and $0 \leq z < x$.
\end{lemma}

\begin{proof}
This is implicit in the proof of Lemma 5.1 in \cite{Kosygina2011}. That lemma
is stated for the process $(V_k)$ under the averaged measure on environments
in an (IID) and (BD) setting, rather than for a fixed environment $\omega$.
But the proof of the equivalent statement to
\eqref{eq:NotToBigJumpOnExitingIntervalOneStepVk} in this setting (given
within the proof of Lemma 5.1 in \cite{Kosygina2011}) uses only the fact that
all cookies stacks are of bounded height $M$.
\end{proof}

\begin{proof}[Proof of Lemma \ref{lem:ModifiedLemma51FromLimitLawsERW}]
We will prove (i) under the assumption that $y \geq 12M$ and $x \geq N \equiv
2M + 1$.  By increasing the constant $c_8$ if necessary, the desired
inequalities in (i) will then hold for all $y \geq 0$. A similar proof gives
(ii) with some different values of the constants.

Under the measure $P_{z,s}^V$ we have $\Vh_0 = V_0 = z < x$, so
$\tau_{x^+}^{\Vh} > 0$. Define a random variable $k_0$ as follows: Given that
$\tau_{x^+}^{\Vh} = j + 1$, for some $j \in \N_0$, let $k_0 = \inf\{k >
\tau_j : V_{k} \geq x \}$. In other words, $(j + 1)$ is the first time $\ell$
that the process $(\Vh_{\ell})$ reaches level $x$ or above, and $k_0$ is the
first time $k > \tau_j$ that the process $(V_k)$ reaches level $x$ or above.
Since the process $(V_k,R_k)$ is Markovian under $P_{z,s}^V$ it follows that,
for each $y \in \N$,
\begin{align*}
P&_{z,s}^V(\Vh_{\tau_{x^+}^{\Vh}} > x + y| \tau_{x^+}^{\Vh} < \tau_0^{\Vh}) \\
&\leq \sup_{k \in \N, 0 \leq w < x, r \in \SM} P_{z,s}^V (\Vh_{\tau_{x^+}^{\Vh}} > x + y| \tau_{x^+}^{\Vh} < \tau_0^{\Vh}, k_0 = k, R_{k-1} = r, V_{k-1} = w) \\
&= \sup_{0 \leq w < x, r \in \SM} P_{w,r}^V(V_{\tau_s^R} > x + y| V_1 \geq x, V_{\tau_s^R} \geq x) \\
& = \sup_{0 \leq w < x, r \in \SM} \frac{P_{w,r}^V(V_{\tau_s^R} > x + y | V_1 \geq x)}{P_{w,r}^V(V_{\tau_s^R} \geq x|V_1 \geq x)}.
\end{align*}
Thus, it will suffice to show the following two claims to establish the lemma. \\

\noindent \emph{Claim 1:} There exist $C_1, C_2 > 0$ such that for all $y
\geq 12M$, $x \geq 2M + 1$, $0 \leq w < x$, and $r \in \SM$
\begin{align*}
P_{w,r}^V(V_{\tau_s^R} > x + y | V_1 \geq x) \leq
\begin{cases}
C_1(1 + y^{2/3}x^{-1/3}) e^{-C_2 y^{2/3} x^{-1/3}} ~,~ \mbox{ if } y \leq x \\
C_1(1 + y^{1/3}) e^{-C_2 y^{1/3}} ~,~ \mbox{ if } y \geq x.
\end{cases}
\end{align*}

\noindent \emph{Claim 2:} There exists $C_3 > 0$ such that for each $x \geq
2M + 1$, $0 \leq w < x$, and $r \in \SM$
\begin{align*}
P_{w,r}^V(V_{\tau_s^R} \geq x|V_1 \geq x) \geq C_3.
\end{align*}

\noindent \emph{Proof of Claim 1.} Fix any $x \geq 2M + 1$, $0 \leq w < x$,
and $r \in \SM$. For each $y \geq 12M$ we have
\begin{align}
\label{eq:DecompositionForClaim1}
& P_{w,r}^V(V_{\tau_s^R} > x + y | V_1 \geq x)
= \frac{P_{w,r}^V(V_{\tau_s^R} > x + y, V_1 \geq x)}{P_{w,r}^V(V_1 \geq x)} \nonumber \\
&~~ = \frac{P_{w,r}^V(V_{\tau_s^R} > x + y, x \leq V_1 \leq x + y/2)}{P_{w,r}^V(V_1 \geq x)} + \frac{P_{w,r}^V(V_{\tau_s^R} > x + y, V_1 > x + y/2)}{P_{w,r}^V(V_1 \geq x)} \nonumber \\
&~~ \leq \frac{P_{w,r}^V(V_{\tau_s^R} > x + y, x \leq V_1 \leq x + y/2)}{P_{w,r}^V(x \leq V_1 \leq x + y/2)} + \frac{P_{w,r}^V(V_1 > x + y/2)}{P_{w,r}^V(V_1 \geq x)} \nonumber \\
&~~= P_{w,r}^V(V_{\tau_s^R} > x + y| x \leq V_1 \leq x + y/2) ~+~P_{w,r}^V(V_1 > x + y/2|V_1 \geq x).
\end{align}
By Lemma \ref{lem:NotToBigJumpOnExitingIntervalOneStepVk},
\begin{align}
\label{eq:BoundProbOfLargeCrossingJump}
P_{w,r}^V(V_1 > x + y/2|V_1 \geq x) \leq C(e^{-c\floor{y/2}^2/x} + e^{-c\floor{y/2}})
\end{align}
for some $C,c > 0$. So, we need only bound $P_{w,r}^V(V_{\tau_s^R} > x + y| x
\leq V_1 \leq x + y/2)$. Define $w_0 = \floor{x + y/2}$. By monotonicity of
$(V_k)$ with respect to its initial condition and Lemma
\ref{lem:ConcentrationEstimateUkVk},
\begin{align*}
P_{w,r}^V&(V_{\tau_s^R} > x + y| x \leq V_1 \leq x + y/2)
\leq \sup_{r' \in \SM} P^V_{w_0,r'}(V_{\tau_s^R} > x + y)\\
& \leq \sup_{r' \in \SM} P^V_{w_0,r'}(|V_{\tau_s^R} - w_0| > y/2)
\leq c_7(1 + (y/2)^{2/3}w_0^{-1/3})e^{-c_6 (y/2)^{2/3} w_0^{-1/3}}.
\end{align*}
Now, for $x \leq y$ we have $\frac{1}{8}y < (y/2)^2 w_0^{-1} < y$, and for $x
\geq y$ we have $\frac{1}{8}y^2/x < (y/2)^2 w_0^{-1} < y^2/x$. So, it follows
that
\begin{align}
\label{eq:BoundProbIfSmallCrossingJump}
P_{w,r}^V&(V_{\tau_s^R} > x + y| x \leq V_1 \leq x + y/2) \leq
\begin{cases}
c_7(1 + y^{2/3}x^{-1/3}) e^{-\frac{1}{2} c_6 y^{2/3} x^{-1/3}} ~,~ \mbox{ if } y \leq x \\
c_7(1 + y^{1/3}) e^{-\frac{1}{2} c_6 y^{1/3}} ~,~ \mbox{ if } y \geq x.
\end{cases}
\end{align}
Combining \eqref{eq:DecompositionForClaim1}, \eqref{eq:BoundProbOfLargeCrossingJump}, and \eqref{eq:BoundProbIfSmallCrossingJump} proves the claim. \\

\noindent \emph{Proof of Claim 2.} Let $(\widetilde{\HM}_i)_{i \in \N}$ be
i.i.d. Geo(1/2) random variables, and let $q =\break \inf_{y \geq 2M + 1}
\P\left(\sum_{i = 1}^{y - M} \widetilde{\HM}_i \geq y\right)$. By the central
limit theorem, $\lim_{y \to \infty} \P\left(\sum_{i = 1}^{y - M}
\widetilde{\HM}_i \geq y\right) = 1/2$, so $q > 0$. Further, since
$(V_k^-)_{k \geq 0}$ is a Markov chain independent of $(R_k)_{k \geq 0}$ it
follows from \eqref{eq:MonotonicityUkVkProcesses} and
\eqref{eq:VkPlusMinusGeometricRepresentation} that, for any $(r_0, \ldots,
r_m) \in \SM^{m+1}$ and $n \geq x \geq 2M + 1$,
\begin{align*}
&P_{n,r_0}^V(V_m \geq x|R_0 = r_0, \ldots, R_m = r_m) \geq P_{n,r_0}^V(V_m^- \geq x|R_0 = r_0, \ldots, R_m = r_m) \\
& = P_{n,r_0}^V(V_m^- \geq x) \geq \prod_{j=1}^{m} P_{n,r_0}^V(V_j^- \geq V_{j-1}^-|V_{j-1}^- \geq \ldots \geq V_1^- \geq V_0^-) \geq q^m.
\end{align*}
To complete the proof we note that by \eqref{eq:tausStausRExponentialTail}
there exist some $m_0 \in \N$ and $p > 0$ such that $\PB_{r}(\tau_s^R \leq
m_0) \geq p$, for each $r \in \SM$. Therefore, for any $x \geq 2M + 1$, $0
\leq w < x$, and $r' \in \SM$,
\begin{align*}
P_{w,r'}^V(V_{\tau_s^R} \geq x|V_1 \geq x)
\geq \inf_{n \geq x, r \in \SM} P_{n,r}^V(V_{\tau_s^R} \geq x)
\geq p \cdot q^{m_0} \equiv C_3 > 0.\tag*{\qedhere}
\end{align*}
\end{proof}

\section{Proof of Theorem \texorpdfstring{\ref{thm:TransienceRecurrence}}{1.11}}
\label{sec:ProofTheoremTransienceRecurrence}

The proof of Theorem \ref{thm:TransienceRecurrence} relies on the following
lemma concerning transience vs. recurrence of integer-valued Markov chains,
which will be proved in Appendix
\ref{app:ProofLemmaTransienceRecurrenceConditionsForMC}.

\begin{lemma}
\label{lem:TransienceRecurrenceConditionsForMC} Let $(Z_k)_{k \geq 0}$ be an
irreducible, time-homogeneous Markov chain on state space $\N_0$. Let $\P_x$
be the probability measure for $(Z_k)$ started from $Z_0 = x$, and let $\E_x$
be the corresponding expectation operator. Assume that there exist constants
$C_1, C_2 > 0$ and $x_0 \in \N$ such that:
\begin{align}
\label{eq:MCConcentrationCondition}
& \P_x( |Z_1 - x| \geq \epsilon x) \leq C_1(1 + \epsilon^{2/3}x^{1/3})e^{-C_2 \epsilon^{2/3} x^{1/3}} ~,~ \mbox{ for all }~ 0 < \epsilon \leq 1 \mbox{ and } x \geq x_0. \nonumber \\
& \P_x( |Z_1 - x| \geq \epsilon x) \leq C_1(1 + \epsilon^{1/3}x^{1/3})e^{-C_2 \epsilon^{1/3} x^{1/3}} ~,~ \mbox{ for all }~ \epsilon \geq 1 \mbox{ and } x \geq x_0.
\end{align}
Define $\rho(x)$, $\nu(x)$, and $\theta(x)$ by
\begin{align}
\label{eq:DefRhoNuTheta}
\rho(x) = \E_x(Z_1- x) ~~,~~ \nu(x) = \frac{\E_x[(Z_1 - x)^2]}{x} ~~,~~ \theta(x) = \dfrac{2 \rho(x)}{\nu(x)}
\end{align}
and assume also that $\liminf_{x \to \infty} \nu(x) > 0$. Then the following
hold.
\begin{itemize}
\item[(i)] If there exists some $a \in (1,\infty)$ such that $\theta(x)
    \leq 1 + \dfrac{1}{a \ln(x)}$ for all sufficiently large $x$, then the
    Markov chain $(Z_k)$ is recurrent.
\item[(ii)]  If there exists some $a \in (1,\infty)$ such that $\theta(x)
    \geq 1 + \dfrac{2a}{\ln(x)}$ for all sufficiently large $x$, then the
    Markov chain $(Z_k)$ is transient.
\end{itemize}
\end{lemma}

\begin{corollary}
\label{cor:TransienceRecurrenceConditionsForMC} Let $(Z_k)_{k \geq 0}$ be a
time-homogeneous Markov chain on state space $\N_0$ and define $\rho(x)$,
$\nu(x)$, and $\theta(x)$ as in Lemma
\ref{lem:TransienceRecurrenceConditionsForMC}. Assume that the concentration
condition \eqref{eq:MCConcentrationCondition} is satisfied and $\liminf_{x
\to \infty} \nu(x) > 0$.  Assume also that
\begin{align}
\label{eq:ZkChain0MaybeAbsorbing}
\P_x(Z_1 = y) > 0, \mbox{ for all $x \geq 1$ and $y \geq 0$}.
\end{align}
Then the following hold.
\begin{itemize}
\item[(i)] If there exists some $a \in (1,\infty)$ such that $\theta(x)
    \leq 1 + \frac{1}{a \ln(x)}$ for all sufficiently large $x$, then
    $\P_z(Z_k > 0, \forall k > 0) = 0$, for each $z \geq 1$.
\item[(ii)] If there exists some $a \in (1,\infty)$ such that $\theta(x)
    \geq 1 + \frac{2a}{ \ln(x)}$ for all sufficiently large $x$, then
    $\P_z(Z_k > 0, \forall k > 0) > 0$, for each $z \geq 1$.
\end{itemize}
\end{corollary}

\begin{remark}
Note that we do not assume in the corollary that the Markov chain $(Z_k)$ is
irreducible. State $0$ may be absorbing, and in fact when we apply the
corollary in the proof of Theorem \ref{thm:TransienceRecurrence} below it
will be for the Markov chain $(\Uh_k)$, where state $0$ is absorbing.
\end{remark}

\begin{proof}[Proof of Corollary \ref{cor:TransienceRecurrenceConditionsForMC}]
The Markov chain $(Z_k)_{k \geq 0}$ is not necessarily irreducible since
state $0$ may be absorbing, but the modified Markov chain $(\Zt_k)_{k \geq
0}$ with transition probabilities
\begin{align*}
\mbox{ $\widetilde{\P}_0(\Zt_1 = 1) = 1$ ~~and~~ $\widetilde{\P}_x(\Zt_1 = y) = \P_x(Z_1 = y)$, $x \geq 1$ and $y \geq 0$}
\end{align*}
is certainly irreducible. Due to \eqref{eq:ZkChain0MaybeAbsorbing}, this
modified chain $(\Zt_k)$, like the original chain $(Z_k)$, has positive
transition probabilities from each state $x \geq 1$ to each state $y \geq 0$.
Thus, if $(\Zt_k)$ is transient it has positive probability never to hit
state $0$ started from any state $z \geq 1$. Applying Lemma
\ref{lem:TransienceRecurrenceConditionsForMC} shows that (i) and (ii) of the
corollary hold for the Markov chain $(\Zt_k)$.  This implies (i) and (ii)
also hold for $(Z_k)$, since $(Z_k)$ and $(\Zt_k)$ have the same transition
probabilities from all non-zero states.
\end{proof}

Before proceeding to the proof of Theorem \ref{thm:TransienceRecurrence} we
isolate one important observation in the following remark. This observation
will also be necessary for the proof of Theorem \ref{thm:Ballisticity} in
Section \ref{subsec:ProofTheoremBallisticity}.

\begin{remark}[Invariance of Assumption (A) under interchange of spatial directions] ~\\
Assume that Assumption (A) is satisfied for the probability measure $\PB$,
and let $(S_k)$ be the associated stack sequence and $R_k = S_{-k}$, $k \in
\Z$. Construct the \emph{spatially reversed probability measure}
$\widetilde{\PB}$ on environments $\widetilde{\omega}$ by the following
coupling: $\widetilde{\omega}(k,i) = \St_k(i)$, for $k \in \Z$ and $i \in
\N$, where the process $(\St_k)$ is given by $\St_k(i) = 1 - S_{-k}(i) = 1 -
R_k(i)$. Thus, the probability of jumping right (left) on the $i$-th visit to
site $k$ for the random walk $(\Xt_n)$ in the spatially reversed environment
$\widetilde{\omega}$ is the same as the probability of jumping left (right)
on the $i$-th visit to site $-k$ for the original random walk $(X_n)$. With
this construction there is also a natural coupling between the random walks
$(\Xt_n)$ and $(X_n)$, when both are started from site $0$, such that
$(\Xt_n)$ jumps right whenever $(X_n)$ jumps left, and vice versa. This
implies $\Xt_n = - X_n$, for all $n$, and so we may consider $(\Xt_n)$ as the
spatially reversed version of the random walk $(X_n)$.

Moreover, since the original model with probability measure $\PB$ is assumed
to satisfy Assumption (A), we know that $(R_k)$ is a uniformly ergodic Markov
chain, which implies $(\St_k)$ is also a uniformly ergodic Markov chain.
Similarly, since the original model is assumed to satisfy Assumption (A), we
know that $(S_k)$ is a uniformly ergodic Markov chain, which implies
$(\Rt_k)$ is a uniformly ergodic Markov chain (where $\Rt_k \equiv
\St_{-k}$). Thus, the spatially reversed probability measure
$\widetilde{\PB}$ also satisfies Assumption (A). Finally, note that
$\widetilde{\delta} = - \delta$, where $\widetilde{\delta}$ and $\delta$ are
the drift parameters for the probability measures $\widetilde{\PB}$ and $\PB$
on environments.
\label{rem:InvarianceAssumptionAWhenInterchangeSpatialDirections}
\end{remark}

\begin{proof}[Proof of Theorem \ref{thm:TransienceRecurrence}]
Since $\omega(k,i) = 1/2$ for each $i > M$ and $k \in \Z$, $\PB_{\pi}$ a.s.,
we have
\begin{align*}
P_{0,\pi}(\liminf_{n \to \infty} X_n = k) = P_{0,\pi}(\limsup_{n \to \infty} X_n = k) = 0 ~,~ \mbox{ for each } k \in \Z.
\end{align*}
It follows from this and Theorem \ref{thm:ZeroOneLawSE} that one of the
following must hold:
\begin{align*}
\mbox{ (a) } P_{0,\pi}(X_n \rightarrow +\infty) = 1,~ \mbox{ (b) } P_{0,\pi}(X_n \rightarrow -\infty) = 1,~ \mbox{or (c) } P_{0,\pi}(X_n = 0, i.o.) = 1.
\end{align*}
Since the probability measure $\pi$ places positive probability on each
states $r \in \SM$ we must have the same trichotomy under the measure $P_0$.
That is, either
\begin{align*}
\mbox{ (a') } P_0(X_n \rightarrow +\infty) = 1,~ \mbox{ (b') } P_0(X_n \rightarrow -\infty) = 1,~ \mbox{or (c') } P_0(X_n = 0, i.o.) = 1
\end{align*}
where (a) holds if and only if (a') holds, (b) holds if and only if (b')
holds, and (c) holds if and only if (c') holds.  We will show below that (a')
holds if $\delta > 1$, and (a') does not hold if $\delta \leq 1$. It follows
from this by interchanging spatial directions of the model, as described in
Remark \ref{rem:InvarianceAssumptionAWhenInterchangeSpatialDirections}, that
(b') holds if $\delta < -1$, and (b') does not hold if $\delta \geq -1$.
Together these facts imply (c') must hold when $\delta \in [-1,1]$. Thus, it
remains only to show the claim about (a') to establish the lemma.

To do this we will apply Corollary
\ref{cor:TransienceRecurrenceConditionsForMC} to the Markov chain $(\Uh_k)_{k
\geq 0}$ with transition probabilities given by
\eqref{eq:TransitionProbsUhk}. We consider this Markov chain under the family
of measures $P_{x,s}^U$, $x \in \N_0$, so that $\Uh_0 = U_0 = x$
deterministically. Thus, the measure $P_{x,s}^U$ for the Markov chain
$(\Uh_k)$ is the equivalent of the measure $\P_x$ for the Markov chain
$(Z_k)$ considered in the corollary.  Condition
\eqref{eq:ZkChain0MaybeAbsorbing} of the corollary is satisfied since the
cookie stacks $r \in \SM$ are all elliptic. Also, the concentration condition
\eqref{eq:MCConcentrationCondition} is satisfied due to Lemma
\ref{lem:ConcentrationEstimateUkVk}. Finally, by Lemmas
\ref{lem:ExpectationUkVkStoppedAtTausS} and
\ref{lem:VarianceUkVkStoppedAtTaus},
\begin{align*}
\rho(x) = \delta \cdot \mu_s + O(e^{-x^{1/4}}) ~~\mbox{ and }~~\nu(x) = 2 \mu_s + O(x^{-1/2}).
\end{align*}
Thus, $\liminf_{x \to \infty} \nu(x) > 0$, as required by the corollary, and
$\theta(x) = 2 \rho(x)/\nu(x) = \delta + O(x^{-1/2})$. Applying the corollary
with $z = 1$ we have
\begin{align*}
P_{1,s}^U(\Uh_k > 0, \forall k > 0) > 0~, \mbox{ if } \delta > 1 ~~~~\mbox{ and }~~~~
P_{1,s}^U(\Uh_k > 0, \forall k > 0) = 0~, \mbox{ if } \delta \leq 1.
\end{align*}
Since $0$ is an absorbing state for the process $(U_k)_{k \geq 0}$ it follows
also that
\begin{align*}
P_{1,s}^U(U_k > 0, \forall k > 0) > 0~, \mbox{ if } \delta > 1 ~~~~\mbox{ and }~~~~
P_{1,s}^U(U_k > 0, \forall k > 0) = 0~, \mbox{ if } \delta \leq 1.
\end{align*}
Hence, by Lemma \ref{lem:TransienceRecurrenceInTermsOfSurvivalUk},
\begin{align*}
P_0(X_n \rightarrow +\infty) = 1~, \mbox{ if } \delta > 1 ~~~~\mbox{ and }~~~~
P_0(X_n \rightarrow +\infty) = 0~, \mbox{ if } \delta \leq 1
\end{align*}
which establishes the claim about (a').
 \end{proof}

\section{Proof of Theorems \texorpdfstring{\ref{thm:Ballisticity}}{1.12} and \texorpdfstring{\ref{thm:LimitLaws}}{1.13}}
\label{sec:ProofTheoremsBallisticityAndLimitLaws}

The proofs of Theorems \ref{thm:Ballisticity} and \ref{thm:LimitLaws} are
based on an analysis of the backward branching process $(V_k)_{k \geq 0}$,
and follow the general strategy used in \cite{Kosygina2011}. However, there
is an additional complication due to our different probability measure $\PB$
on environments. If $\PB$ is (IID), as considered in \cite{Kosygina2011},
then the process $(V_k)_{k \geq 0}$ is Markovian. However, if $\PB$ is
Markovian, as we consider, then the process $(V_k)_{k \geq 0}$ is not. Thus,
we will instead analyze the modified process $(\Vh_k)_{k \geq 0}$, which is
Markovian, and then translate the results from the
$(\Vh_k)$ process back to the $(V_k)$ process. The outline for our proofs and the remainder of this section is described below. \\

\noindent \underline{\emph{Step 1}}: In Appendix
\ref{app:ProofOfPropsDecayAndAccumulatedSumTimeTau0Vh} we will establish the
following two propositions which are analogous to Theorems 2.1 and 2.2 of
\cite{Kosygina2011}. The general methodology is very similar to that in
\cite{Kosygina2011}, so we will provide only a general outline and reprove a
few key lemmas from which everything else follows just as in
\cite{Kosygina2011}.

\begin{proposition} If $\delta > 1$ then there exists some $c_{10} > 0$ such that \\ $P_{0,s}^{V}\left(\tau_0^{\Vh} > x\right) \sim c_{10} \cdot x^{-\delta}$.
\label{prop:DecayOfHittingTimeTau0Vh}
\end{proposition}

\begin{proposition} If $\delta > 1$ then there exists some $c_{11} > 0$ such that \\ $P_{0,s}^V\left( \sum_{k=0}^{\tau_0^{\Vh}} \Vh_k > x \right) \sim c_{11} \cdot x^{-\delta/2}$.
\label{prop:DecayOfAccumulatedSumTillTimeTau0Vh}
\end{proposition}

\noindent \underline{\emph{Step 2}}: Define $\sigma_0^V = \inf \{k > 0: V_k =
0 \mbox{ and } R_k = s\}$. In Appendix
\ref{app:ProofOfPropsDecayAndAccumulatedSumTimeSigma0V} we will translate the
above results for the process $(\Vh_k)$ to the process $(V_k)$ proving the
following propositions.

\begin{proposition} If $\delta > 1$ then $P_{0,s}^V\left(\sigma_0^V > x\right) \sim c_{12} \cdot x^{-\delta}$, where $c_{12} \equiv c_{10} \cdot \mu_s^{\delta}$.
\label{prop:DecayOfHittingTimeSigma0V}
\end{proposition}

\begin{proposition} If $\delta > 1$ then $P_{0,s}^V\left( \sum_{k=0}^{\sigma_0^V} V_k > x \right) \sim c_{13} \cdot x^{-\delta/2}$, where  $c_{13} \equiv c_{11} \cdot \mu_s^{\delta/2}$.
\label{prop:DecayOfAccumulatedSumTillTimeSigma0V}
\end{proposition}

\noindent \underline{\emph{Step 3}}: In Section
\ref{subsec:ProofTheoremBallisticity} we will use Propositions
\ref{prop:DecayOfHittingTimeSigma0V}
and \ref{prop:DecayOfAccumulatedSumTillTimeSigma0V} to prove Theorem \ref{thm:Ballisticity}. \\

\noindent \underline{\emph{Step 4}}: In Section
\ref{subsec:ProofTheoremLimitLaws} we will use Propositions
\ref{prop:DecayOfHittingTimeSigma0V}
and \ref{prop:DecayOfAccumulatedSumTillTimeSigma0V} to prove Theorem \ref{thm:LimitLaws}. \\

For future reference we observe the following simple corollary of Proposition
\ref{prop:DecayOfHittingTimeSigma0V}.

\begin{corollary}
\label{cor:HittingTime0sIsFinite} If $\delta > 1$ then, for each $x \in \N_0$
and $r \in \SM$, $P_{x,r}^V \left(\sigma_0^V < \infty \right) = 1$.
\end{corollary}

\begin{proof}
The process $(V_k, R_k)_{k \geq 0}$ is a time-homogeneous Markov chain under
$P_{x,r}^V$, for any $x,r$. Furthermore since the Markov chain $(R_k)$ is
irreducible and the cookie stacks in $\SM$ are elliptic this Markov chain on
pairs $(V_k, R_k)$ is also irreducible. By Propostion
\ref{prop:DecayOfHittingTimeSigma0V} the pair $(0, s)$ is a recurrent state
for the Markov chain $(V_k, R_k)$, so the Markov chain itself is recurrent,
so the hitting time of state $(0,s)$ is a.s. finite starting from any initial
state $(x,r)$.
\end{proof}

\subsection{Proof of Theorem \texorpdfstring{\ref{thm:Ballisticity}}{1.12}}
\label{subsec:ProofTheoremBallisticity} Define stopping times $(\sigma_i)_{i
\geq 0}$ by
\begin{align*}
\sigma_0 = \inf \{k \geq 0: V_k = 0, R_k = s\} ~~\mbox{ and }~~ \sigma_{i+1} = \inf\{k > \sigma_i: V_k = 0, R_k = s\} ~,~ i \geq 0.
\end{align*}
Also, define $\Delta_{\sigma, i} = \sigma_i - \sigma_{i-1}$, for $i \geq 1$,
and let
\begin{align*}
Q_0 = \sum_{k = 0}^{\sigma_0} V_k ~~\mbox{ and }~~ Q_i = \sum_{k = \sigma_{i-1} + 1}^{\sigma_i} V_k ~,~ i \geq 1.
\end{align*}
Since $(V_k, R_k)_{k \geq 0}$ is a time-homogeneous Markov chain (under any
of the measures $P_x^V$, $P_{x,\pi}^V$, or $P_{x,r}^V$), it follows from
Corollary \ref{cor:HittingTime0sIsFinite} that if $\delta > 1$ then (under
any of these same measures)
\begin{align}
\label{eq:SigmaiQiAreIID}
&\mbox{The times $\sigma_i$, $i \geq 0$, are all a.s. finite and} \nonumber \\
&(Q_i)_{i \geq 1} \mbox{ and } (\Delta_{\sigma,i})_{i \geq 1} \mbox{ are each i.i.d sequences. }
\end{align}
We denote the mean of the $Q_i$'s by $\mu_Q$ and the mean of the
$\Delta_{\sigma,i}$'s by $\mu_{\sigma}$.

\begin{proof}[Proof of Theorem \ref{thm:Ballisticity}]
By Theorem \ref{thm:LLNSE} there exists some deterministic $v \in [-1,1]$
such that $X_n/n \rightarrow v$, $P_{0,\pi}$ a.s. Since the probability
measure $\pi$ places positive probability on each states $r \in \SM$, it
follows that $X_n/n \rightarrow v$, $P_{0,r}$ a.s., for each $r \in \SM$.
Hence, also $X_n/n \rightarrow v$, $P_0$ a.s. By Theorem
\ref{thm:TransienceRecurrence}, the walk is recurrent for $\delta \in
[-1,1]$, so we must have $v = 0$ in this case. We will show that
\begin{align}
\label{eq:SpeedEquationDeltaGreater1}
v > 0 ~\mbox{ if }~ \delta > 2 ~~~~\mbox{ and }~~~~ v = 0 ~\mbox{ if }~ \delta \in (1,2].
\end{align}
It follows from this by interchanging spatial directions of the model (see
Remark \ref{rem:InvarianceAssumptionAWhenInterchangeSpatialDirections}) that
$v < 0$ if $\delta < -2$ and $v = 0$ if $\delta \in [-2,-1)$.

For the remainder of the proof we assume $\delta > 1$. For $n \in \N$, define
$i_n$ by
\begin{align*}
i_n = \begin{cases}
\sup \{i \geq 0:\sigma_i \leq n \} ~,~\mbox{ if } n \geq \sigma_0 \\
-1 ~,~\mbox{ if } n < \sigma_0.
\end{cases}
\end{align*}
Observe that (with the convention $\sigma_{-1} \equiv 0$)
\begin{align}
\label{eq:inVersusn}
\sigma_{i_n} \leq n < \sigma_{i_n + 1} ~~~\mbox{ and }~~~ \sum_{j = 1}^{i_n} Q_j \leq \sum_{k = 0}^n V_k \leq \sum_{j = 0}^{i_n + 1} Q_j ~,~~ \mbox{for each $n \in \N$.}
\end{align}
Thus, by \eqref{eq:SigmaiQiAreIID} and the strong law of large numbers,
\begin{align*}
\limsup_{n \to \infty} \frac{1}{n} \sum_{k = 0}^n V_k
\leq \limsup_{n \to \infty} \frac{i_n}{\sigma_{i_n}} \cdot \left[ \frac{1}{i_n} \sum_{j = 0}^{i_n + 1} Q_j \right]
= \limsup_{i \to \infty} \frac{i}{\sigma_i} \cdot \left[ \frac{1}{i} \sum_{j = 0}^{i + 1} Q_j \right]
= \frac{\mu_{Q}}{\mu_{\sigma}}
~,~ P_{0,\pi}^V ~ \mbox{a.s.}
\end{align*}
and
\begin{align*}
\liminf_{n \to \infty} \frac{1}{n} \sum_{k = 0}^n V_k
&\geq \liminf_{n \to \infty} \frac{i_n + 1}{\sigma_{i_n + 1}} \cdot \left[ \frac{1}{i_n + 1} \sum_{j = 1}^{i_n} Q_j \right]\\
&= \liminf_{i \to \infty} \frac{i+1}{\sigma_{i+1}} \cdot \left[ \frac{1}{i+1} \sum_{j = 1}^{i} Q_j \right]
= \frac{\mu_{Q}}{\mu_{\sigma}}
~,~ P_{0,\pi}^V ~ \mbox{a.s.}
\end{align*}
So, $\frac{1}{n} \sum_{k=0}^n V_k \rightarrow \frac{\mu_{Q}}{\mu_{\sigma}}$
a.s. (and in probability) under $P_{0,\pi}^V$. Hence, by Lemma
\ref{lem:DowncrossingBackwardBPSameDistribution}, $\frac{1}{n} \sum_{k=0}^n
D_{n,k} \rightarrow \frac{\mu_{Q}}{\mu_{\sigma}}$, in probability under
$P_{0,\pi}$. Now, since the random walk $(X_n)$ is right transient with
$\delta > 1$, $\limsup_{n \to \infty} \sum_{k < 0} D_{n,k} < \infty$,
$P_{0,\pi}$ a.s. So, it follows from
\eqref{eq:RelationHittingTimesAndDownCrossings} that $\tau_n^X/n \rightarrow
1 + \frac{2\mu_Q}{\mu_\sigma}$ in probability under $P_{0,\pi}$. Since we
know a priori that $X_n/n \rightarrow v$, $P_{0,\pi}$ a.s. (for some unknown
$v \in [-1,1]$), this in fact implies $\tau_n^X/n \rightarrow 1 + \frac{2
\mu_Q}{\mu_{\sigma}}$, $P_{0,\pi}$ a.s., and
\begin{align}
\label{eq:VelocityFormula}
v = \frac{1}{1 + 2 \frac{\mu_Q}{\mu_{\sigma}}}.
\end{align}
Now, by Proposition \ref{prop:DecayOfHittingTimeSigma0V}, $\mu_{\sigma}$ is
finite for all $\delta > 1$, and by Proposition
\ref{prop:DecayOfAccumulatedSumTillTimeSigma0V}, $\mu_Q$ is finite for
$\delta > 2$ but infinite for $\delta \in (1,2]$. So, it follows that
\eqref{eq:SpeedEquationDeltaGreater1} holds.
\end{proof}

\subsection{Proof of Theorem \texorpdfstring{\ref{thm:LimitLaws}}{1.13}}
\label{subsec:ProofTheoremLimitLaws}

Throughout Section \ref{subsec:ProofTheoremLimitLaws} the random variables
$Z_{\alpha, b}$ are as in \eqref{eq:CharacteristicFunctionStableLaw}, the
random variables $Q_i$, $\sigma_i$, $\Delta_{\sigma,i}$ and $i_n$ are as in
Section \ref{subsec:ProofTheoremBallisticity}, and $m_n \equiv
\floor{n/\mu_{\sigma}}$. The general proof strategy for Theorem
\ref{thm:LimitLaws} will be to first prove a limiting distribution for
$\sum_{k=0}^n V_k$, then translate to a limiting distribution for the hitting
times $\tau_n^X$ using  \eqref{eq:RelationHittingTimesAndDownCrossings} and
Lemma \ref{lem:DowncrossingBackwardBPSameDistribution}, then translate to a
limiting distribution for the walk $(X_n)$ itself. This basic approach has
been used before in \cite{Basdevant2008b, Kosygina2011, Kosygina2017}, and
our methods will be quite similar to these works, but the details differ a
bit because the process $(V_k)$ is not Markovian when the environment is not
(IID). Thus, we must consider renewal times $(\sigma_i)_{i \geq 0}$, rather
than simply successive times at which $V_k = 0$. Also, we have the additional
minor complication of the $Q_0$ and $\sigma_0$ terms to deal with (which
would be 0 in the (IID) case).

Unless other specified it is assumed throughout that $V_0 = 0$ and the
probability measure on environments is $\PB_{\pi}$ (rather than $\PB$).
Everything will be proved initially under the stationary measure $\PB_{\pi}$,
and then at the very end after proving Theorem \ref{thm:LimitLaws} under the
stationary measure $\PB_{\pi}$ we will translate the result to the given
measure $\PB$.

To state our first lemma for the limiting distribution of $\sum_{k=0}^n V_k$
we first need to introduce a little notation. Let $\mu_Q(t)$ be the truncated
expectation of the random variables $Q_i$: $\mu_Q(t) \equiv E_{0,\pi}^V[Q_i
\cdot \indicator_{\{Q_i \leq t\}}]$, $i \geq 1$, where $E_{0,\pi}^V(\cdot)$
is expectation with respect to the probability measure $P_{0,\pi}^V$. Also,
let $Z_{1, b, c}$ be a random variable with characteristic function
\begin{align}
\label{eq:CharacteristicFunctionZ1bc}
E[ e^{itZ_{1,b,c}}] = \exp\Big[  i t c -b|t| \Big(1 + \frac{2i}{\pi} \log|t| \mathrm{sgn}(t) \Big) \Big],
\end{align}
for $b > 0$ and $c \in \R$. For future reference we note the following
scaling relations hold for the stable random variables $Z_{\alpha,b}$ and
$Z_{1,b,c}$:
\begin{align}
\label{ZalphabScaling}
& a Z_{\alpha, b} \stackrel{d.}{=} Z_{\alpha, ba^{\alpha}} ~,~\mbox{ for all $\alpha \in (0,1) \cup (1,2], b > 0, a > 0$}. \\
\label{Z1bcScaling}
& a_1 Z_{1,b,c} + a_2 \stackrel{d.}{=} Z_{1, ba_1, [c a_1 + a_2 - \frac{2}{\pi} b a_1 \log(a_1)]} ~,~\mbox{ for all $b > 0$, $c \in \R$, $a_1 > 0$, $a_2 \in \R$}.
\end{align}
Also, we note that $\mu_{\sigma}$ is finite for all $\delta > 1$ by
Propositions \ref{prop:DecayOfHittingTimeSigma0V}, and $\mu_Q$ is finite for
all $\delta > 2$ by Proposition
\ref{prop:DecayOfAccumulatedSumTillTimeSigma0V}.

\begin{lemma}
\label{eq:LemDistributionSumVk} Under the probability measure $P_{0,\pi}^V$
for the process $(V_k)_{k \geq 0}$ the following hold:
\begin{itemize}
\item[(i)] If $\delta \in (1,2)$ then there is some $b > 0$ such that
    $\dfrac{ \sum_{k=0}^n V_k }{n^{2/\delta}}
    \stackrel{d.}{\longrightarrow} Z_{\delta/2, b}$, as $n \rightarrow
    \infty$.
\item[(ii)] If $\delta = 2$ then there are constants $b > 0$ and $c \in \R$
    such that $\dfrac{ \sum_{k=0}^n V_k -
    \frac{\mu_Q(n/\mu_{\sigma})}{\mu_{\sigma}} n}{n}
    \stackrel{d.}{\longrightarrow} Z_{1, b, c}$, as $n \rightarrow \infty$.
\item[(iii)] If $\delta \in (2,4)$ then there is some $b > 0$ such that
    $\dfrac{ \sum_{k=0}^n V_k - \frac{\mu_Q}{\mu_{\sigma}} n}{n^{2/\delta}}
    \stackrel{d.}{\longrightarrow} Z_{\delta/2, b}$, as $n \rightarrow
    \infty$.
\item[(iv)] If $\delta =4$ then there is some $b > 0$ such that $\dfrac{
    \sum_{k=0}^n V_k - \frac{\mu_Q}{\mu_{\sigma}} n }{ [n \log(n)]^{1/2}}
    \stackrel{d.}{\longrightarrow} Z_{2, b}$, as $n \rightarrow \infty$.
\item[(v)] If $\delta >4$ then there is some $b > 0$ such that $\dfrac{
    \sum_{k=0}^n V_k - \frac{\mu_Q}{\mu_{\sigma}} n }{n^{1/2}}
    \stackrel{d.}{\longrightarrow} Z_{2, b}$, as $n \rightarrow \infty$.
\end{itemize}
\end{lemma}

To prove Lemma \ref{eq:LemDistributionSumVk} we will need two general results
about the limiting distributions of sums of i.i.d. random variables. The
first result is a particular case of convergence to stable distributions for
sums of i.i.d. random variables with regularly varying tails. The second
result concerns sums of a random number of i.i.d. random variables with
finite variance.

\begin{theorem}[Special case of Theorem 1, page 172 (for $\alpha \geq 2$) and Theorem 2, page 175 (for $\alpha < 2$) in \cite{Gnedenko1968}]
Let $Z$ be a random variable with distribution such that:
\label{the:ConvergenceInDistributionStableLaws}
\begin{align*}
& \P(Z > x) \sim C x^{-\alpha}, \mbox{ as $x \rightarrow \infty$, for some constants $C > 0$ and $\alpha > 0$}.  \\
& \P(Z < x_0) = 0, \mbox{ for some $x_0 \in (-\infty,0]$}.
\end{align*}
Also, let $Z_1, Z_2, \ldots$ be i.i.d. random variables distributed as $Z$.
Then the following hold:
\begin{itemize}
\item[(i)] If $\alpha \in (0,1)$, $\dfrac{ \sum_{k=1}^n Z_k}{n^{1/\alpha}}
    \stackrel{d.}{\longrightarrow} Z_{\alpha,b}$, for some $b > 0$.
\item[(ii)] If $\alpha = 1$, $\dfrac{ \sum_{k=1}^n Z_k - n \cdot \E[Z \cdot
    \indicator_{\{Z \leq n\}}]}{n} \stackrel{d.}{\longrightarrow}
    Z_{1,b,c}$, for some $b > 0$ and $c \in \R$.
\item[(iii)] If $\alpha \in (1,2)$, $\dfrac{ \sum_{k=1}^n Z_k - n \cdot
    \E(Z)}{n^{1/\alpha}} \stackrel{d.}{\longrightarrow} Z_{\alpha,b}$, for
    some $b > 0$.
\item[(iv)] If $\alpha = 2$, $\dfrac{ \sum_{k=1}^n Z_k - n \cdot \E(Z)}{[n
    \log(n)]^{1/2}} \stackrel{d.}{\longrightarrow} Z_{2,b}$, for some $b >
    0$.
\item[(v)] If $\alpha > 2$, $\dfrac{ \sum_{k=1}^n Z_k - n \cdot
    \E(Z)}{n^{1/2}} \stackrel{d.}{\longrightarrow} Z_{2,b}$, for some $b >
    0$.
\end{itemize}
\end{theorem}

\begin{theorem}[Theorem 3.1, page 17 in \cite{Gut2009}]
\label{the:ConvergenceInDistbitutionRandomNumberTermsIID} Let $(Z_k)_{k \geq
1}$ be i.i.d. random variables with $\E(Z_k) = 0$ and $\Var(Z_k) =
{\sigma}_Z^2 \in (0, \infty)$. Also, let $(N_n)_{n \geq 1}$ be a sequence of
random variables such that $N_n/n \rightarrow \theta$, in probability, for
some $\theta \in (0,\infty)$. Then $\big(\sum_{k=1}^{N_n}
Z_k\big)\big/\big(\sigma_Z \sqrt{n \theta }\big)
\stackrel{d.}{\longrightarrow} N(0,1)$, as $n \rightarrow \infty$.
\end{theorem}

In addition to these two theorems we will also need the following lemma for
the proof of part (v) of Lemma \ref{eq:LemDistributionSumVk}.

\begin{lemma}
\label{lem:RenewalTheoryLemmaForin} For any $\delta > 1$ the following hold:
\begin{itemize}
\item[(i)] $i_n/n \rightarrow 1/\mu_{\sigma}$ a.s. under $P_{0,\pi}^V$.

\item[(ii)] There exists some constant $c_{14} > 0$ such that:
\begin{align*}
& \lim_{n \to \infty} P_{0,\pi}^V(n - \sigma_{i_n} > k) \leq c_{14} \cdot k^{1 - \delta} ~,~\mbox{ for each $k \in \N$.} \\
& \lim_{n \to \infty} P_{0,\pi}^V(\sigma_{i_n + 1} - n > k) \leq c_{14} \cdot k^{1 - \delta} ~,~\mbox{ for each $k \in \N$.}
\end{align*}
\end{itemize}
\end{lemma}

\begin{proof}[Proof of Lemma \ref{lem:RenewalTheoryLemmaForin} (i)]
Since $\Delta_{\sigma,i}$, $i \geq 1$, are i.i.d., the times $(\sigma_i)_{i
\geq 0}$ are the renewal times for a (delayed) renewal process. By definition
$i_n$ is the number of renewals up to time $n$ (excluding $\sigma_0)$. So, by
\cite[Proposition 3.5.1]{Ross1996} $i_n/n \rightarrow 1/\mu_\sigma$ a.s.
\end{proof}

\begin{proof}[Proof of Lemma \ref{lem:RenewalTheoryLemmaForin} (ii)]
Since the Markov chain $(R_k)$ is uniformly ergodic it is aperiodic, which
implies the Markov chain on pairs $(V_k, R_k)$ is aperiodic (due to
ellipticity of the cookie stacks), which implies the discrete renewal process
with renewal times $(\sigma_i)_{i \geq 0} $ is itself aperiodic. Let
$A_{m,k}$ be the event that there are no renewal times $\sigma_i$ in $\{m,
m+1, \ldots, m + k - 1\}$ and let $\Delta_{\sigma}$ be a random variable (on
some probability space) with the common distribution of the random variables
$\Delta_{\sigma,i}$, $i \geq 1$. It follows from aperiodicity and
\cite[Example 4.3 (C)]{Ross1996} that $P_{0,\pi}^V(A_{m,k}) \rightarrow
\E[(\Delta_{\sigma}-k)^+]/\E[\Delta_{\sigma}]$ as $m \rightarrow \infty$, for
each fixed $k \in \N$. Moreover, by Proposition
\ref{prop:DecayOfHittingTimeSigma0V},
$\E[(\Delta_{\sigma}-k)^+]/\E[\Delta_{\sigma}] \leq c_{14} \cdot k^{1 -
\delta}$, for some $c_{14} > 0$ and all $k \in \N$. Thus,
\begin{align*}
\lim_{n \to \infty} P_{0,\pi}^V(n - \sigma_{i_n} > k)
& = \lim_{n \to \infty} P_{0,\pi}^V(\mbox{no renewals times in } \{n-k, n-k+1,\ldots, n\}) \\
& = \lim_{n \to \infty} P_{0,\pi}^V(A_{n-k, k+1}) \leq \lim_{n \to \infty} P_{0,\pi}^V(A_{n-k, k}) \leq c_{14} \cdot k^{1-\delta}.
\end{align*}
and
\begin{align*}
\lim_{n \to \infty} P_{0,\pi}^V(\sigma_{i_n + 1} - n > k)
& = \lim_{n \to \infty} P_{0,\pi}^V(\mbox{no renewals times in } \{n+1, \ldots, n+k \}) \\
& = \lim_{n \to \infty} P_{0,\pi}^V(A_{n+1, k}) \leq c_{14} \cdot k^{1-\delta}.\qedhere
\end{align*}
\end{proof}

\begin{proof}[Proof of Lemma \ref{eq:LemDistributionSumVk}]
The proofs of the various parts of the lemma will be given separately, but we
begin with one important general observation that is necessary for the proof
of several parts. If $1 \leq \ell \leq m_n$ and $|\sigma_{m_n} - n| \leq
\ell$, then
\begin{align*}
\left| \sum_{k=0}^n V_k - \sum_{i=1}^{m_n} Q_i \right|
~~\leq~~ Q_0 +\hspace{-2 mm} \sum_{k = \sigma_{m_n} - \ell}^{\sigma_{m_n} + \ell} V_k
~~\leq~~Q_0 + \hspace{-2 mm}\sum_{k = \sigma_{(m_n - \ell)}}^{\sigma_{(m_n + \ell)}} V_k
~~=~~ Q_0 + \hspace{-2 mm} \sum_{i = m_n - \ell+1}^{m_n + \ell} Q_i
\end{align*}
since the random variables $V_k$ are nonnegative and $\sigma_{m_n - \ell}
\leq \sigma_{m_n} - \ell < \sigma_{m_n} + \ell \leq \sigma_{m_n + \ell}$.
Thus, if $(a_n)$ and $(b_n)$ are any sequences of positive real numbers such
that $b_n \rightarrow \infty$ and $a_n \rightarrow \infty$ with $a_n = o(n)$,
then
\begin{align}
\label{eq:VkQiSumsDifference}
\limsup_{n \to \infty} ~& P_{0,\pi}^V\Big( \Big| \sum_{k=0}^n V_k - \sum_{i=1}^{m_n} Q_i \Big| > 2 b_n \Big) \nonumber \\
& \leq \limsup_{n \to \infty} \left[ P_{0,\pi}^V(|\sigma_{m_n} - n| > a_n) + P_{0,\pi}^V(Q_0 > b_n) + P_{0,\pi}^V\Big(\sum_{i = m_n - \floor{a_n}+1}^{m_n + \floor{a_n}} Q_i > b_n\Big) \right] \nonumber \\
& \leq \limsup_{n \to \infty} P_{0,\pi}^V(|\sigma_{m_n} - n| > a_n) + \limsup_{n \to \infty} P_{0,\pi}^V\Big(\sum_{i = 1}^{2 \floor{a_n}} Q_i > b_n\Big).
\end{align}
In the last line we have used the fact that the random variables $(Q_i)_{i \geq 1}$ are i.i.d., and also the fact that $Q_0$ is $P_{0,\pi}^V$ a.s. finite (due to Corollary \ref{cor:HittingTime0sIsFinite}), which implies $P_{0,\pi}^V(Q_0 > b_n) \rightarrow 0$, if $b_n \rightarrow \infty$. \\

\noindent
\emph{Proof of (i):} \\
With $\delta \in (1,2)$ it follows from Proposition
\ref{prop:DecayOfHittingTimeSigma0V} and Theorem
\ref{the:ConvergenceInDistributionStableLaws}-(iii) that $\frac{\sum_{i =
1}^m \Delta_{\sigma,i} - m \mu_{\sigma}}{m^{1/\delta}}
\stackrel{d.}{\longrightarrow} Z_{\delta,B_1}$, as $m \rightarrow \infty$,
for some $B_1 > 0$. Also, since $\sigma_0$ is a.s. finite $\frac{\sigma_0 +
(m_n \mu_{\sigma} - n)}{m_n^{1/\delta}} \rightarrow 0$ a.s, as $n \rightarrow
\infty$. Hence,
\begin{align}
\label{eq:sigmamnLimitDistParti}
\frac{\sigma_{m_n} - n}{m_n^{1/\delta}}
= \left[ \frac{\sigma_0 + (m_n \mu_{\sigma} - n)}{m_n^{1/\delta}} + \frac{\sum_{i = 1}^{m_n} \Delta_{\sigma,i} - m_n \mu_{\sigma}}{m_n^{1/\delta}} \right]
\stackrel{d.}{\longrightarrow} Z_{\delta,B_1}, \mbox{ as $n \rightarrow \infty$}.
\end{align}
Further, by Proposition \ref{prop:DecayOfAccumulatedSumTillTimeSigma0V} and
Theorem \ref{the:ConvergenceInDistributionStableLaws}-(i), there is some $B_2
> 0$ such that
\begin{align}
\label{eq:QiSumLimitDistParti}
\frac{\sum_{i=1}^m Q_i }{m^{2/\delta}} \stackrel{d.}{\longrightarrow} Z_{\delta/2, B_2}, \mbox{ as $m \rightarrow \infty$.}
\end{align}
Now,
\begin{align*}
\frac{\sum_{k=0}^n V_k}{n^{2/\delta}}
= \left( \frac{m_n^{2/\delta}}{n^{2/\delta}} \cdot \frac{\sum_{i=1}^{m_n} Q_i}{m_n^{2/\delta}} \right) + \frac{\sum_{k=0}^n V_k - \sum_{i=1}^{m_n} Q_i}{n^{2/\delta}}
\equiv (I) + (II).
\end{align*}
Since $\frac{m_n^{2/\delta}}{n^{2/\delta}} \rightarrow
(1/\mu_{\sigma})^{2/\delta}$ it follows from \eqref{eq:QiSumLimitDistParti}
and \eqref{ZalphabScaling} that $(I) \stackrel{d.}{\longrightarrow}
Z_{\delta/2, b}$, for some $b > 0$. So it will suffice to show that $(II)
\stackrel{p.}{\longrightarrow} 0$. To do this, fix an arbitrary $\epsilon >
0$. Applying \eqref{eq:VkQiSumsDifference} with $a_n = n^p$, $p \in
(1/\delta, 1)$, and $b_n = \epsilon n^{2/\delta}$ gives
\begin{align*}
\limsup_{n \to \infty}&~ P_{0,\pi}^V\Big( \Big| \sum_{k=0}^n V_k - \sum_{i=1}^{m_n} Q_i \Big| > 2 \epsilon n^{2/\delta} \Big) \\
& \leq \limsup_{n \to \infty} P_{0,\pi}^V(|\sigma_{m_n} - n| > n^p) + \limsup_{n \to \infty} P_{0,\pi}^V\Big(\sum_{i = 1}^{2 \floor{n^p}} Q_i > \epsilon n^{2/\delta}\Big).
\end{align*}
The first term on the right hand side is 0 by
\eqref{eq:sigmamnLimitDistParti}, and the second term on the right hand side
is 0 by \eqref{eq:QiSumLimitDistParti}.
Since $\epsilon > 0$ is arbitrary, it follows that $(II) \stackrel{p.}{\longrightarrow} 0$. \\

\noindent
\emph{Proof of (ii):} \\
With $\delta = 2$ it follows from Proposition
\ref{prop:DecayOfAccumulatedSumTillTimeSigma0V} that $P_{0,\pi}^V(Q_i > x)
\sim c_{13} x^{-1}$, for $i \geq 1$, which implies
\begin{align}
\label{eq:muQnScaling}
\mu_Q(t) \sim c_{13} \log(t), \mbox{ as $t \rightarrow \infty$} ~~\mbox{ and }~~ \lim_{n \to \infty} \left[ \mu_Q(n/\mu_{\sigma}) - \mu_Q(m_n) \right] = 0.
\end{align}
Also, similar arguments as in the proof of part (i) of the lemma using
Propositions \ref{prop:DecayOfHittingTimeSigma0V} and
\ref{prop:DecayOfAccumulatedSumTillTimeSigma0V} along with parts (ii) and
(iv) of Theorem \ref{the:ConvergenceInDistributionStableLaws} show that there
are constants $B_1, B_2 > 0$ and $C \in \R$ such that:
\begin{align}
\label{eq:sigmamnLimitDistPartii}
&\frac{\sigma_{m_n} - n}{\sqrt{m_n \log(m_n)}} \stackrel{d.}{\longrightarrow} Z_{2, B_1}, \mbox{ as $n \rightarrow \infty$.} \\
\label{eq:QiSumLimitDistPartii}
& \frac{\sum_{i=1}^{m} Q_i - m \mu_Q(m)}{m} \stackrel{d.}{\longrightarrow} Z_{1, B_2, C}, \mbox{ as $m \rightarrow \infty$.}
\end{align}
Now,
\begin{align*}
& \frac{\sum_{k=0}^n V_k - \frac{\mu_Q(n/\mu_{\sigma})}{\mu_{\sigma}} n}{n} \\
~~& = \left( \frac{m_n}{n} \cdot \frac{\sum_{i=1}^{m_n} Q_i - m_n \mu_Q(m_n)}{m_n} \right)
+ \frac{m_n \mu_Q(m_n) - \frac{n}{\mu_{\sigma}} \mu_Q(n/\mu_{\sigma})}{n}
+ \frac{\sum_{k=0}^n V_k - \sum_{i=1}^{m_n} Q_i }{n} \\
~~ & \equiv (I) + (II) + (III).
\end{align*}
Since $m_n/n \rightarrow 1/\mu_{\sigma}$ it follows from
\eqref{eq:QiSumLimitDistPartii} and \eqref{Z1bcScaling} that $(I)
\stackrel{d.}{\longrightarrow} Z_{1, b, c}$, for some $b > 0$ and $c \in \R$.
Also, by \eqref{eq:muQnScaling}, $(II) \rightarrow 0$ (deterministically).
So, it will suffice to show that $(III)  \stackrel{p.}{\longrightarrow} 0$.
To do this, fix an arbitrary $\epsilon > 0$. Applying
\eqref{eq:VkQiSumsDifference} with $a_n = n^p$, $p \in (1/2, 1)$, and $b_n =
\epsilon n$ gives
\begin{align*}
\limsup_{n \to \infty}~ &P_{0,\pi}^V \Big( \Big| \sum_{k=0}^n V_k - \sum_{i=1}^{m_n} Q_i \Big| > 2 \epsilon n \Big) \\
& \leq \limsup_{n \to \infty} P_{0,\pi}^V(|\sigma_{m_n} - n| > n^p) + \limsup_{n \to \infty} P_{0,\pi}^V\Big(\sum_{i = 1}^{2 \floor{n^p}} Q_i > \epsilon n\Big).
\end{align*}
The first term on the right hand side is 0 by \eqref{eq:sigmamnLimitDistPartii}, and the second term on the right hand side is 0 by \eqref{eq:muQnScaling} and \eqref{eq:QiSumLimitDistPartii}. Since $\epsilon > 0$ is arbitrary, it follows that $(III) \stackrel{p.}{\longrightarrow} 0$. \\

\noindent
\emph{Proof of (iii) and (iv):} \\
Define $d(n) = n^{2/\delta}$ if $\delta \in (2,4)$ and $d(n) = [n
\log(n)]^{1/2}$ if $\delta = 4$. We wish to show that
\begin{align*}
\frac{\sum_{k=0}^n V_k - \frac{\mu_Q}{\mu_{\sigma}} n}{d(n)} \stackrel{d.}{\longrightarrow} Z_{\delta/2, b} ~,~\mbox{ for some $b > 0$}.
\end{align*}
Similar arguments as in the proof of part (i) of the lemma using Propositions
\ref{prop:DecayOfHittingTimeSigma0V} and
\ref{prop:DecayOfAccumulatedSumTillTimeSigma0V} along with parts (iii), (iv),
and (v) of Theorem \ref{the:ConvergenceInDistributionStableLaws} show that
there are constants $B_1, B_2 > 0$ such that:
\begin{align}
\label{eq:sigmamnLimitDistPartiiiAndiv}
&\frac{\sigma_{m_n} - n}{m_n^{1/2}} \stackrel{d.}{\longrightarrow} Z_{2, B_1}, \mbox{ as $n \rightarrow \infty$.} \\
\label{eq:QiSumLimitDistPartiiiAndiv}
& \frac{\sum_{i=1}^{m} Q_i - m \mu_Q}{d(m)} \stackrel{d.}{\longrightarrow} Z_{\delta/2, B_2}, \mbox{ as $m \rightarrow \infty$.}
\end{align}
Now,
\begin{align*}
\frac{\sum_{k=0}^n V_k - \frac{\mu_Q}{\mu_{\sigma}} n}{d(n)}
& = \left( \frac{d(m_n)}{d(n)} \cdot \frac{ \sum_{i=1}^{m_n} Q_i - m_n \mu_Q}{d(m_n)} \right)
+ \frac{\mu_Q(m_n - n/\mu_{\sigma})}{d(n)}
+ \frac{\sum_{k=0}^n V_k - \sum_{i=1}^{m_n} Q_i }{d(n)} \\
& \equiv (I) + (II) + (III).
\end{align*}
Since $d(m_n)/d(n) \rightarrow (1/\mu_{\sigma})^{2/\delta}$ it follows from
\eqref{eq:QiSumLimitDistPartiiiAndiv} and \eqref{ZalphabScaling} that $(I)
\stackrel{d.}{\longrightarrow} Z_{\delta/2, b}$, for some $b > 0$. Also,
$(II) \rightarrow 0$ (deterministically). So, it will suffice to show that
$(III) \stackrel{p.}{\longrightarrow} 0$. To do this, fix an arbitrary
$\epsilon > 0$. Applying \eqref{eq:VkQiSumsDifference} with $a_n = n^{1/2}
(\log n)^{1/4}$ and $b_n = \epsilon d(n)$ gives
\begin{align*}
& \limsup_{n \to \infty} P_{0,\pi}^V\Big( \Big| \sum_{k=0}^n V_k - \sum_{i=1}^{m_n} Q_i \Big| > 2 \epsilon d(n) \Big) \\
&~~ \leq \limsup_{n \to \infty} P_{0,\pi}^V\Big(|\sigma_{m_n} - n| > n^{1/2} (\log n)^{1/4} \Big)
+ \limsup_{n \to \infty} P_{0,\pi}^V\Big(\sum_{i = 1}^{2 \floor{n^{1/2} (\log n)^{1/4}}} Q_i > \epsilon d(n) \Big).
\end{align*}
The first term on the right hand side is 0 by
\eqref{eq:sigmamnLimitDistPartiiiAndiv}, and the second term on the right
hand side is 0 by \eqref{eq:QiSumLimitDistPartiiiAndiv}.
Since $\epsilon > 0$ is arbitrary, it follows that $(III) \stackrel{p.}{\longrightarrow} 0$. \\

\noindent
\emph{Proof of (v):} \\
Let $Z_i = Q_i - \frac{\mu_Q}{\mu_{\sigma}} \Delta_{\sigma,i}$, $i \geq 1$.
Note that with $\delta > 4$ the random variables $(Z_i)_{i \geq 1}$ are
i.i.d. under $P_{0,\pi}^V$ with mean $0$ and finite variance, due to
Propositions \ref{prop:DecayOfHittingTimeSigma0V} and
\ref{prop:DecayOfAccumulatedSumTillTimeSigma0V}. By \eqref{eq:inVersusn},
\begin{align*}
\frac{ \sum_{k=0}^n V_k - \frac{\mu_Q}{\mu_{\sigma}} n}{n^{1/2}}
\geq \frac{ \sum_{i = 1}^{i_n} Z_i}{n^{1/2}} + \frac{ \frac{\mu_Q}{\mu_{\sigma}}(\sigma_{i_n} - \sigma_0 - n)}{n^{1/2}}
\equiv (I) + (II)
\end{align*}
and
\begin{align*}
\frac{ \sum_{k=0}^n V_k - \frac{\mu_Q}{\mu_{\sigma}} n}{n^{1/2}}
\leq \frac{ \sum_{i = 1}^{i_n+1} Z_i}{n^{1/2}} + \frac{ \frac{\mu_Q}{\mu_{\sigma}}(\sigma_{i_n+1} - \sigma_0 - n) + Q_0}{n^{1/2}}
\equiv (I') + (II').
\end{align*}
By Lemma \ref{lem:RenewalTheoryLemmaForin}-(i), $i_n/n
\stackrel{p.}{\longrightarrow} 1/\mu_{\sigma}$. So, it follows from Theorem
\ref{the:ConvergenceInDistbitutionRandomNumberTermsIID} that $(I)
\stackrel{d.}{\longrightarrow} Z_{2,b}$ and $(I')
\stackrel{d.}{\longrightarrow} Z_{2,b}$ where $b \equiv \Var(Z_i)/(2
\mu_{\sigma})$. Furthermore, by Corollary \ref{cor:HittingTime0sIsFinite} and
Lemma \ref{lem:RenewalTheoryLemmaForin}-(ii), $(II)
\stackrel{p.}{\longrightarrow} 0$ and $(II') \stackrel{p.}{\longrightarrow}
0$. This completes the proof.
\end{proof}

The next lemma gives the limiting distribution of the hitting times
$\tau_n^X$.

\begin{lemma}
\label{eq:LemDistributionTaunX} Let $v = 1/(1 + 2\frac{\mu_Q}{\mu_{\sigma}})$
be the velocity of the ERW $(X_n)_{n \geq 0}$ from
\eqref{eq:VelocityFormula}. Then under the probability measure $P_{0,\pi}$
for the ERW $(X_n)$ the following hold:
\begin{itemize}
\item[(i)] If $\delta \in (1,2)$ then there is some $b > 0$ such that
    $\dfrac{\tau_n^X}{n^{2/\delta}} \stackrel{d.}{\longrightarrow}
    Z_{\delta/2,b}$, as $n \rightarrow \infty$.
\item[(ii)] If $\delta = 2$ then there are constants $b > 0$ and $c \in \R$
    such that $\dfrac{\tau_n^X - \left[1 + \frac{2
    \mu_Q(n/\mu_{\sigma})}{\mu_{\sigma}} \right]n}{n}
    \stackrel{d.}{\longrightarrow} Z_{1,b,c}$, as $n \rightarrow \infty$.
\item[(iii)] If $\delta \in (2,4)$ then there is some $b > 0$ such that
    $\dfrac{\tau_n^X - \frac{1}{v} n}{n^{2/\delta}}
    \stackrel{d.}{\longrightarrow} Z_{\delta/2, b}$, as $n \rightarrow
    \infty$.
\item[(iv)] If $\delta =4$ then there is some $b > 0$ such that
    $\dfrac{\tau_n^X - \frac{1}{v} n}{[n \log(n)]^{1/2}}
    \stackrel{d.}{\longrightarrow} Z_{2, b}$, as $n \rightarrow \infty$.
\item[(v)] If $\delta >4$ then there is some $b > 0$ such that $\dfrac{
    \tau_n^X - \frac{1}{v} n}{n^{1/2}} \stackrel{d.}{\longrightarrow} Z_{2,
    b}$, as $n \rightarrow \infty$.
\end{itemize}
\end{lemma}

\begin{proof}
Since $(X_n)$ is $P_{0,\pi}$ a.s. right transient with any $\delta > 1$,
$\limsup_{n \to \infty} \sum_{k < 0} D_{n,k}$ is $P_{0,\pi}$ a.s. finite. So,
$\left(\sum_{k < 0} D_{n,k}\right)/n^{\alpha} \stackrel{p.}{\longrightarrow}
0$, for any $\alpha > 0$. Thus, parts (i)-(v) of this lemma follow directly
from (i)-(v) of Lemma \ref{eq:LemDistributionSumVk} using
\eqref{eq:RelationHittingTimesAndDownCrossings} and Lemma
\ref{lem:DowncrossingBackwardBPSameDistribution}, along with
\eqref{Z1bcScaling} for part (ii) and \eqref{ZalphabScaling} for the other
parts. (Note that the constants $b$, $c$ are modified from Lemma
\ref{eq:LemDistributionSumVk}.)
\end{proof}

The proof of Theorem \ref{thm:LimitLaws} below will be based on Lemma
\ref{eq:LemDistributionTaunX}, but first we will need one final lemma about
backtracking probabilities. The same statement has been given in \cite[Lemma
6.1]{Peterson2012} for the case of (IID) environments, and our proof will be
very similar, but must be adjusted slightly to use the renewal times
$(\sigma_i)_{i \geq 0}$, instead of the successive times at which $V_k = 0$.

\begin{lemma}
\label{lem:BacktrackingProbabilities} Assume that $\delta > 1$ and let
$c_{14}$ be as in Lemma \ref{lem:RenewalTheoryLemmaForin}. Then
\begin{align*}
P_{0,\pi}\Big( \inf_{m \geq \tau_{n + k}^X} X_m \leq n \Big) \leq c_{14} \cdot k^{1 - \delta} ~,~ \mbox{ for all $k, n \in \N$.}
\end{align*}
In particular, $P_{0,\pi}\Big( \inf\limits_{m \geq \tau_{n + k}^X} X_m \leq n
\Big) \rightarrow 0$, as $k \rightarrow \infty$, uniformly in $n$.
\end{lemma}

\begin{proof}
As in the proof of Lemma \ref{lem:RenewalTheoryLemmaForin} we consider the
renewal times $\sigma_i$ and note that, for each fixed $k$,
\begin{align*}
\lim_{N \to \infty} P_{0,\pi}^V\Big( \mbox{no renewals in times } \{N, N + 1, \ldots, N + k - 1\} \Big) \leq c_{14} \cdot k^{1-\delta}.
\end{align*}
Using this along with Lemma \ref{lem:DowncrossingBackwardBPSameDistribution}
gives
\begin{align*}
& P_{0,\pi}\Big( \inf_{m \geq \tau_{n + k}^X} X_m \leq n \Big)
= \lim_{N \to \infty} P_{0,\pi}\Big( \inf_{\tau_{n + k}^X \leq m < \tau_N^X} X_m \leq n \Big) \\
& \leq \lim_{N \to \infty} P_{0,\pi}\Big( D_{N,j} \geq 1, \forall j \in \{n+1, n+2, \ldots, n+k\} \Big) \\
& = \lim_{N \to \infty} P_{0,\pi}^V\Big(V_j \geq 1, \forall j \in \{N - n - k, \ldots, N - n - 1\}\Big) \\
& \leq \lim_{N \to \infty} P_{0,\pi}^V\Big(\mbox{no renewals in times } \{N - n - k, \ldots, N - n - 1\}\Big) \\
& \leq c_{14} \cdot k^{1-\delta}.\qedhere
\end{align*}
\end{proof}

\begin{proof}[Proof of Theorem \ref{thm:LimitLaws}]
Recall that for the probability measure $\PB$ on environments, the marginal
distribution of $S_0$ is $\phi$.
We will first prove the theorem in the case that $\phi = \pi$ is the stationary distribution. We will then extend to the case of general $\phi$. \\

\noindent
\underline{Case 1}: $\phi = \pi$.\\
Let $X_n^+ = \sup_{i \leq n} X_i$ and $X_n^- = \inf_{i \geq n} X_i$. Since
$X_n^- \leq X_n \leq X_n^+$, for all $n$, it will suffice to show that
\begin{align*}
& \mbox{(a) Parts (i)-(v) of the theorem all hold with $X_n$ replaced with $X_n^+$, and } \\
& \mbox{(b) Parts (i)-(v) of the theorem all hold with $X_n$ replaced with $X_n^-$. }
\end{align*}
Now observe that, for any $n,m,k \in \N$,
\begin{align}
\label{eq:XnPlusLessThanm}
\{X_n^+ < m\} = \{\tau_m^X > n\}
\end{align}
and
\begin{align}
\label{eq:XnMinusLessThanm}
\{X_n^+ < m\} \subset \{X_n^- < m\} \subset \Big( \{X_n^+ < m + k\} \cup \big\{ \inf_{i \geq \tau_{m+k}^X} X_i < m\big\} \Big).
\end{align}

Standard computations using \eqref{eq:XnPlusLessThanm} and Lemma
\ref{eq:LemDistributionTaunX} give (a) for parts (i), (iii), (iv), and (v) of
the theorem (with a modified value of the constant $b$ in parts (iii)-(v)).
The proof of (a) for part (ii) of the theorem is a bit more subtle and will
be given below, following closely the method in \cite[Appendix
B]{Kosygina2017}. The second statement (b) follows from (a) using
\eqref{eq:XnMinusLessThanm} and Lemma \ref{lem:BacktrackingProbabilities}.

We proceed now to the proof of (a) for part (ii) of the theorem. Assume
$\delta = 2$ and define $D(t) = c + 1 +  \frac{2
\mu_Q(t/\mu_{\sigma})}{\mu_{\sigma}}$ and $a = \mu_{\sigma}/(2c_{13})$, where
$c$ is the same constant from Lemma \ref{eq:LemDistributionTaunX}-(ii). Then,
by Lemma \ref{eq:LemDistributionTaunX}-(ii) and \eqref{Z1bcScaling},
\begin{align}
\label{eq:LimitTaunDelta2withDn}
\frac{\tau_n^X - D(n) n}{n} \stackrel{d.}{\longrightarrow} Z_{1,b}, \mbox { for some $b > 0$}.
\end{align}
Also, since $P_{0,\pi}^V(Q_i > x) \sim c_{13} x^{-1}$, $i \geq 1$, it follows
from the definition of $\mu_Q(t)$ that
\begin{align}
\label{eq:AsymptoticsDt}
D(t) \sim \frac{1}{a} \log(t), \mbox{ as $t \rightarrow \infty$} ~~\mbox{ and }~~ \lim_{n \to \infty} D(k_n) - D(n) = 0, \mbox{ if $k_n \sim n$}.
\end{align}
For $t > 0$, let $\Gamma(t) = \inf\{s > 0: sD(s) \geq t\}$. Note that $D(t)
\sim \frac{1}{a} \log(t) \implies \Gamma(t) \sim a t /\log(t)$. Further we
claim that
\begin{align}
\label{eq:AsymptoticsGammatDt}
\Gamma(t) D(\Gamma(t)) = t + o(\Gamma(t)), \mbox{ as $t \rightarrow \infty$}.
\end{align}
To see this note that the function $g(s) = s D(s)$ is right continuous and
strictly increasing for all sufficiently large $s$. Moreover, jump
discontinuities in this function $g(s)$ can occur only at $s = k
\mu_{\sigma}$ for integer $k$, and at such an $s$ the size of the jump
discontinuity is $s [\frac{2}{\mu_{\sigma}} \cdot \frac{s}{\mu_{\sigma}}
P_{0,\pi}^V(Q_i = \frac{s}{\mu_{\sigma}})] = 2(\frac{s}{\mu_{\sigma}})^2
P_{0,\pi}^V(Q_i = \frac{s}{\mu_{\sigma}})$. It follows from these
observations and the definition of $\Gamma(t)$ that
\begin{align}
\label{eq:GammaDGammatBound}
|\Gamma(t) D(\Gamma(t)) - t| \leq 2(\Gamma(t)/\mu_{\sigma})^2 P_{0,\pi}^V(Q_i = \Gamma(t)/\mu_{\sigma}),
\end{align}
for all sufficiently large $t$. Now, since $P_{0,\pi}^V(Q_i > x) \sim c_{13}
x^{-1}$, we have $x P_{0,\pi}^V(Q_i  = x) \rightarrow 0$, as $x \rightarrow
\infty$. So, the right hand side of \eqref{eq:GammaDGammatBound} is
$o(\Gamma(t))$, which proves \eqref{eq:AsymptoticsGammatDt}.

Now, for $x \in \R$ and $n \in \N$, let $k_{n,x} = \max\{\lceil{\Gamma(n) +
\frac{xn}{(\log n)^2}\rceil}, 0\}$. Note that $k_{n,x} \sim \Gamma(n)$ as $n
\rightarrow \infty$, since $\Gamma(n) \sim an/\log(n)$. Using this fact along
with \eqref{eq:AsymptoticsDt} and \eqref{eq:AsymptoticsGammatDt} and, again,
the tail asymptotics for $\Gamma(n)$ shows that, for any fixed $x$,
\begin{align}
\label{eq:negxa2Limit}
\lim_{n \to \infty} \frac{n - k_{n,x} D(k_{n,x})}{k_{n,x}}
= \lim_{n \to \infty} \frac{n - \left[\Gamma(n) + \frac{xn}{(\log n)^2}\right] D(\Gamma(n))}{\Gamma(n) + \frac{xn}{(\log n)^2}}
= \frac{-x}{a^2}.
\end{align}
Further, since $X_n^+$ takes only integer values it follows from
\eqref{eq:XnPlusLessThanm} that, for all sufficiently large $n$,
\begin{align*}
P_{0,\pi}^V \left(\frac{X_n^+ - \Gamma(n)}{n/(\log n)^2} < x \right)
= P_{0,\pi}^V\left(\frac{\tau^X_{k_{n,x}} - k_{n,x} D(k_{n,x})}{k_{n,x}} > \frac{n - k_{n,x} D(k_{n,x})}{k_{n,x}} \right).
\end{align*}
Taking the limit of both sides as $n \rightarrow \infty$ and using
\eqref{eq:LimitTaunDelta2withDn} and \eqref{eq:negxa2Limit} shows that
$\lim_{n \to \infty} P_{0,\pi}^V \left(\frac{X_n^+ - \Gamma(n)}{n/(\log n)^2}
< x \right) = \P(Z_{1,b} > -x/a^2)$, which implies
$\frac{X_n^+ - \Gamma(n)}{a^2n/(\log n)^2} \stackrel{d.}{\longrightarrow} - Z_{1,b}$. \\

\noindent
\underline{Case 2}: General $\phi$. \\
We will extend from Case 1 using a coupling argument. Let $(S_k^{\phi})_{k
\in \Z}$ and $(S_k^{\pi})_{k \in \Z}$ denote the stack sequences when $S_0$
has marginal distribution $\phi$ and $\pi$ respectively. Couple these
processes as follows. First sample $(S_k^{\phi})_{k \leq 0}$ and
$(S_k^{\pi})_{k \leq 0}$ independently. Then run the Markov chains
$(S_k^{\phi})$ and $(S_k^{\pi})$ forward in time independently, starting from
the given values of $S_0^{\phi}$ and $S_0^{\pi}$, until the first time $N >
0$ such that $S_N^{\phi} = S_N^{\pi}$ (due to the uniform ergodicity
hypothesis $N$ is a.s. finite). After the chains first meet at time $N$, run
them forward together, so that $S_k^{\phi} = S_k^{\pi}$, for all $k \geq N$.

Now, let $\omega^{\phi}$ and $\omega^{\pi}$ be the corresponding environments
given by $\omega^{\phi}(k,i) = S_k^{\phi}(i)$ and $\omega^{\pi}(k,i) =
S_k^{\pi}(i)$, as in \eqref{eq:DefineEnvironmentBySk}, and couple the random
walks $(X_n^{\phi})$ and $(X_n^{\pi})$ in these environments as follows:
\begin{itemize}
\item Let $(\theta_{k,i})_{k \in \Z, i \in \N}$ be i.i.d. Uniform([0,1])
    random variables.
\item Set $X_0^{\phi} = X_0^{\pi} = 0$. Then define inductively:
\begin{align*}
X_{n+1}^{\phi} &= \begin{cases}
X_{n}^{\phi} + 1 ~,~ \mbox{ if } \theta_{X_{n}^{\phi}, I_n^{\phi}} < \omega^{\phi}(X_{n}^{\phi}, I_n^{\phi}) \\
X_{n}^{\phi} - 1 ~,~ \mbox{ else }
\end{cases}
\mbox{ and } \\
X_{n+1}^{\pi} &= \begin{cases}
X_{n}^{\pi} + 1 ~,~ \mbox{ if } \theta_{X_{n}^{\pi}, I_n^{\pi}} < \omega^{\pi}(X_{n}^{\pi}, I_n^{\pi}) \\
X_{n}^{\pi} - 1 ~,~ \mbox{ else }
\end{cases}
\end{align*}
where $I_n^{\phi} = |\{0 \leq m \leq n: X_m^{\phi} = X_n^{\phi}\}|$ and
$I_n^{\pi} = |\{0 \leq m \leq n: X_m^{\pi} = X_n^{\pi}\}|$.
\end{itemize}
In words, the walk $(X_n^{\phi})$ jumps right on its $i$-th visit to site $k$
if $\theta_{k,i} < \omega^{\phi}(k,i)$, and left otherwise. Similarly, the
walk $(X_n^{\pi})$ jumps right on its $i$-th visit to site $k$ if
$\theta_{k,i} < \omega^{\pi}(k,i)$, and left otherwise.

With this construction both walks $(X_n^{\pi})$ and $(X_n^{\phi})$ have the
correct averaged laws. Moreover, due to the coupling between environments
$\omega^{\pi}(k,i) = \omega^{\phi}(k,i)$, for all $k \geq N$ and $i \in \N$.
So, both walks $(X_n^{\pi})$ and $(X_n^{\phi})$ have the same theoretical
``jump sequence'' at each site $k \geq N$. That is, both walks will jump the
same direction from any site $k \geq N$ on their $i$-th visit to that site,
if such an $i$-th visit occurs.

Now, since it is assumed that $\delta > 1$ in all cases of the theorem, we
know both walks are right transient. So, with probability 1, each walk
eventually reaches site $N$ and also returns to this site after any leftward
excursion from it. Combining this with the previous observation about
matching jump sequences at all sites $k \geq N$ shows that
\begin{align*}
\limsup_{n \to \infty} |X_n^{\phi} - X_n^{\pi}| < \infty \mbox{ a.s. }
\end{align*}
Since \eqref{eq:LimitLawDeltaIn12}-\eqref{eq:LimitLawDeltaGreater4} hold with
$X_n = X_n^{\pi}$, by Case 1, it follows that
\eqref{eq:LimitLawDeltaIn12}-\eqref{eq:LimitLawDeltaGreater4} also hold with
$X_n = X_n^{\phi}$.
\end{proof}

\appendix

\section{Proof of Lemma \texorpdfstring{\ref{lem:FiniteModificationArgument}}{2.2}}
\label{app:ProofLemmaFiniteModificationArgument}

\begin{proof}[Proof of Lemma \ref{lem:FiniteModificationArgument}]
Fix $x \in \N$. Clearly, $P_{x,s}(A^+) = 0$ if $P_{x,s}(X_n \rightarrow +
\infty) = 0$, so we need only show that $P_{x,s}(A^+) > 0$ if $P_{x,s}(X_n
\rightarrow + \infty) > 0$. By definition, $P_{x,s}(A^+) =
\EB_s[P_x^{\omega}(A^+)]$ and $P_{x,s}(X_n \rightarrow + \infty) =  \EB_s[
P_x^{\omega}(X_n \rightarrow + \infty) ]$, so
it will suffice to show the following claim. \\

\noindent \emph{Claim:} Let $\omega \in \Omega$ be any environment satisfying
$\omega(k,i) \in (0,1)$, for all $k \in \Z$ and $i \in \N$. Then
$P_x^{\omega}(A^+) > 0$ if $P_x^{\omega}(X_n \rightarrow + \infty) > 0$. \\

\noindent \emph{Proof of Claim:} For $m \in \N$ and any nearest neighbor path
$\zeta = (x_0, \ldots, x_m) \in \Z^{m+1}$ with $x_0 = x$, define the events
$A_m$ and $A_{\zeta}$ by
\begin{align*}
A_m = \{X_n > x, \forall n \geq m ~\mbox{ and }~ X_n \rightarrow +\infty\} ~~\mbox{ and }~~A_{\zeta} = A_m \cap \{X_0 = x_0, \ldots, X_m = x_m\}.
\end{align*}
If $P_x^{\omega}(X_n \rightarrow + \infty) > 0$, then there must be some
finite path $\zeta = (x_0,\ldots,x_m)$ such that $x_ 0 = x$, $x_m = x + 1$,
and $P_x^{\omega}(A_{\zeta}) > 0$. This, of course, implies
$P_x^{\omega}(A_m|X_0 = x_0, \ldots, X_m = x_m) > 0$.

We construct from $\zeta = (x_0,\ldots,x_m)$ the reduced path $\zetat =
(\xt_0,\ldots,\xt_{\mt})$ by setting $\xt_0 = x_0 = x$, and then removing
from the tail $(x_1,\ldots,x_m)$ all steps in any leftward excursions from
site $x$. For example,
\begin{align*}
\mbox{ if } \zeta & = (x, \textbf{x}-\textbf{1}, \textbf{x}-\textbf{2}, \textbf{x}-\textbf{1}, \textbf{x}, x+1, x, \textbf{x}-\textbf{1}, \textbf{x}, \textbf{x}-\textbf{1}, \textbf{x}, x+1, x+2, x+3, x+2, x+1) \\
\mbox{ then } \zetat & = (x, x+1, x, x+1, x+2, x+3, x+2, x+1)
\end{align*}
(where we denote the removed steps in bold for visual clarity). By
construction, $\xt_{\mt} = x_m = x+1$ and the number of visits to each site
$k \geq x+1$ up to time $m$ if $(X_0, \ldots, X_m) = \zeta$ is exactly the
same as the number of visits to each site $k \geq x+1$ up to time $\mt$ if
$(X_0, \ldots, X_{\mt}) =  \zetat$. Thus,
\begin{align*}
P_x^{\omega}(A_{\mt}|X_0 = \xt_0, \ldots, X_{\mt} = \xt_{\mt}) = P_x^{\omega}(A_m|X_0 = x_0, \ldots, X_m = x_m) > 0,
\end{align*}
which implies
\begin{align*}
P_x^{\omega}(A^+)
& \geq P_x^{\omega} (X_0 = \xt_0,\ldots,X_{\mt} = \xt_{\mt}) \cdot P_x^{\omega}(A^+ |X_0 = \xt_0,\ldots,X_{\mt} = \xt_{\mt}) \\
& \geq P_x^{\omega} (X_0 = \xt_0,\ldots,X_{\mt} = \xt_{\mt}) \cdot P_x^{\omega}(A_{\mt} |X_0 = \xt_0,\ldots,X_{\mt} = \xt_{\mt}) > 0.
\end{align*}
(Note that $P_x^{\omega} (X_0 = \xt_0,\ldots,X_{\mt} = \xt_{\mt}) > 0$, since
we assume $\omega(k,i) \in (0,1)$, $\forall k \in \Z$, $i \in \N$.)
\end{proof}

\section{Proof of Lemma \texorpdfstring{\ref{lem:TransienceRecurrenceConditionsForMC}}{3.1}}
\label{app:ProofLemmaTransienceRecurrenceConditionsForMC}

The proof of Lemma \ref{lem:TransienceRecurrenceConditionsForMC} is based on
the following much more general, but less explicit, condition for transience
vs. recurrence of Markov chains on the nonnegative integers given in
\cite{Kozma2016}.

\begin{lemma}[Theorem A.1 of \cite{Kozma2016}]
\label{lem:GeneralTransienceRecurrenceConditionsMC} Let $(Z_k)_{k \geq 0}$ be
an irreducible, time-homogenous Markov chain on state space $\N_0$. Let
$\P_x(\cdot)$ be the probability measure for the Markov chain started from
$Z_0 = x$, and let $\E_x(\cdot)$ be the corresponding expectation operator.
\begin{itemize}
\item[(i)] If there exists a function $F: \N_0 \rightarrow (0,\infty)$ such
    that $\lim_{x \to \infty} F(x) = \infty$ and $\E_x[F(Z_1)] \leq F(x)$
    for all sufficiently large $x$, then the Markov chain $(Z_k)$ is
    recurrent.
\item[(ii)] If there exists a function $F: \N_0 \rightarrow (0,\infty)$
    such that $\lim_{x \to \infty} F(x) = 0$ and $\E_x[F(Z_1)] \leq F(x)$
    for all sufficiently large $x$, then the Markov chain $(Z_k)$ is
    transient.
\end{itemize}
\end{lemma}

The function $F(x)$ is called a Lyapunov function. Our proof of Lemma
\ref{lem:TransienceRecurrenceConditionsForMC} using Lemma
\ref{lem:GeneralTransienceRecurrenceConditionsMC} follows closely the proof
of Theorem 1.3 in \cite{Kozma2016}. In particular, we will use the same
choice of Lyapunov functions $F(x)$ and Taylor expand in the same fashion.
However, controlling the error term in the Taylor expansion becomes somewhat
more involved, because of the weaker concentration condition we assume for
the transition probabilities.

\begin{proof}[Proof of Lemma \ref{lem:TransienceRecurrenceConditionsForMC} (i)]
Let $F(x) : [0, \infty) \rightarrow (0, \infty)$ be a smooth function such
that $F(x) = \ln \ln (x)$, for $x > 3$. Then for all $x > 3$
\begin{align}
\label{eq:Derivativesloglogx}
& F'(x) = \frac{1}{x \ln(x)}, \nonumber \\
& F''(x) = - \frac{1}{x^2 \ln(x)} - \frac{1}{x^2 \ln^2(x)}, \nonumber \\
& F'''(x) = \frac{2}{x^3 \ln(x)} + \frac{3}{x^3 \ln^2(x)} + \frac{2}{x^3 \ln^3(x)}.
\end{align}
By Taylor's Theorem with remainder
\begin{align*}
F(Z_1) = F(x) + F'(x)(Z_1 - x) + \frac{1}{2} F''(x)(Z_1 - x)^2 + \frac{1}{6}F'''(\xi)(Z_1 - x)^3, \mbox{ for each } x \in \N,
\end{align*}
where $\xi$ is some (random, depending on $Z_1$) number between $Z_1$ and
$x$. Thus, for all positive integer $x > 3$,
\begin{align*}
&\E_x[F(Z_1)] \\
& = F(x) + \frac{\E_x[Z_1 - x]}{x \ln(x)} + \frac{1}{2} \left[- \frac{1}{x^2 \ln(x)} - \frac{1}{x^2 \ln^2(x)} \right]  \E_x[(Z_1 - x)^2] ~+~ \frac{1}{6} \E_x[F'''(\xi) (Z_1 - x)^3] \\
& = F(x) + \frac{1}{x \ln(x)} \left[ \rho(x) - \frac{1}{2}\left(1 + \frac{1}{\ln(x)}\right)\nu(x) \right] ~+~ \frac{1}{6} \E_x[F'''(\xi) (Z_1 - x)^3].
\end{align*}
So,
\begin{align*}
\E_x[F(Z_1)] \leq F(x)
& ~\iff~ \frac{1}{x \ln(x)} \left[ \rho(x) - \frac{1}{2}\left(1 + \frac{1}{\ln(x)}\right)\nu(x) \right] + \frac{1}{6} \E_x[F'''(\xi) (Z_1 - x)^3] \leq 0 \\
& ~\iff~ \theta(x) \leq 1 + \frac{1}{\ln(x)} - \frac{1}{3} \frac{x \ln(x)}{\nu(x)} \E_x[F'''(\xi) (Z_1 - x)^3].
\end{align*}
Therefore, in light of Lemma
\ref{lem:GeneralTransienceRecurrenceConditionsMC} and the assumption that
$\liminf_{x \to \infty} \nu(x) > 0$, it will suffice to show
\begin{align}
\label{eq:BoundOnTaylorErrorTermWanted1}
\E_x[F'''(\xi) (Z_1 - x)^3] = o\Big(\frac{1}{x \ln^2(x)}\Big).
\end{align}
Now,
\begin{align}
\label{eq:Term123Breakdown}
& \big| \E_x[F'''(\xi) (Z_1- x)^3] \big| \leq \E_x \left| F'''(\xi) (Z_1 - x)^3 \right| \nonumber \\
& = \begin{cases} \E_x \left| F'''(\xi) (Z_1 - x)^3 \indicator\{Z_1 \leq x/2\} \right| + \E_x \left| F'''(\xi) (Z_1 - x)^3 \indicator\{x/2 < Z_1 \leq 2x\} \right| \\
       +~ \E_x \left| F'''(\xi) (Z_1 - x)^3 \indicator\{Z_1 > 2x\} \right| \end{cases} \nonumber \\
& \equiv (I) + (II) + (III).
\end{align}
Below we will show that $(I) = O\left(e^{-x^{1/4}}\right)$, $(II) =
O\left(\dfrac{1}{ \ln(x) \cdot x^{4/3} } \right)$, and $(III) =
O\left(e^{-x^{1/4}}\right)$,
which establishes \eqref{eq:BoundOnTaylorErrorTermWanted1}. \\

\noindent \emph{Bound on (I)}: Define $C = \max_{x \in [0, \infty)} |F'''(x)|
< \infty$. Using \eqref{eq:MCConcentrationCondition} gives,
\begin{align*}
(I) & \equiv \E_x \left| F'''(\xi) (Z_1 - x)^3 \indicator\{Z_1 \leq x/2\} \right|
\leq C x^3 \cdot \P_x(Z_1 \leq x/2)
\leq C x^3 \cdot \P_x(|Z_1 - x| \geq x/2) \\
& \leq C x^3 \cdot C_1\big[1 + (1/2)^{2/3}x^{1/3}\big]e^{-C_2 (1/2)^{2/3} x^{1/3}} = O\left(e^{-x^{1/4}}\right).
\end{align*}

\noindent \emph{Bound on (II)}: By \eqref{eq:Derivativesloglogx}, for all $x
> 6$,
\begin{align}
\label{eq:BoundOnTermII}
(II) & \equiv \E_x \left| F'''(\xi) (Z_1 - x)^3 \indicator\{x/2 < Z_1 \leq 2x\} \right| \nonumber \\
& \leq \frac{7}{(x/2)^3 \ln(x/2)} \E_x \Big( |Z_1 - x|^3 \indicator\{x/2 < Z_1 \leq 2x\} \Big).
\end{align}
We define $\widetilde{Z}_1 =  Z_1 \cdot \indicator_{\{x/2 < Z_1 \leq 2x \}} +
x \cdot \indicator_{\{ Z_1 \not\in (x/2, 2x]\}}$, and bound the expectation
on the right hand side as follows:
\begin{align}
\label{eq:TermIIMainEstimate}
\E_x  \Big( |Z_1 - x|^3 \indicator\{&x/2 < Z_1 \leq 2x\} \Big)
 = \E_x\left( |\widetilde{Z}_1 - x|^3 \right) \nonumber \\
& = \int_0^x \P_x(|\widetilde{Z}_1 - x | > y) \cdot 3y^2 dy  \nonumber \\
& \leq 3 \int_0^x \P_x(|Z_1 - x| > y) \cdot y^2 dy  \nonumber \\
& = 3 \int_0^1 \P_x(|Z_1 - x| > \epsilon x) \cdot x^3 \epsilon^2 d\epsilon ~~~~~\mbox{(substitute $y = \epsilon x$)}  \nonumber \\
& \leq 3x^3 \int_0^1 \epsilon^2 \cdot C_1(1 + \epsilon^{2/3}x^{1/3})e^{-C_2 \epsilon^{2/3} x^{1/3}} d\epsilon  ~~~~~\mbox{(by \eqref{eq:MCConcentrationCondition})}  \nonumber \\
& = \frac{9}{2} x^3 \int_0^1 C_1 (t^{7/2} + t^{9/2} x^{1/3}) e^{-C_2 x^{1/3} t} dt  ~~~~~\mbox{(substitute $t = \epsilon^{2/3}$)}  \nonumber \\
& \leq \frac{9}{2} x^3 \int_0^1 C_1 (t^3 + t^4 x^{1/3}) e^{-C_2 x^{1/3} t} dt  \nonumber \\
& = O\left(x^{5/3} \right) ~~~~~~\mbox{(repeated integration by parts)}.
\end{align}
Combining \eqref{eq:BoundOnTermII} and \eqref{eq:TermIIMainEstimate} shows that $(II) = O\left(\dfrac{1}{ \ln(x) \cdot x^{4/3} } \right)$. \\

\noindent \emph{Bound on (III)}: By \eqref{eq:Derivativesloglogx}, for all $x
> 3$,
\begin{align}
\label{eq:BoundOnTermIII}
(III) \equiv \E_x \left| F'''(\xi) (Z_1 - x)^3 \indicator\{Z_1 > 2x\} \right|
\leq \frac{7}{x^3 \ln(x)} \E_x \Big( |Z_1 - x|^3 \indicator\{Z_1 > 2x\} \Big).
\end{align}
We define $\widetilde{Z}_1 =  Z_1 \cdot \indicator_{\{Z_1 > 2x \}} + x \cdot
\indicator_{\{ Z_1 \leq 2x\}}$ and write the expectation on the right hand
side as
\begin{align}
\label{eq:SplittingIII}
& \E_x \Big( |Z_1 - x|^3 \indicator\{Z_1 > 2x\} \Big)
= \E_x \Big( |\widetilde{Z}_1 - x|^3 \Big)
= \int_0^{\infty} \P_x(|\widetilde{Z}_1 - x| > y) \cdot 3y^2 dy \nonumber \\
& = 3 \left[ \int_0^x \P_x(Z_1 - x > x) y^2 dy ~+~ \int_x^{\infty} \P_x(Z_1 - x > y) y^2 dy \right].
\end{align}
The first integral in the brackets on the right hand side of
\eqref{eq:SplittingIII} is easily bounded using
\eqref{eq:MCConcentrationCondition}:
\begin{align}
\label{eq:FirstIntegralInBrackets}
\int_0^x \P_x(& Z_1 - x > x) y^2 dy
\leq \int_0^x \left[ C_1\left(1 + x^{1/3}\right)e^{-C_2 x^{1/3}} \right] y^2 dy
= O \left( e^{-x^{1/4}} \right).
\end{align}
The second integral in the brackets on the right hand side of
\eqref{eq:SplittingIII} is bounded as follows:
\begin{align}
\label{eq:SecondIntegralInBrackets}
\int_x^{\infty} \P_x(Z_1 - x > y) y^2 dy
& = \int_x^{\infty}   \P_x(|Z_1 - x| > y) y^2 dy \nonumber ~~~~~\mbox{($Z_1$ is nonnegative)}\\
& = \int_1^{\infty} \P_x(|Z_1 - x| > \epsilon x) x^3 \epsilon^2 d\epsilon ~~~~~\mbox{(substitute $y = \epsilon x$)} \nonumber \\
& \leq x^3 \int_1^{\infty} \epsilon^2 \cdot C_1\left(1 + x^{1/3} \epsilon^{1/3}\right) e^{-C_2 x^{1/3} \epsilon^{1/3}} d\epsilon ~~~~~\mbox{(by \eqref{eq:MCConcentrationCondition})} \nonumber \\
& = 3 x^3 \int_1^{\infty} C_1(t^8 + t^9 x^{1/3}) e^{-C_2 x^{1/3} t} dt  ~~~~~\mbox{(substitute $t = \epsilon^{1/3}$)} \nonumber \\
& = O \left( e^{-x^{1/4}} \right) ~~~~~~\mbox{(repeated integration by parts)}.
\end{align}
Combining \eqref{eq:SplittingIII}, \eqref{eq:FirstIntegralInBrackets}, and
\eqref{eq:SecondIntegralInBrackets} shows that
\begin{align}
\label{eq:TermIIIMainEstimate}
\E_x \Big( |Z_1 - x|^3 \indicator\{Z_1 > 2x\} \Big) = O \left( e^{-x^{1/4}} \right).
\end{align}
Hence, by \eqref{eq:BoundOnTermIII}, $(III) = O \left( e^{-x^{1/4}} \right)$.
\end{proof}

\begin{proof}[Proof of Lemma \ref{lem:TransienceRecurrenceConditionsForMC} (ii)]
Let $F(x) : [0, \infty) \rightarrow (0, \infty)$ be a smooth function such
that $F(x) = \dfrac{1}{\ln (x)}$, for $x > 2$. Then for all $x > 2$
\begin{align}
\label{eq:DerivativesLogxInverse}
& F'(x) = - \frac{1}{x \ln^2(x)}, \nonumber \\
& F''(x) = \frac{1}{x^2 \ln^2(x)} + \frac{2}{x^2 \ln^3(x)}, \nonumber \\
& F'''(x) = - \frac{2}{x^3 \ln^2(x)} - \frac{6}{x^3 \ln^3(x)} - \frac{6}{x^3 \ln^4(x)}.
\end{align}
Thus, by Taylor's Theorem with remainder, for all positive integer $x > 2$
\begin{align*}
& \E_x[F(Z_1)] \\
& = F(x) - \frac{\E_x[Z_1 - x]}{x \ln^2(x)} + \frac{1}{2} \left[ \frac{1}{x^2 \ln^2(x)} + \frac{2}{x^2 \ln^3(x)} \right]  \E_x[(Z_1 - x)^2] + \frac{1}{6} \E_x[F'''(\xi) (Z_1 - x)^3] \\
& = F(x) + \frac{1}{x \ln^2(x)} \left[ - \rho(x) + \frac{1}{2}\left(1 + \frac{2}{\ln(x)}\right)\nu(x) \right] ~+~ \frac{1}{6} \E_x[F'''(\xi) (Z_1 - x)^3]
\end{align*}
where $\xi$ is some (random) number between $Z_1$ and $x$. So,
\begin{align*}
\E_x&[F(Z_1)] \leq F(x) \\
& ~\iff~ \frac{1}{x \ln^2(x)} \left[ - \rho(x) + \frac{1}{2}\left(1 + \frac{2}{\ln(x)}\right)\nu(x) \right] + \frac{1}{6} \E_x[F'''(\xi) (Z_1 - x)^3] \leq 0 \\
& ~\iff~ \theta(x) \geq 1 + \frac{2}{\ln(x)} + \frac{1}{3} \frac{x \ln^2(x)}{\nu(x)} \E_x[F'''(\xi) (Z_1 - x)^3].
\end{align*}
Therefore, in light of Lemma
\ref{lem:GeneralTransienceRecurrenceConditionsMC} and the assumption that
$\liminf_{x \to \infty} \nu(x) > 0$, it will suffice to show
\begin{align}
\label{eq:BoundOnTaylorErrorTermWanted2}
 \E_x[F'''(\xi) (Z_1 - x)^3] = o\Big(\frac{1}{x \ln^3(x)}\Big).
\end{align}
Now, $|\E_x[F'''(\xi) (Z_1 - x)^3]| \leq (I) + (II) + (III)$, where terms
(I), (II), and (III) are defined just as in \eqref{eq:Term123Breakdown}, but
with our new Lyapunov function $F(x) = 1/\ln(x)$. The exact same proof as
above for part (i) of the lemma shows that $(I) =
O\left(e^{-x^{1/4}}\right)$. The bounds on (II) and (III) are also quite
similar:
\begin{align*}
& (II) \stackrel{\eqref{eq:DerivativesLogxInverse}}{\leq} \left[ \frac{14}{(x/2)^3 \ln^2(x/2)} \right] \cdot \E_x \Big( |Z_1 - x|^3 \indicator\{x/2 < Z_1 \leq 2x\} \Big)
\stackrel{\eqref{eq:TermIIMainEstimate}}{=} O\left(\frac{1}{\ln^2(x) \cdot x^{4/3}}\right). \\
& (III) \stackrel{\eqref{eq:DerivativesLogxInverse}}{\leq} \left[ \frac{14}{x^3 \ln^2(x)} \right] \cdot \E_x \Big( |Z_1 - x|^3 \indicator\{Z_1 > 2x\} \Big)
\stackrel{\eqref{eq:TermIIIMainEstimate}}{=} O\left(e^{-x^{1/4}}\right).
\end{align*}
Combining these estimates on (I), (II), and (III) shows that $\E_x[F'''(\xi)
(Z_1- x)^3] = O\left(\dfrac{1}{\ln^2(x) \cdot x^{4/3}}\right)$, which implies
\eqref{eq:BoundOnTaylorErrorTermWanted2}.
\end{proof}

\section{Proof of Propositions \texorpdfstring{\ref{prop:DecayOfHittingTimeTau0Vh}}{4.1}
and \texorpdfstring{\ref{prop:DecayOfAccumulatedSumTillTimeTau0Vh}}{4.2}}
\label{app:ProofOfPropsDecayAndAccumulatedSumTimeTau0Vh}

Our proof of Propositions \ref{prop:DecayOfHittingTimeTau0Vh} and
\ref{prop:DecayOfAccumulatedSumTillTimeTau0Vh} is based on the following
slightly more general proposition, which we isolate in this form because it
seems it may be applicable for analyzing other critical branching (or
branching-type) processes.

\begin{proposition}
\label{prop:GeneralizedMachineryFromLimitLawsERW} Let $(Z_k)_{k \geq 0}$ be
an irreducible, time-homogeneous Markov chain on state space $\N_0$, and
denote by $P_x^Z$ the probability measure for the Markov chain $(Z_k)$
started from $Z_0  = x$. Also, let $E_x^Z$ and $\Var_x^Z$ denote,
respectively, expectation and variance with respect to the measure $P_x^Z$.
Assume that the following four conditions are satisfied.
\begin{enumerate}
\item[(A)] \emph{\underline{Monotonicity}:} For any $x, y, z \in \N_0$ with
    $x > y$ \footnotemark{}, \footnotetext{Note that this condition (along
    with the Markov property) implies one can couple together versions of
    the process $(Z_k)$ starting from two different initial conditions $x >
    y$ in such a way that $Z_k^{(x)} \geq Z_k^{(y)}$, for all $k \geq 0$.}
\begin{align}
\label{eq:ConditionA}
P_x^Z(Z_1 \geq z) \geq P_y^Z(Z_1 \geq z).
\end{align}
\item[(B)] \emph{\underline{Expectation and Variance}:} There exist
    constants $\alpha > 0$ and $\beta < 0$ such that
\begin{align}
\label{eq:ConditionB}
E_x^Z(Z_1) = x +  \alpha \beta + O(x^{-1/3}) ~~\mbox{ and }~~ \Var_x^Z(Z_1) = 2\alpha x + O(x^{2/3}).
\end{align}
\item[(C)] \emph{\underline{Single Step Concentration Estimate}:} There
    exist constants $C_1, C_2 > 0$ such that:
\begin{align}
\label{eq:ConditionC}
&P_x^Z(|Z_1 - x| \geq \epsilon x) \leq C_1(1 + \epsilon^{2/3} x^{1/3}) e^{-C_2 \epsilon^{2/3}x^{1/3}} ,~ \mbox{ for all } x \in \N \mbox{ and } 0 < \epsilon \leq 1. \nonumber \\
&P_x^Z(|Z_1 - x| \geq \epsilon x) \leq C_1(1 + \epsilon^{1/3} x^{1/3}) e^{-C_2 \epsilon^{1/3}x^{1/3}} ,~ \mbox{ for all } x \in \N \mbox{ and } \epsilon \geq 1.
\end{align}
\item[(D)] \emph{\underline{Exit Probability Concentration Estimates}:}
    There exist constants $C_3, C_4 > 0$ and $N \in \N$ such that for all
    $x \geq N$ the following hold:
\begin{align}
\label{eq:ConditionD1}
& \sup_{0 \leq z < x} P_z^Z(Z_{\tau_{x^+}^Z} > x + y|\tau_{x^+}^Z < \tau_0^Z) \leq
\begin{cases}
C_3(1 + y^{2/3}x^{-1/3})e^{-C_4 y^{2/3} x^{-1/3}} ,~ \mbox{ for } 0 \leq y \leq x  \\
C_3(1 + y^{1/3})e^{-C_4 y^{1/3}} ,~~~~~~~~~~~~~~~~~ \mbox{ for } y \geq x
\end{cases}
\end{align}
\begin{align}
\label{eq:ConditionD2}
& \sup_{x < z < 4x} P_z^Z(Z_{\tau_{x^-}^Z \wedge {\tau_{(4x)^+}^Z}} < x - y) \leq C_3(1 + y^{2/3}x^{-1/3})e^{-C_4 y^{2/3} x^{-1/3}},~ \mbox{ for } 0 \leq y \leq x.
\end{align}
Here, as usual, $\tau_x^Z = \inf\{k > 0: Z_k = x\}$, and $\tau_{x^+}^Z$,
$\tau_{x^-}^Z$ are defined by $\tau_{x^+}^Z = \inf\{k > 0: Z_k \geq x\}$,
$\tau_{x^-}^Z = \inf\{k > 0: Z_k \leq x\}$.
\end{enumerate}
Then there exist constants $C_5, C_6 > 0$ such that, as $t \rightarrow
\infty$,
\begin{align}
\label{eq:ScalingOfTau0ZAndAccumulatedSum}
P_0^Z(\tau_0^Z > t) \sim C_5 t^{(\beta-1)} ~~\mbox{ and }~~~ P_0^Z\Big(\sum_{k=0}^{\tau_0^Z} Z_k > t \Big) \sim C_6 t^{(\beta-1)/2}.
\end{align}
\end{proposition}

\begin{proof}[Proof of Propositions \ref{prop:DecayOfHittingTimeTau0Vh} and \ref{prop:DecayOfAccumulatedSumTillTimeTau0Vh}]
We simply apply Proposition \ref{prop:GeneralizedMachineryFromLimitLawsERW}
to the (irreducible, time-homogeneous) Markov chain $(\Vh_k)_{k \geq 0}$ with
transition probabilities give by \eqref{eq:TransitionProbsVhk}. We consider
this Markov chain under the family of measure $P_{x,s}^V$, $x \in \N_0$, so
that $\Vh_0 = V_0 = x$ deterministically. Thus, the measure $P_{x,s}^V$ for
$(\Vh_k)$ is the equivalent of the measure $P_x^Z$ for the Markov chain
$(Z_k)$ in Proposition \ref{prop:GeneralizedMachineryFromLimitLawsERW}.

By construction the Markov chain $(\Vh_k)$ satisfies the monotonicity
condition (A) of Proposition \ref{prop:GeneralizedMachineryFromLimitLawsERW}.
Also, by Lemmas \ref{lem:ExpectationUkVkStoppedAtTausS} and
\ref{lem:VarianceUkVkStoppedAtTaus}, condition (B) is satisfied with $\alpha
= \mu_s > 0$ and $\beta = (1 - \delta) < 0$. Finally, condition (C) is
satisfied due to Lemma \ref{lem:ConcentrationEstimateUkVk} and condition (D)
is satisfied due to Lemma \ref{lem:ModifiedLemma51FromLimitLawsERW}. Thus,
the proposition is applicable and we have $P_{0,s}^V(\tau_0^{\Vh} > t) \sim
c_{10} t^{-\delta}$ and $P_{0,s}^V\Big(\sum_{k=0}^{\tau_0^{\Vh}} \Vh_k > t
\Big) \sim c_{11} t^{-\delta/2}$ for some constants $c_{10}, c_{11} > 0$.
\end{proof}

\begin{proof}[Proof of Proposition \ref{prop:GeneralizedMachineryFromLimitLawsERW} (Sketch)]
Our proof of Proposition \ref{prop:GeneralizedMachineryFromLimitLawsERW}
follows very closely the approach used in \cite{Kosygina2011} to prove
\emph{Theorems 2.1} and \emph{2.2}, so we will provide only a rough sketch.
Throughout we will use \emph{italics} when referring to all theorems, lemmas,
sections... etc. from \cite{Kosygina2011}, to distinguish from the
corresponding items in our paper. For the sake of comparison we restate
\emph{Theorems 2.1} and \emph{2.2} below explicitly.
\begin{theorem}[\emph{Theorems 2.1 and 2.2} of \cite{Kosygina2011}]
Let $\widetilde{\PB}$ be a probability measure on cookie environments
satisfying (IID), (BD), and (ELL), and let $(\Vt_k)_{k \geq 0}$ be the
associated backward branching process for ERW in this environment. Assume
$\delta > 0$ (where $\delta$ is given by \eqref{eq:OriginalDeltaDef}). Then
there exists constants $\Ct_1, \Ct_2 > 0$ such that
\begin{align}
\label{eq:ScalingOfTau0VtAndAccumulatedSum}
P_0^{\Vt}(\tau_0^{\Vt} > x) \sim \Ct_1x^{-\delta} ~~~\mbox{ and }~~~P_0^{\Vt}\Big(\sum_{k=0}^{\tau_0^{\Vt}} \Vt_k > x\Big) \sim \Ct_2x^{-\delta/2}.
\end{align}
\end{theorem}
The Markov chain $(Z_k)_{k \geq 0}$ of our proposition is the equivalent of
the process $(\Vt_k)_{k \geq 0}$, under the correspondence $\beta = 1 -
\delta$. More precisely, if $\delta > 1$ then $(\Vt_k)$ satisfies conditions
(A)-(D) of the proposition with $\beta = 1 - \delta$ and $\alpha = 1$, and
the decay rates in \eqref{eq:ScalingOfTau0ZAndAccumulatedSum} and
\eqref{eq:ScalingOfTau0VtAndAccumulatedSum} are the same with $\beta = 1 -
\delta$. The value of $\alpha$ does not effect the decay rates in
\eqref{eq:ScalingOfTau0ZAndAccumulatedSum}, only the constants $C_5$ and
$C_6$.

The main elements of the proof of \emph{Theorems 2.1} and \emph{2.2} in
\cite{Kosygina2011} are \emph{Lemmas 3.1-3.3} and \emph{3.5} in \emph{Section
3}, and \emph{Lemmas 5.1-5.3} and \emph{Corollary 5.5} in \emph{Section 5}.
\emph{Lemma 3.1} establishes convergence of the discrete process $(\Vt_k)$ to
a limiting diffusion $(\Yt_t)$, and the other lemmas in \emph{Section 3} give
properties of the limiting diffusion. The lemmas of \emph{Section 5} then
give tight estimates on exit probabilities of the process $(\Vt_k)$ from
certain intervals. The entire series of additional lemmas and propositions
used to establish \emph{Theorems 2.1} and \emph{2.2} in \emph{Sections 4, 6,
7} and \emph{8} use only the results of \emph{Sections 3} and \emph{5}, along
with the fact that the process $(\Vt_k)$ is an irreducible Markov chain on
state space $\N_0$ that is monotonic in the sense of \eqref{eq:ConditionA}.
Essentially the goal of these other sections is to show that the discrete
process $(\Vt_k)$ has the same type of scaling properties as the limiting
diffusion $(\Yt_t)$, and to do this requires some technical work, using the
estimates of \emph{Section 5}.

Now, let us compare to our situation for the process $(Z_k)$. Lemma
\ref{Lem:DiffusionApproximationLimit}, given below in Section
\ref{app:ProofLemmaDiffusionApproximationLimit}, is the analog of \emph{Lemma
3.1}, and Lemma \ref{Lem:DiffusionApproximationProperties} in Section
\ref{app:ProofLemmaDiffusionApproximationLimit} is the analog of \emph{Lemmas
3.2, 3.3,} and \emph{3.5}. Also, Condition (D), which we assume to hold for
the process $(Z_k)$, is the analog of \emph{Lemma 5.1} in
\cite{Kosygina2011}, where the corresponding property of the process
$(\Vt_k)$ is proven. Finally, Lemmas
\ref{lem:ModifiedLemma52FromLimitLawsERW},
\ref{lem:ModifiedLemma53FromLimitLawsERW}, and
\ref{lem:ModifiedCorollary55FromLimitLawsERW} in Section
\ref{app:ProofModifiedLemma52FromLimitLawsERW} are, respectively, the analogs
of \emph{Lemma 5.2}, \emph{Lemma 5.3}, and \emph{Corollary 5.5} in
\cite{Kosygina2011}.

With the analogous results to the lemmas in \emph{Section 3} and
\emph{Section 5} of \cite{Kosygina2011} established, the proof or our
Proposition \ref{prop:GeneralizedMachineryFromLimitLawsERW} for the process
$(Z_k)$ proceeds almost the same way, line by line, as the proof of
\emph{Theorems 2.1} and \emph{2.2} for the process $(\Vt_k)$. So, we will not
repeat it. However, for the sake of completeness, let us point out the few
small differences in our analog lemmas from the originals.
\begin{enumerate}
\item The bound of $\exp(-a^{n/10})$ in our Lemma
    \ref{lem:ModifiedLemma52FromLimitLawsERW}-(i) is instead
    $\exp(-a^{n/4})$ in \emph{Lemma 5.2}-(i). This is irrelevant for how
    the lemma is applied; a bound of $\exp(-a^{cn})$, for any $c > 0$,
    would be sufficient.
\item The concentration bounds \eqref{eq:ConditionD1} and
    \eqref{eq:ConditionD2} in our condition (D) are instead the following
    in \emph{Lemma 5.1}:
\begin{align}
\label{eq:AnalogOfD4}
& \sup_{0 \leq z < x} P_z^{\Vt}(\Vt_{\tau_{x^+}^{\Vt}} > x + y|\tau_{x^+}^{\Vt} < \tau_0^{\Vt}) \leq C(e^{-c y^2/x} + e^{-cy}) ~,~y \geq 0.\\
\label{eq:AnalogOfD5}
& \sup_{x < z < 4x} P_z^{\Vt}(\Vt_{\tau_{x^-}^{\Vt} \wedge \tau_{(4x)^+}^{\Vt}} < x - y) \leq Ce^{-c y^2/x} ~,~0 \leq y \leq x.
\end{align}
\end{enumerate}
The latter bounds are slightly stronger than ours. However, when $x$ is
large, with either \eqref{eq:AnalogOfD4} and \eqref{eq:AnalogOfD5} or with
\eqref{eq:ConditionD1} and \eqref{eq:ConditionD2}, the right hand sides
become small only when $y \gg x^{1/2}$. In \cite{Kosygina2011} the
inequalities \eqref{eq:AnalogOfD4} and \eqref{eq:AnalogOfD5} are applied
either when $y$ is of order $x$ or larger, or in other instances when $y$ is
of order $x^{2/3}$, giving respectively bounds on the right hand side which
are $O(e^{-cy})$ or $O(e^{-cx^{1/3}})$. If instead \eqref{eq:ConditionD1} and
\eqref{eq:ConditionD2} are used these bounds reduce to, respectively,
$O(e^{-y^{1/4}})$ and $O(e^{-x^{1/10}})$. But, again, this is irrelevant for
how the estimates are applied; any sort of stretched exponential decay would
be sufficient. With our weaker estimates slightly larger error terms arise in
the proofs, but they always remain negligible in comparison with all other
terms.
\end{proof}

\subsection{Diffusion approximation lemmas}
\label{app:ProofLemmaDiffusionApproximationLimit}

In the statement of the following lemma $(Z_k)_{k \geq 0}$ is an irreducible,
time-homogeneous Markov chain on state space $\N_0$ satisfying (A)-(D), as in
Proposition \ref{prop:GeneralizedMachineryFromLimitLawsERW}.

\begin{lemma}[Diffusion Approximation, Analog of Lemma 3.1 from \cite{Kosygina2011}]
\label{Lem:DiffusionApproximationLimit} Fix any $0 < \epsilon < y < \infty$,
and let $(Y(t))_{t \geq 0}$ be the solution of
\begin{align}
\label{eq:DiffusionApproximationProcessY}
dY(t) = \alpha \beta dt + \sqrt{2 \alpha  Y(t)^+} dB(t)  ~,~Y(0) = y
\end{align}
where $(B(t))_{t \geq 0}$ is a standard 1-dimensional Brownian
motion\footnotemark{}. \footnotetext{Note that the process $(\Yh(t))_{t \geq
0}$ defined by $\Yh(t) = 2 Y(t)/ \alpha$ is a squared Bessel process of
generalized dimension $2\beta$, since it satisfies $d\Yh(t) = 2\beta dt + 2
\sqrt{\Yh(t)^+} dB(t)$.} Also, let $Y_{\epsilon}(t) = Y(t \wedge
\tau_{\epsilon}^Y)$. For each $n \in \N$, let $(Z_{n,k})_{k \geq 0}$ be a
process with the distribution of $(Z_k)_{k \geq 0}$ when $Z_0 = \floor{yn}$,
and define $Y_{\epsilon,n}(t) =  \dfrac{Z_{n, \floor{nt} \wedge
\kappa_{\epsilon,n}}}{n}$, where $\kappa_{\epsilon,n} = \inf \{k \geq 0:
Z_{n,k} \leq \epsilon n\}$. In addition, let $\tau_{\epsilon,n} =
\kappa_{\epsilon,n}/n$. Then:
\begin{itemize}
\item[(i)] $(Y_{\epsilon,n}(t))_{t \geq 0} \stackrel{J_1}{\longrightarrow}
    (Y_{\epsilon}(t))_{t \geq 0}$, as $n \rightarrow \infty$, where
    $\stackrel{J_1}{\longrightarrow}$ denotes convergence in distribution
    with respect to the Skorokhod $J_1$ topology.
\item[(ii)] $\tau_{\epsilon,n} \stackrel{d.}{\longrightarrow}
    \tau_{\epsilon}^Y$, as $n \rightarrow \infty$.
\item[(iii)] $\int_0^{\tau_{\epsilon,n}} Y_{\epsilon,n}(t) dt
    \stackrel{d.}{\longrightarrow} \int_0^{\tau_{\epsilon}^Y}
    Y_{\epsilon}(t) dt$, as $n \rightarrow \infty$.
\end{itemize}
\end{lemma}

The following properties of the diffusion $Y(t)$ are established in
\cite{Kosygina2011} for the case $\alpha = 1$, with $\beta \equiv 1 -
\delta$, but generalize to any $\alpha > 0$. Indeed note that if $Y^{\alpha}$
is the solution to \eqref{eq:DiffusionApproximationProcessY} with a given
value $\alpha$ and $Y^1$ is the solution to
\eqref{eq:DiffusionApproximationProcessY} with $\alpha = 1$ (both with the
same value of $\beta$ and initial value $y$) then $(Y^{\alpha}(t))_{t \geq 0}
\stackrel{d.}{=} (Y^1(\alpha t))_{t \geq 0}$. All properties of $Y^{\alpha}$
follow from the corresponding properties for $Y^1$ and this relation.

\begin{lemma}[Properties of Limiting Diffusion, Analog of Lemmas 3.2, 3.3, and 3.5 from \cite{Kosygina2011}]
\label{Lem:DiffusionApproximationProperties} Let $P_y^Y(\cdot)$ be the
probability measure for the process $(Y(t))_{t \geq 0}$ in
\eqref{eq:DiffusionApproximationProcessY} with given initial value $Y(0)= y >
0$.
\begin{itemize}
\item[(i)] $\exists ~K_1, K_2 > 0$ such that $P_1^Y(\tau_0^Y > x) \sim K_1
    x^{\beta - 1}$ and $P_1^Y(\int_0^{\tau_0^Y} Y(t) dt > x^2) \sim K_2
    x^{\beta - 1}$, as $x \rightarrow \infty$.
\item[(ii)] For $0 \leq a < y < b$, $P_y^Y(\tau_a^Y < \tau_b^Y) = (b^{1 -
    \beta} - y^{1 - \beta})/ (b^{1 - \beta} - a^{1 - \beta})$.
\item[(iii)] The process $(Y(t))_{t \geq 0}$ when $Y(0) = 1$ has the same
    law as the process $(\frac{Y(ty)}{y})_{t \geq 0}$ when $Y(0) = y$.
\end{itemize}
\end{lemma}

In the remainder of Section \ref{app:ProofLemmaDiffusionApproximationLimit}
we will use the generic probability measure $\P$ and corresponding
expectation operator $\E$ for all random variables, including the Markov
chains $(Z_{n,k})_{k \geq 0}$ of Lemma \ref{Lem:DiffusionApproximationLimit}.
The proof of Lemma \ref{Lem:DiffusionApproximationLimit} is based on the
following result from \cite{Kosygina2017}.

\begin{lemma}(\cite[Lemma 7.1]{Kosygina2017}) Let $b \in \R$ and $D > 0$, and let $(\YM(t))_{t \geq 0}$ be the solution of
\label{lem:GeneralDiffusionApproximation}
\begin{align}
\label{eq:GeneralDiffusionEquation}
d\YM(t) = b \hspace{0.7 mm} dt + \sqrt{D \YM(t)^+} dB(t) ~,~\YM(0) = y > 0
\end{align}
where $(B(t))_{t \geq 0}$ is a standard 1-dimensional Brownian motion. Let
$(\ZM_{n,k})_{k \geq 0}$, $n \in \N$, be integer-valued (time-homogenous)
Markov chains such that $\ZM_{n,0} = \floor{n y_n}$, where $y_n \rightarrow
y$ as $n \rightarrow \infty$, and such that conditions (i) and (ii) below are
satisfied.
\begin{itemize}
\item[(i)] There is a sequence of positive integers $(N_n)_{n \geq 1}$ such
    that $N_n \rightarrow \infty$ with $N_n = o(n)$, a function $f:\N
    \rightarrow [0,\infty)$ with $f(x) \rightarrow 0$ as $x \rightarrow
    \infty$, and a function $g:\N \rightarrow [0,\infty)$ with $g(x)
    \searrow 0$ as $x \rightarrow \infty$, such that:
\begin{align*}
& (E) ~~ |\E(\ZM_{n,k+1} - \ZM_{n,k}|\ZM_{n,k} = m) - b| \leq f(m \vee N_n). \\
& (V) ~~ \left| \frac{\Var(\ZM_{n,k+1}|\ZM_{n,k} = m)}{m \vee N_n} - D \right| \leq g(m \vee N_n).
\end{align*}
\item[(ii)] For each $T, a > 0$
\begin{align*}
\E \left[ \max_{1 \leq k \leq (Tn) \wedge t_n} (\ZM_{n,k} - \ZM_{n,k-1})^2 \right] = o(n^2) ~,~ \mbox{ as } n \rightarrow \infty
\end{align*}
where $t_n \equiv \inf \{k \geq 0:\ZM_{n,k} \geq an\}$.
\end{itemize}
Then $(\YM_n(t))_{t \geq 0} \stackrel{J_1}{\longrightarrow} (\YM(t))_{t \geq
0}$, as $n \rightarrow \infty$, where $\YM_n(t) \equiv \ZM_{n,\floor{nt}}/n$.
\end{lemma}

We will also need the following lemma, which is a minor extension of Lemma
3.3 in \cite{Kosygina2014} where the same result is stated in the case $D =
2$.

\begin{lemma}
\label{lem:ContinuousMappingTheorem} For a function $f$ in the Skorokhod
space $D[0,\infty)$, define $\tau_{\epsilon}^f = \inf \{t \geq 0: f(t) \leq
\epsilon\}$, and define $\varphi_{\epsilon}(f) \in D[0,\infty)$ by
$(\varphi_{\epsilon}(f))(t) = f(t \wedge \tau_{\epsilon}^f)$. Let $\psi$ be
any of the following three mappings defined on $D[0,\infty):$
\begin{align*}
f \mapsto \tau_{\epsilon}^f \in [0,\infty], ~~~ f \mapsto \varphi_{\epsilon}(f) \in D[0,\infty), ~~~ f \mapsto \int_0^{\tau_{\epsilon}^f} f^+(t)dt \in [0,\infty].
\end{align*}
Denote by $\mbox{Cont}(\psi) = \{f \in D[0,\infty):\psi \mbox{ is continuous
at } f\}$ the set of continuity points of $\psi$. Then the solution $\YM =
(\YM(t))_{t \geq 0}$ of \eqref{eq:GeneralDiffusionEquation} satisfies $\P(\YM
\in \mbox{Cont}(\psi)) = 1$, for any $b \in \R$, $D > 0$, and $0 < \epsilon <
y < \infty$.
\end{lemma}

\begin{proof}[Proof of Lemma \ref{Lem:DiffusionApproximationLimit}]
For concreteness take $N_n = \floor{n^{1/2}}$ and define, for $n \in \N$,
integer-valued (time homogeneous) Markov chains $\ZM_n \equiv (\ZM_{n,k})_{k
\geq 0}$ by
\begin{align}
\label{eq:TransProbsZnk}
& \ZM_{n,0} = \floor{yn} ~\mbox{ and }~ \nonumber \\
& \P(\ZM_{n,k+1} = x + z | \ZM_{n,k} = x) = \begin{cases}
\P(Z_{n,k+1} = x + z| Z_{n,k} = x),~ z \in \Z \mbox{ and } x \geq N_n \\
\P(Z_{n,k+1} = N_n + z| Z_{n,k} = N_n),~ z \in \Z \mbox{ and } x < N_n. \\
\end{cases}
\end{align}
Thus, $\ZM_{n,0} = Z_{n,0} \equiv \floor{yn}$, and if the Markov chain
$(\ZM_{n,k})_{k \geq 0}$ is currently at level $x \geq N_n$, then it has the
same transition probabilities as the Markov chain $(Z_{n,k})_{k \geq 0}$. On
the other hand, if the Markov chain $(\ZM_{n,k})_{k \geq 0}$ is currently at
some level $x < N_n$, then the difference $(\ZM_{n,k+1} - \ZM_{n,k})$ between
the current value and next value has the same law as the difference
$(Z_{n,k+1} - Z_{n,k})$ when $Z_{n,k} = N_n$.

With this construction the chains $(Z_{n,k})_{k \geq 0}$ and $(\ZM_{n,k})_{k
\geq 0}$ can be naturally coupled until the first time $k$ that they fall
below level $\epsilon n$ (for all $n$ large enough that $N_n =
\floor{n^{1/2}} < \epsilon n$). Thus, by Lemma
\ref{lem:ContinuousMappingTheorem}
and the continuous mapping theorem it will suffice to show the following claim to establish the proposition. \\

\noindent \emph{Claim:} Define $(\YM_n(t))_{t \geq 0}$ by $\YM_n(t) = \ZM_{n,
\floor{nt}}/n$. Then $(\YM_n(t))_{t \geq 0} \stackrel{J_1}{\longrightarrow}
(Y(t))_{t \geq 0}$, where $(Y(t))_{t \geq 0}$
is the solution of \eqref{eq:DiffusionApproximationProcessY}. \\

\noindent \emph{Proof of Claim:} We apply Lemma
\ref{lem:GeneralDiffusionApproximation} with $b =  \alpha \beta$ and $D =
2\alpha$. By definition $(Z_{n,k})_{k \geq 0}$ has the distribution of
$(Z_k)_{k \geq 0}$ when $Z_0 = \floor{yn}$, where the Markov chain $(Z_k)_{k
\geq 0}$ is originally defined in Proposition
\ref{prop:GeneralizedMachineryFromLimitLawsERW}. It follows immediately from
condition (B) in the statement of this proposition and the definition
\eqref{eq:TransProbsZnk} that the Markov chains $(\ZM_{n,k})_{k \geq 0}$
satisfy conditions (E) and (V) in (i) of Lemma
\ref{lem:GeneralDiffusionApproximation}. Indeed, $f(x)$ and $g(x)$ are both
$O(x^{-1/3})$. To show (ii) of Lemma \ref{lem:GeneralDiffusionApproximation}
we define $M_n = \max_{1 \leq k \leq (Tn) \wedge t_n} |\ZM_{n,k} -
\ZM_{n,k-1}|$ and write the expectation as
\begin{align*}
\E \left[ \max_{1 \leq k \leq (Tn) \wedge t_n} (\ZM_{n,k} - \ZM_{n,k-1})^2 \right]
= \E(M_n^2)
= \int_0^{\infty} \P(M_n^2 > x) dx
=  \int_0^{\infty} \P(M_n > x^{1/2}) dx.
\end{align*}
The last integral may be decomposed as
\begin{align}
\label{eq:MnDecompose3Integrals}
\int_0^{\infty} \P(M_n > x^{1/2}) dx = \int_0^{n^{3/2}} \P(M_n > x^{1/2})dx + \int_{n^{3/2}}^{\infty} \P(M_n > x^{1/2})dx.
\end{align}
The first integral on the right hand side of \eqref{eq:MnDecompose3Integrals}
is at most $n^{3/2}$. Using the definition \eqref{eq:TransProbsZnk} along
with condition (C) in the statement of Proposition
\ref{prop:GeneralizedMachineryFromLimitLawsERW} and the union bound estimate
\begin{align*}
\P(M_n > x^{1/2}) \leq Tn \left[ \max_{m \leq an} \P\Big(|\ZM_{n,k} - \ZM_{n,k-1}| > x^{1/2} \Big| \ZM_{n,k-1} = m\Big) \right]
\end{align*}
one finds the second integral on the right hand side of
\eqref{eq:MnDecompose3Integrals} is o(1), as $n \rightarrow \infty$. This
establishes (ii) of Lemma \ref{lem:GeneralDiffusionApproximation}, and,
hence, the claim.
\end{proof}

\subsection{Exit probability lemmas}
\label{app:ProofModifiedLemma52FromLimitLawsERW}

In the statement of the following lemmas $(Z_k)_{k \geq 0}$ is an
irreducible, time-homogeneous Markov chain on state space $\N_0$ satisfying
(A)-(D), as in Proposition \ref{prop:GeneralizedMachineryFromLimitLawsERW}.

\begin{lemma}[Analog of Lemma 5.2 from \cite{Kosygina2011}]
\label{lem:ModifiedLemma52FromLimitLawsERW} Fix any $a \in (1,2]$, and define
$\IM_n = [a^n - a^{\frac{2}{3}n}, a^n + a^{\frac{2}{3}n}]$ and $\gamma_n =
\inf\{k \geq 0:Z_k \not\in(a^{n-1},a^{n+1})\}$. Then, for all sufficiently
large $n$ the following hold:
\begin{itemize}
\item[(i)] $P_x^Z\left( \mbox{dist}\big(Z_{\gamma_n},
    (a^{n-1},a^{n+1})\big) \geq a^{\frac{2}{3}(n-1)}\right) \leq
    \exp(-a^{n/10})$,~ for each $x \in \IM_n$.
\item[(ii)] $\Big|P_x^Z(Z_{\gamma_n} \leq a^{n-1}) - a^{1 - \beta}/(1 +
    a^{1- \beta}) \Big| \leq a^{-n/4}$,~ for each $x \in \IM_n$.
\end{itemize}
\end{lemma}

\begin{lemma}[Analog of Lemma 5.3 from \cite{Kosygina2011}]
\label{lem:ModifiedLemma53FromLimitLawsERW} For each $a \in (1,2]$ there is
some $\ell_a \in \N$ such that if $\ell, m, u, x \in \N$ satisfy $\ell_a \leq
\ell < m < u$ and $x \in \IM_m$ (where $\IM_m$ is as in Lemma
\ref{lem:ModifiedLemma52FromLimitLawsERW}) then
\begin{align*}
\frac{h^-_{a,\ell}(m) - 1}{h^-_{a,\ell}(u) - 1} \leq P_x^Z\Big(\tau^Z_{(a^{\ell})^-} > \tau^Z_{(a^{u})^+}\Big) \leq \frac{h^+_{a,\ell}(m) - 1}{h^+_{a,\ell}(u) - 1}
\end{align*}
where $h_{a,\ell}^{\pm}(i) = \prod_{j = \ell + 1}^i (a^{1 - \beta} \mp a^{-
\lambda j})$, $i > \ell$, and $\lambda$ is some small positive number not
depending on $\ell$.
\end{lemma}

\begin{lemma}[Analog of Corollary 5.5 from \cite{Kosygina2011}]
\label{lem:ModifiedCorollary55FromLimitLawsERW} For each $x \in \N_0$ there
exists $C = C(x) > 0$ such that
\begin{align*}
P_x^Z(\tau^Z_{n^+} < \tau^Z_0) \leq C/n^{1 - \beta} ~,~\mbox{ for all } n \in \N.
\end{align*}
Moreover, for each $\epsilon > 0$ there exists $c = c(\epsilon) > 0$ such
that
\begin{align*}
P_n^Z(\tau^Z_0 > \tau^Z_{(c n)^+}) < \epsilon ~,~\mbox{ for all } n \in \N.
\end{align*}
\end{lemma}

We will prove Lemma \ref{lem:ModifiedLemma52FromLimitLawsERW} below. The
proof is similar to the proof of Lemma 5.2 in \cite{Kosygina2011}, but the
remainder term $r^n_k$ in \eqref{eq:ExpectationWnkplus1GivenFnk} must be
bounded differently, because we do not have the same explicit form for the
transition probabilities of the Markov chain $(Z_k)$. The proofs of Lemmas
\ref{lem:ModifiedLemma53FromLimitLawsERW} and
\ref{lem:ModifiedCorollary55FromLimitLawsERW} are essentially the same as the
proofs of their counterparts in \cite{Kosygina2011}, and are therefore
omitted.

\begin{proof}[Proof of Lemma \ref{lem:ModifiedLemma52FromLimitLawsERW}]
Part (i) follows directly from condition (D) in the statement of Proposition
\ref{prop:GeneralizedMachineryFromLimitLawsERW} where the Markov chain
$(Z_k)$ is defined. To prove (ii), fix $a \in (1,2]$ and let $g \in
C_c^{\infty}([0,\infty))$ be any non-negative function such that $g(t) = t^{1
- \beta}$ for $t \in (\frac{2}{3a}, \frac{3a}{2})$. Then, for each $n \in
\N$, define a process $W^n \equiv (W^n_k)_{k \geq 0}$ by
\begin{align*}
W^n_k = g\Big(\frac{Z_{k \wedge \gamma_n}}{a^n}\Big).
\end{align*}
Let $\FM_k = \sigma(Z_0, \ldots, Z_k) \supseteq \sigma(W^n_0, \ldots,
W^n_k)$.
At the end of the main proof we will establish the following two claims. \\

\noindent \emph{Claim 1:} There exists some $B_1 = B_1(a) > 0$ such that
\begin{align*}
E_x^Z(\gamma_n) \leq B_1 a^n,~\mbox{for each $n \in \N$ and $x \in \IM_n$}.
\end{align*}

\noindent \emph{Claim 2:} The process $(W^n_k)_{k \geq 0}$ is ``close'' to
being a martingale, when $n$ is large, in the following sense: There exists
some $B_2 = B_2(a) > 0$ such that
\begin{align*}
|E_x^Z(W^n_{k+1}|\FM_k) - W^n_k| \leq B_2a^{-\frac{4}{3}n} ~a.s.,~\mbox{ for each $k \in \N_0$, $n \in \N$, and $x \in \IM_n$}.
\end{align*}

We now show how these two claims can be used to prove the lemma. Assume $Z_0
= x \in \IM_n$ and define a process $(\RM^n_k)_{k \geq 0}$ by
\begin{align*}
\RM^n_0 = 0 ~~\mbox{ and }~~ \RM_k^n = \sum_{j=1}^{k \wedge \gamma_n} \Big[ E_x^Z(W_j^n|\FM_{j-1}) - W_{j-1}^n) \Big] ~,~  k \geq 1.
\end{align*}
Observe that $(W^n_k - \RM^n_k)_{k \geq 0}$ is a martingale with initial
value $W^n_0$. Moreover, by Claims 1 and 2,
\begin{align}
\label{eq:ExpectationZngammaBound}
\left|E^Z_x(\RM^n_{\gamma_n})\right|
\leq E^Z_x\left( \sum_{j=1}^{\gamma_n} \left| E^Z_x(W_j^n|\FM_{j-1}) - W_{j-1}^n)\right| \right)
\leq E^Z_x(\gamma_n) \cdot B_2a^{-\frac{4}{3}n}
\leq B_3 a^{-\frac{1}{3}n}.
\end{align}
Since $|\RM_k^n| \leq \sum_{j=1}^{\gamma_n} \left| E^Z_x(W_j^n|\FM_{j-1}) -
W_{j-1}^n)\right|$, for all $k$, this shows that $(\RM_k^n)_{k \geq 0}$ is
uniformly integrable, and $(W_k^n)_{k \geq 0}$ is also uniformly integrable,
since $|W^n_k| \leq \norm{g}_{\infty}$ with probability 1. Thus, the
martingale  $(W^n_k - \RM^n_k)_{k \geq 0}$ is itself uniformly integrable, so
we may apply the optional stopping theorem to conclude
\begin{align}
\label{eq:WnkStoppingTimeDecomposition}
W^n_0 = E^Z_x(W^n_{\gamma_n}) - E^Z_x(\RM^n_{\gamma_n}).
\end{align}
Combining \eqref{eq:ExpectationZngammaBound} and
\eqref{eq:WnkStoppingTimeDecomposition} and using the fact that $g(t)$ is
equal to $t^{1-\beta}$ on $(\frac{2}{3a}, \frac{3a}{2})$ shows that
\begin{align*}
W_0^n - B_3 a^{-\frac{1}{3} n} \leq E^Z_x(W^n_{\gamma_n}) \leq
\begin{cases}
P_{x}^Z(Z_{\gamma_n} \in [a^{n+1}, a^{n+1} + a^{\frac{2}{3}(n-1)})) \cdot (a + a^{-\frac{1}{3}(n+2)})^{1-\beta} \\
+ P_{x}^Z(Z_{\gamma_n} \in (a^{n-1} - a^{\frac{2}{3}(n-1)}, a^{n-1}]) \cdot a^{-(1 - \beta)} \\
+ E^Z_x[W^n_{\gamma_n} \cdot \indicator{ \{Z_{\gamma_n} \not\in (a^{n-1} - a^{\frac{2}{3}(n-1)}, a^{n+1} + a^{\frac{2}{3}(n-1)})\}}]
\end{cases}
\end{align*}
and
\begin{align*}
W_0^n + B_3 a^{-\frac{1}{3} n} \geq E^Z_x(W^n_{\gamma_n}) \geq
\begin{cases}
P_{x}^Z(Z_{\gamma_n} \in [a^{n+1}, a^{n+1} + a^{\frac{2}{3}(n-1)})) \cdot a^{1-\beta} \\
+ P_{x}^Z(Z_{\gamma_n} \in (a^{n-1} - a^{\frac{2}{3}(n-1)}, a^{n-1}]) \cdot (a^{-1} - a^{-\frac{1}{3}(n+2)})^{1 - \beta} \\
+ E^Z_x[W^n_{\gamma_n} \cdot \indicator{ \{Z_{\gamma_n} \not\in (a^{n-1} - a^{\frac{2}{3}(n-1)}, a^{n+1} + a^{\frac{2}{3}(n-1)})\}}].
\end{cases}
\end{align*}
Using part (i) and the fact that $W^n_{\gamma_n}$ is bounded by
$\norm{g}_{\infty}$ gives
\begin{align*}
W_0^n = (1- p) \cdot a^{1-\beta} ~+~ p \cdot a^{-(1-\beta)} ~+~ O(a^{-\frac{1}{3} n})
\end{align*}
uniformly in the initial value $Z_0 = x \in \IM_n$, where $p \equiv
P_{x}^Z(Z_{\gamma_n} \leq a^{n-1})$. Now, using the definition $W_0^n =
g(Z_0/a^n) = g(x/a^n)$ shows also that $W_0^n = 1 + O(a^{-\frac{1}{3} n})$,
uniformly in $x \in \IM_n$. So, we have
\begin{align*}
1 = (1- p) \cdot a^{1 - \beta} ~+~ p \cdot a^{-(1-\beta)} ~+~ O(a^{-\frac{1}{3} n})
\end{align*}
uniformly in $x \in \IM_n$. Solving for $p$ gives $p = a^{1-\beta}/(1+a^{1-\beta}) + O(a^{-\frac{1}{3} n})$, which implies (ii). \\

\noindent \emph{Proof of Claim 1:} It will suffice to prove the claim for all
sufficiently large $n$. Assume $n$ is large enough that $\IM_n \subset
(a^{n-1},a^{n+1})$, and define $\gamma_n^- = \inf \{k \geq 0: Z_k \leq
a^{n-1}\}$. By monotonicity of the process $(Z_k)_{k \geq 0}$ with respect to
its initial condition
\begin{align*}
P_{z}^Z(\gamma_n < a^n) \geq P_{z}^Z(\gamma_n^- < a^n) \geq P_{\floor{a^{n+1}}}^Z(\gamma_n^- < a^n) ~,\mbox{ for all } z \in (a^{n-1},a^{n+1}).
\end{align*}
Further, by Lemma \ref{Lem:DiffusionApproximationLimit}-(ii), $\liminf_{n \to
\infty} P_{\floor{a^{n+1}}}^Z(\gamma_n^- < a^n) > 0$. So, there exist some $c
> 0$ and $n_0 \in \N$ such that
\begin{align}
\label{eq:ConstantChanceHitInLinearTime}
P_{z}^Z(\gamma_n < a^n) \geq c ~,~ \forall n \geq n_0 \mbox{ and } z \in (a^{n-1},a^{n+1}).
\end{align}
Let $t_0 = 0$ and $t_{i+1} = t_i + \ceil{a^n}$, $i \geq 0$. By
\eqref{eq:ConstantChanceHitInLinearTime} $P_{x}^Z(\gamma_n > t_{i+1} |
\gamma_n > t_i) \leq 1 - c$, for all $n \geq n_0$, $x \in \IM_n$, and $i \geq
0$. Thus, for each $n \geq n_0$, $x \in \IM_n$, and $m \geq 0$
\begin{align*}
P_{x}^Z(\gamma_n > m \cdot \ceil{a^n}) = P_{x}^Z(\gamma_n > t_m) \leq (1-c)^m.
\end{align*}
This implies the claim. \\

\noindent \emph{Proof of Claim 2:} Since $|W^n_k|$ is bounded by
$\norm{g}_{\infty}$, for all $n, k$, it will suffice to show the claim for
sufficiently large $n$. Throughout we assume $n$ is sufficiently large that
$(\frac{a^{n-1} - a^{\frac{2}{3}n}}{a^n},\frac{a^{n+1} +
a^{\frac{2}{3}n}}{a^n}) \subset (\frac{2}{3a},\frac{3a}{2})$ and that $Z_0 =
x \in \IM_n$. All $O(\cdot)$ estimates stated will be uniform in $x \in
\IM_n$ and $k \in \N_0$. By Taylor's Theorem,
\begin{align*}
g\Big(\frac{Z_{k+1}}{a^n}\Big)
= g\Big(\frac{Z_k}{a^n}\Big) + g'\Big(\frac{Z_k}{a^n}\Big) \frac{Z_{k+1} - Z_k}{a^n} + \frac{1}{2} g''\Big(\frac{Z_k}{a^n}\Big) \frac{(Z_{k+1} - Z_k)^2}{a^{2n}} + \frac{1}{6} g'''(t) \frac{(Z_{k+1} - Z_k)^3}{a^{3n}}
\end{align*}
where $t$ is some random point between $Z_k/a_n$ and $Z_{k+1}/a^n$. Thus, on
the event $\{\gamma_n > k\}$, we have
\begin{align}
\label{eq:ExpectationWnkplus1GivenFnk}
E^Z_x&(W^n_{k+1}|\FM_k) - W^n_k
= E^Z_x\Big[g\Big(\frac{Z_{k+1}}{a^n}\Big)\Big|\FM_k\Big] - g\Big(\frac{Z_k}{a^n}\Big) \nonumber \\
& = \frac{1}{a^n} g'\Big(\frac{Z_k}{a^n}\Big) E^Z_x[Z_{k+1} - Z_k|\FM_k] + \frac{1}{2a^{2n}} g''\Big(\frac{Z_k}{a^n}\Big) E^Z_x[(Z_{k+1} - Z_k)^2|\FM_k] + r^n_k \nonumber \\
& = \frac{1}{a^n} g'\Big(\frac{Z_k}{a^n}\Big) \Big[ \alpha \beta + O\left(a^{-\frac{1}{3}n} \right) \Big]
+  \frac{1}{2a^{2n}} g''\Big(\frac{Z_k}{a^n}\Big) \Big[ 2\alpha Z_k + O\left( a^{\frac{2}{3}n} \right) \Big] + r^n_k
\end{align}
by \eqref{eq:ConditionB}, where the remainder $r^n_k$ satisfies
\begin{align*}
|r^n_k| \leq \frac{1}{6} \norm{g'''}_{\infty} E^Z_x\Big[ \frac{|Z_{k+1} - Z_k|^3}{a^{3n}} \Big| \FM_k \Big].
\end{align*}
Now, for $t \in (\frac{2}{3a}, \frac{3a}{2})$, $t g''(t) = - \beta g'(t)$.
So, for $k < \gamma_n$, $\frac{Z_k}{a^n} g''(\frac{Z_k}{a^n}) = - \beta
g'(\frac{Z_k}{a^n})$. Plugging this relation back into
\eqref{eq:ExpectationWnkplus1GivenFnk} and simplifying we find that, on the
event $\{\gamma_n > k\}$,
\begin{align}
\label{eq:ExpectationWnkplus1GivenFnkSimplified}
E^Z_x(W^n_{k+1}|\FM_k) - W^n_k
&=  \frac{1}{a^n} g'\Big(\frac{Z_k}{a^n}\Big) \cdot O\left(a^{-\frac{1}{3}n}\right) ~+~ \frac{1}{2a^{2n}} g''\Big(\frac{Z_k}{a^n}\Big) \cdot O\left( a^{\frac{2}{3}n} \right) ~+~ r^n_k \nonumber \\
& = O\left(a^{-\frac{4}{3}n}\right) + r^n_k.
\end{align}
The remainder term $r^n_k$ can be bounded as
\begin{align}
\label{eq:rnkEstimate}
|r^n_k|
&\leq \frac{ \norm{g'''}_{\infty}}{6 a^{3n}} \max_{z \in (a^{n-1}, a^{n+1})} E^Z_x\left( |Z_{k+1} - Z_k|^3 \Big| Z_k = z \right) \nonumber \\
&= \frac{ \norm{g'''}_{\infty}}{6 a^{3n}} \max_{z \in (a^{n-1}, a^{n+1})} E^Z_z(|Z_1 - z|^3).
\end{align}
We split this last expectation into three pieces:
\begin{align*}
E^Z_z(|Z_1 - z|^3)
& = E^Z_z\Big(\indicator\{Z_1 \leq z/2\} \cdot |Z_1 - z|^3\Big) + E^Z_z\Big(\indicator\{z/2 < Z_1 \leq 2z\} \cdot |Z_1 - z|^3\Big) \\
& ~~~~ + E^Z_z\Big(\indicator\{Z_1 > 2z\} \cdot |Z_1 - z|^3\Big).
\end{align*}
By \eqref{eq:ConditionC} the first term on the right hand side is
$O(e^{-z^{1/4}})$, and using \eqref{eq:ConditionC} and calculations exactly
as in the derivation of \eqref{eq:TermIIMainEstimate} and
\eqref{eq:TermIIIMainEstimate} shows that the second and third terms are,
respectively,  $O(z^{5/3})$ and $O(e^{-z^{1/4}})$. Plugging these estimates
back into \eqref{eq:rnkEstimate} gives $|r^n_k| = O(a^{-\frac{4}{3}n})$, and
combining that with \eqref{eq:ExpectationWnkplus1GivenFnkSimplified} shows
also that
\begin{align*}
E^Z_x(W^n_{k+1}|\FM_k) - W^n_k = O(a^{-\frac{4}{3}n}) ~,~\mbox{ on the event } \{\gamma_n > k\}.
\end{align*}
Of course, on the event $\{\gamma_n \leq k\}$, $W^n_{k+1} = W^n_k$
(deterministically), so this proves the claim.
\end{proof}

\section{Proof of Propositions \texorpdfstring{\ref{prop:DecayOfHittingTimeSigma0V}}{4.3}
and \texorpdfstring{\ref{prop:DecayOfAccumulatedSumTillTimeSigma0V}}{4.4}}
\label{app:ProofOfPropsDecayAndAccumulatedSumTimeSigma0V}

In this section we use Propositions \ref{prop:DecayOfHittingTimeTau0Vh} and
\ref{prop:DecayOfAccumulatedSumTillTimeTau0Vh} to prove Propositions
\ref{prop:DecayOfHittingTimeSigma0V} and
\ref{prop:DecayOfAccumulatedSumTillTimeSigma0V}. It is assumed throughout, if
not otherwise specified, that $R_0 = s$ and $V_0 = 0$. Thus, $\tau_0 = 0$,
$\Vh_0 = 0$, and $\sigma_0^V = \tau_{\tau_0^{\Vh}}$ (where $(\tau_k)_{k \geq
0}$ is as in \eqref{eq:DefTaukTaukprime}). The details of the proofs are
somewhat technical, but the general ideas are fairly simple, so we will
present these first before proceeding to the formal proofs. First, for
Proposition \ref{prop:DecayOfHittingTimeSigma0V}, observe that if
$\tau_0^{\Vh} = m$, for some large $m$, then
\begin{align*}
\sigma_0^V = \tau_m = \sum_{k=1}^m (\tau_k - \tau_{k-1}) \approx m \cdot \mu_s.
\end{align*}
So, by Proposition \ref{prop:DecayOfHittingTimeTau0Vh}, for large $x$,
\begin{align*}
P_{0,s}^V(\sigma_0^V > x) \approx P_{0,s}^V(\tau_0^{\Vh} > x/ \mu_s) \approx  c_{10} \cdot (x/\mu_s)^{-\delta} = c_{12} \cdot x^{-\delta}.
\end{align*}

Next, for Proposition \ref{prop:DecayOfAccumulatedSumTillTimeSigma0V}, note
that if $\sum_{k=0}^{\sigma_0^V} V_k$ is large, generally it will be the case
that $\sigma_0^V$ is large as well, and also that $V_k$ will be relatively
large for most times $k$ between $0$ and $\sigma_0^V$ (because the $V_k$
process is unlikely to remain close to 0 very long without hitting 0). Thus,
by Lemma \ref{lem:ConcentrationEstimateUkVk}, the $\Vh_j$ process, and also
the $V_k$ process, will not fluctuate too much too rapidly, relative to their
current values. So, very roughly speaking, we should expect that
\begin{align*}
\sum_{k=0}^{\sigma_0^V} V_k
= \sum_{j=0}^{\tau_0^{\Vh}-1} \sum_{k = \tau_j}^{\tau_{j+1} - 1} V_k
\approx \sum_{j=0}^{\tau_0^{\Vh}-1} \sum_{k = \tau_j}^{\tau_{j+1} - 1} \Vh_j
\approx \sum_{j = 0}^{\tau_0^{\Vh}-1} \Vh_j \cdot \mu_s
= \sum_{j = 0}^{\tau_0^{\Vh}} \Vh_j \cdot \mu_s,
\end{align*}
when either the sum on the right hand side or left hand side (equivalently
both sums) are large. Therefore, by Proposition
\ref{prop:DecayOfAccumulatedSumTillTimeTau0Vh}, we should expect that, for
large $x$,
\begin{align*}
P_{0,s}^V\left( \sum_{k=0}^{\sigma_0^V} V_k > x \right)
\approx P_{0,s}^V\left( \sum_{j = 0}^{\tau_0^{\Vh}} \Vh_j > x/\mu_s \right)
\approx c_{11} \cdot (x/\mu_s)^{-\delta/2} = c_{13} \cdot x^{-\delta/2}.
\end{align*}

\begin{proof}[Proof of Proposition \ref{prop:DecayOfHittingTimeSigma0V}]
Fix any $\epsilon > 0$. It will suffice to show that
\begin{align}
\label{eq:LimsupBoundProbSigma0VGreatern}
\limsup_{n \to \infty}~ n^{\delta} \cdot P_{0,s}^V(\sigma_0^V > n) \leq c_{12} + \epsilon
\end{align}
and
\begin{align}
\label{eq:LiminfBoundProbSigma0VGreatern}
\liminf_{n \to \infty}~ n^{\delta} \cdot P_{0,s}^V(\sigma_0^V > n) \geq c_{12} - \epsilon.
\end{align}

\noindent \underline{Proof of \eqref{eq:LimsupBoundProbSigma0VGreatern}}:
Choose $\rho > 0$ sufficiently small that $(1+\rho)^{\delta} \mu_s^{\delta}
(c_{10} + \rho) + \rho \leq c_{10} \mu_s^{\delta} + \epsilon = c_{12} +
\epsilon$. For $n \in \N$, let $m = m(n) = (1 + \rho) \mu_s n$ (it is not
assumed that $m$ is an integer). Then $\{\sigma_0^V > m\} \subseteq
\{\tau_0^{\Vh} > n\} \cup \{\tau_n > m\}$. So,
\begin{align}
\label{eq:SplitProbSigma0VGreatermTwoPartsLimsupBound}
m^{\delta} P_{0,s}^V(\sigma_0^V > m) \leq m^{\delta} \left[ P_{0,s}^V(\tau_0^{\Vh} > n) + P_{0,s}^V(\tau_n > m) \right].
\end{align}
By Proposition \ref{prop:DecayOfHittingTimeTau0Vh}, for all sufficiently
large $n$,
\begin{align}
\label{eq:ProbSigma0VGreatermPart1LimsupBound}
m^{\delta}  P_{0,s}^V(\tau_0^{\Vh} > n)
= \frac{m^{\delta}}{n^{\delta}} \cdot  \left[ n^{\delta} P_{0,s}^V(\tau_0^{\Vh} > n) \right] \leq (1 + \rho)^{\delta} \mu_s^{\delta} \cdot (c_{10} + \rho).
\end{align}
Also, by Lemma \ref{lem:GeneralLargeDeviationBoundSumsIID}, $P_{0,s}^V(\tau_n
> m) = P_{0,s}^V\big( \sum_{i=1}^n (\tau_i - \tau_{i-1}) > (1+\rho)\mu_s
n\big)$ decays exponentially in $n$, since the random variables $(\tau_i -
\tau_{i-1})_{i \geq 1}$ are i.i.d. with exponential tails and mean $\mu_s$.
Thus, for all sufficiently large $n$,
\begin{align}
\label{eq:ProbSigma0VGreatermPart2LimsupBound}
m^{\delta} P_{0,s}^V(\tau_n > m) \leq \rho.
\end{align}
Combining the estimates
\eqref{eq:SplitProbSigma0VGreatermTwoPartsLimsupBound},
\eqref{eq:ProbSigma0VGreatermPart1LimsupBound}, and
\eqref{eq:ProbSigma0VGreatermPart2LimsupBound} shows that, for all
sufficiently large $n$,
\begin{align*}
m^{\delta} P_{0,s}^V(\sigma_0^V > m) \leq (1+\rho)^{\delta} \mu_s^{\delta}(c_{10} + \rho) + \rho \leq c_{12} + \epsilon,
\end{align*}
which proves \eqref{eq:LimsupBoundProbSigma0VGreatern}. \\

\noindent \underline{Proof of \eqref{eq:LiminfBoundProbSigma0VGreatern}}:
Choose $\rho \in (0,1)$ sufficiently small that $(1 -\rho)^{\delta}
\mu_s^{\delta} (c_{10} - \rho) - \rho \geq c_{10} \mu_s^{\delta} - \epsilon =
c_{12} - \epsilon$. For $n \in \N$, let $m = m(n) = (1 - \rho) \mu_s n$
(again, it is not assumed that $m$ is an integer). Then $\{\sigma_0^V > m\}
\supseteq \{\tau_0^{\Vh} > n\} \cap \{\tau_n \geq m\}$. So,
\begin{align}
\label{eq:SplitProbSigma0VGreatermTwoPartsLiminfBound}
P_{0,s}^V(\sigma_0^V > m)
\geq P_{0,s}^V(\tau_0^{\Vh} > n, \tau_n \geq m)
\geq P_{0,s}^V(\tau_0^{\Vh} > n) - P_{0,s}^V(\tau_n < m).
\end{align}
Now, by Proposition \ref{prop:DecayOfHittingTimeTau0Vh},
$P_{0,s}^V(\tau_0^{\Vh} > n) \geq n^{-\delta}(c_{10} - \rho)$, for all
sufficiently large $n$. Also, by Lemma
\ref{lem:GeneralLargeDeviationBoundSumsIID}, $P_{0,s}^V(\tau_n < m) =
P_{0,s}^V\big( \sum_{i=1}^n (\tau_i - \tau_{i-1}) < (1-\rho)\mu_s n\big)$
decays exponentially in $n$. So, $P_{0,s}^V(\tau_n < m) \leq \rho \cdot
m^{-\delta} $, for all sufficiently large $n$. Plugging these estimates back
into \eqref{eq:SplitProbSigma0VGreatermTwoPartsLiminfBound} shows that,
\begin{align*}
m^{\delta} P_{0,s}^V(\sigma_0^V > m)
\geq \left(\frac{m}{n}\right)^{\delta}(c_{10} - \rho) - \rho
= (1 -\rho)^{\delta} \mu_s^{\delta} (c_{10} - \rho) - \rho
\geq c_{12} - \epsilon
\end{align*}
for all sufficiently large $n$, which proves
\eqref{eq:LiminfBoundProbSigma0VGreatern}.
\end{proof}

\begin{proof}[Proof of Proposition \ref{prop:DecayOfAccumulatedSumTillTimeSigma0V}]
Fix any $\epsilon_1 \in (0,\frac{1}{22})$ and $\epsilon_2 \in (0, \frac{1}{4}
\epsilon_1)$. Then, for $n \in \N$, define the following random variables:
\begin{itemize}
\item $T_0 = 0$ and $T_{i+1} = \inf\{k \geq T_i + \floor{n^{\epsilon_1}} :
    R_k = s\}$, $i \geq 0$.
\item $i_{\max} = \max\{i \geq 0:T_i \leq \sigma_0^V\}$, $k_{\max} =
    T_{i_{\max}}$, and $j_{\max}$ is the unique $j$ such that
    $\tau_{j_{\max}} = k_{\max}$.
\item $K_i = \{T_{i-1}, T_{i-1} + 1, \ldots, T_i\}$ and $J_i = \{j \in
    \N_0: \tau_j \in K_i\}$, $i \geq 1$.
\item $j_i^{\max} = \max\{j: j \in J_i\}$ and $j_i^{\min} = \min\{j: j \in
    J_i\}$.
\item $J_i^0 = J_i \backslash \{j_i^{\max}\}$ and $\Jt_i = \{j_i^{\min},
    j_i^{\min} + 1, \ldots, j_i^{\min} + \floor{2n^{\epsilon_1}} \}$.
\item $K_i^0 = K_i \backslash \{T_i\}$ and \\
$\Kt_i = \{\tau_{j_i^{\min}}, \tau_{j_i^{\min}}+1, \ldots, \tau_{j_i^{\min}
+ \floor{2n^{\epsilon_1}}}\} = \{T_{i-1}, T_{i-1} + 1, \ldots,
\tau_{j_i^{\min} + \floor{2n^{\epsilon_1}}}\}$.
\end{itemize}
Also, denote $V_{\max} = \max \{V_k: 0 \leq k \leq \sigma_0^V\}$ and
$\Delta_{\tau,j} = \tau_j - \tau_{j -1}$, for $j \geq 1$, and define the
following events:
\begin{itemize}
\item $E_i = \big\{ \max_{k \in K_i} |V_{T_{i-1}} - V_k| > 2n^{\epsilon_1}
    n^{\frac{1}{3}(1 + \epsilon_1)} \big\}$, $i \geq 1$.
\item $F_i =  \Big\{ \left| \sum_{j \in J_i^0}  (\mu_s - \Delta_{\tau,j+1})
    \right| > n^{\frac{1}{2} \epsilon_1 + \epsilon_2}\Big\}$, $i \geq 1$.
\item $A_1 = \{\sigma_0^V > n^{\frac{1}{2}(1 + \epsilon_2)} \}$.
\item $A_2 = \{V_{\max} > 2n^{\frac{1}{2}(1 + \epsilon_2)} \}$.
\item $A_3 = \{ \exists 1 \leq i \leq i_{\max} \mbox{ such that $E_i$
    occurs} \}$.
\item $A_4 = \{ \exists 1 \leq i \leq i_{\max} \mbox{ such that $F_i$
    occurs} \}$.
\item $A_5 = \{ \exists 1 \leq i \leq i_{\max} + 1 \mbox{ such that $T_i -
    T_{i-1} > 2n^{\epsilon_1}$} \}$.
\item $G = A_1^c \cap A_2^c \cap A_3^c \cap A_4^c \cap A_5^c$ (the ``good
    event'').
\end{itemize}
At the end of the main proof we will establish the following claim. \\

\noindent
\emph{Claim 1:} For all sufficiently large $n$, $P_{0,s}^V(G^c) \leq 5 n^{-(\frac{\delta}{2} + \frac{\delta \epsilon_2}{4})}$. \\

\noindent By the triangle inequality,
\begin{align*}
\left|\sum_{k=0}^{\sigma_0^V} V_k - \mu_s \sum_{j=0}^{\tau_0^{\Vh}} \Vh_j \right|
& \leq \left| \sum_{k=0}^{\sigma_0^V} V_k - \sum_{i=1}^{i_{\max}}V_{T_{i-1}}(T_i - T_{i-1}) \right| + \left|\sum_{i=1}^{i_{\max}}V_{T_{i-1}}(T_i - T_{i-1}) -  \mu_s \sum_{j=0}^{\tau_0^{\Vh}} \Vh_j \right| \\
& \leq \begin{cases} \left| \sum_{k=0}^{k_{\max} - 1} V_k - \sum_{i=1}^{i_{\max}}V_{T_{i-1}}(T_i - T_{i-1}) \right|
+ \sum_{k = k_{\max}}^{\sigma_0^V} V_k \\
+ \left| \sum_{j=0}^{j_{\max} - 1} \mu_s \Vh_j - \sum_{i=1}^{i_{\max}}V_{T_{i-1}}(T_i - T_{i-1}) \right|
+ \mu_s \sum_{j = j_{\max}}^{\tau_0^{\Vh}} \Vh_j \end{cases} \\
& \equiv (I) + (II) + (III) + (IV).
\end{align*}
We now show how each of the terms (I), (II), (III), and (IV) can be bounded on the event $G$. \\

\noindent \underline{Bound on (I)}: On the event $G$,
\begin{align*}
&(I) \equiv \left| \sum_{k=0}^{k_{\max} - 1} V_k - \sum_{i=1}^{i_{\max}}V_{T_{i-1}}(T_i - T_{i-1}) \right|
= \left| \sum_{i=1}^{i_{\max}} \sum_{k \in K_i^0} V_k - \sum_{i=1}^{i_{\max}} \sum_{k \in K_i^0} V_{T_{i-1}} \right| \\
& \stackrel{(a)}{\leq} \sum_{i=1}^{i_{\max}}\sum_{k \in K_i^0} 2n^{\epsilon_1} n^{\frac{1}{3}(1 + \epsilon_1)}
= k_{\max} \cdot 2n^{\epsilon_1} n^{\frac{1}{3}(1 + \epsilon_1)}
\leq \sigma_0^V \cdot 2n^{\epsilon_1} n^{\frac{1}{3}(1 + \epsilon_1)} \\
& \stackrel{(b)}{\leq} n^{\frac{1}{2}(1 + \epsilon_2)} 2n^{\epsilon_1} n^{\frac{1}{3}(1 + \epsilon_1)}
\stackrel{(c)}{\leq} 2n^{\frac{5}{6} + \frac{11}{6} \epsilon_1}.
\end{align*}
Step (a) follows from the fact that $G \subset A_3^c$, step (b) follows from the fact that $G \subset A_1^c$, and step (c) follows from the fact that $\epsilon_2 < \epsilon_1$. \\

\noindent \underline{Bound on (II) and (IV)}: First note that $(IV) \equiv
\mu_s \cdot \sum_{j = j_{\max}}^{\tau_0^{\Vh}} \Vh_j \leq \mu_s \cdot \sum_{k
= k_{\max}}^{\sigma_0^V} V_k \equiv \mu_s \cdot (II)$, so it will suffice to
bound (II). Now, on the event $G$,
\begin{align*}
(II) \leq V_{\max} |\sigma_0^V - k_{\max}|
\stackrel{(a)}{\leq} 2n^{\frac{1}{2}(1 + \epsilon_2)} \cdot 2n^{\epsilon_1}
\stackrel{(b)}{\leq} 4n^{\frac{1}{2}(1 + 3 \epsilon_1)}.
\end{align*}
Step (a) follows from the fact that $G \subset A_2^c$ and $G \subset A_5^c$, and step (b) follows from the fact that $\epsilon_2 < \epsilon_1$. \\

\noindent \underline{Bound on (III)}: On the event $G$,
\begin{align*}
(III) & \equiv \left| \sum_{j=0}^{j_{\max} - 1} \mu_s \Vh_j ~-~ \sum_{i=1}^{i_{\max}}V_{T_{i-1}}(T_i - T_{i-1}) \right| \\
& = \left| \sum_{i=1}^{i_{\max}} \sum_{j \in J_i^0} \mu_s \Vh_j ~-~ \sum_{i=1}^{i_{\max}} \sum_{j \in J_i^0} V_{T_{i-1}} \cdot \Delta_{\tau,j+1} \right| \\
& = \left| \sum_{i=1}^{i_{\max}} \sum_{j \in J_i^0} \mu_s (\Vh_j - V_{T_{i-1}})  ~+~ \sum_{i=1}^{i_{\max}} \sum_{j \in J_i^0} V_{T_{i-1}} (\mu_s - \Delta_{\tau,j+1}) \right| \\
\end{align*}
\begin{align*}
& \leq \sum_{i=1}^{i_{\max}} \sum_{j \in J_i^0} \mu_s |\Vh_j - V_{T_{i-1}}| ~+~ \sum_{i=1}^{i_{\max}} V_{T_{i-1}} \left| \sum_{j \in J_i^0} (\mu_s - \Delta_{\tau,j+1}) \right| \\
& \stackrel{(a)}{\leq} \sum_{i=1}^{i_{\max}}  \sum_{j \in J_i^0} \mu_s \cdot 2n^{\epsilon_1} n^{\frac{1}{3}(1 + \epsilon_1)} ~+~ \sum_{i=1}^{i_{\max}} 2n^{\frac{1}{2}(1 + \epsilon_2)} \cdot n^{\epsilon_1/2 + \epsilon_2} \\
& \stackrel{(b)}{\leq} \tau_0^{\Vh} \cdot \mu_s \cdot 2n^{\epsilon_1} n^{\frac{1}{3}(1 + \epsilon_1)} ~+~ \frac{\sigma_0^V}{\floor{n^{\epsilon_1}}} 2n^{\frac{1}{2}(1 + \epsilon_2)} n^{\epsilon_1/2 + \epsilon_2} \\
& \stackrel{(c)}{\leq} n^{\frac{1}{2}(1 + \epsilon_2)} \cdot \mu_s \cdot 2n^{\epsilon_1} n^{\frac{1}{3}(1 + \epsilon_1)} ~+~ \left( n^{\frac{1}{2}(1 + \epsilon_2)} 2n^{-\epsilon_1} \right) \cdot 2n^{\frac{1}{2}(1 + \epsilon_2)} n^{\epsilon_1/2 + \epsilon_2} \\
& \stackrel{(d)}{\leq} 2 \mu_s n^{\frac{5}{6} + \frac{11}{6} \epsilon_1} ~+~ 4n^{1 + 2\epsilon_2 - \epsilon_1/2}.
\end{align*}
Step (a) follows from the fact that $G \subset A_2^c$, $G \subset A_3^c$, and
$G \subset A_4^c$. Step (b) follows from the relations $j_{\max} \leq
\tau_0^{\Vh}$ and $i_{\max} \leq k_{\max}/\floor{n^{\epsilon_1}} \leq
\sigma_0^V/\floor{n^{\epsilon_1}}$. Step (c) follows from the inequality
$\tau_0^{\Vh} \leq \sigma_0^V$
and the fact that $G \subset A_1^c$. Finally, Step (d) follows from the fact that $\epsilon_2 < \epsilon_1$. \\

Now, let $\alpha = \max\left\{(\frac{5}{6} + \frac{11}{6} \epsilon_1), (1 +
2\epsilon_2 - \frac{\epsilon_1}{2})\right\}$. By the choice of $\epsilon_1$
and $\epsilon_2$, we have $\frac{1}{2}(1 + 3 \epsilon_1) < \alpha$ and
$\alpha < 1$. Combining the estimates on the terms (I)-(IV) we find that, on
the event $G$,
\begin{align*}
& \Big|\sum_{k=0}^{\sigma_0^V} V_k - \mu_s \sum_{j=0}^{\tau_0^{\Vh}} \Vh_j \Big|
\leq (I) + (II) + (III) + (IV) \\
& \leq 2n^{\alpha} + 4n^{\alpha} + \left[2 \mu_s n^{\alpha} + 4 n^{\alpha}\right] + 4 \mu_s n^{\alpha}
 \leq 16 \mu_s n^{\alpha}.
\end{align*}
Using this estimate along with Claim 1 and Proposition
\ref{prop:DecayOfAccumulatedSumTillTimeTau0Vh}
we can now establish the proposition. \\

\noindent Given any $\epsilon \in (0,1)$:
\begin{itemize}
\item Choose $N_1 \in \N$ such that $P_{0,s}^V(G^c) \leq 5
    n^{-(\frac{\delta}{2} + \frac{\delta \epsilon_2}{4})}$, for $n \geq
    N_1$ (possible by Claim 1).
\item Choose $N_2 \in \N$ such that $16 \mu_s n^{\alpha} \leq \epsilon n$,
    for $n \geq N_2$ (possible since $\alpha < 1$).
\item Finally, choose $N_3 \in \N$ such that, for all $n \geq N_3$,
\begin{align*}
& P_{0,s}^V\Big( \sum_{j = 0}^{\tau_0^{\Vh}} \Vh_j > \frac{1 - \epsilon}{\mu_s} \cdot n \Big) \leq (c_{11} + \epsilon) \left(  \frac{1 - \epsilon}{\mu_s} \cdot n \right)^{-\delta/2}  \mbox{ and } \\
& P_{0,s}^V\Big( \sum_{j = 0}^{\tau_0^{\Vh}} \Vh_j > \frac{1 + \epsilon}{\mu_s} \cdot n \Big) \geq (c_{11} - \epsilon) \left(  \frac{1 + \epsilon}{\mu_s} \cdot n \right)^{-\delta/2}.
\end{align*}
This is possible by Proposition
\ref{prop:DecayOfAccumulatedSumTillTimeTau0Vh}.
\end{itemize}
Then for each $n \geq N_0 \equiv \max\{N_1, N_2, N_3\}$ we have
\begin{align*}
&P_{0,s}^V\Big(\sum_{k=0}^{\sigma_0^V} V_k > n\Big)
\leq P_{0,s}^V\Big(\mu_s \sum_{j=0}^{\tau_0^{\Vh}} \Vh_j > (1 - \epsilon) n \Big) +  P_{0,s}^V\Big( \Big| \sum_{k=0}^{\sigma_0^V} V_k - \mu_s \sum_{j=0}^{\tau_0^{\Vh}} \Vh_j \Big| > \epsilon n \Big) \\
&\leq (c_{11} + \epsilon) \left(  \frac{1 - \epsilon}{\mu_s} \cdot n \right)^{-\delta/2} + P_{0,s}^V(G^c)
\leq (c_{11} + \epsilon) \left(  \frac{1 - \epsilon}{\mu_s} \cdot n \right)^{-\delta/2} + 5 n^{-(\frac{\delta}{2} + \frac{\delta \epsilon_2}{4})}
\end{align*}
and
\begin{align*}
&P_{0,s}^V\Big(\sum_{k=0}^{\sigma_0^V} V_k > n\Big)
\geq P_{0,s}^V\Big(\mu_s \sum_{j=0}^{\tau_0^{\Vh}} \Vh_j > (1 + \epsilon) n \Big) -  P_{0,s}^V\Big( \Big| \sum_{k=0}^{\sigma_0^V} V_k - \mu_s \sum_{j=0}^{\tau_0^{\Vh}} \Vh_j \Big| > \epsilon n \Big) \\
&\geq (c_{11} - \epsilon) \left(  \frac{1 + \epsilon}{\mu_s} \cdot n \right)^{-\delta/2} - P_{0,s}^V(G^c)
\geq (c_{11} - \epsilon) \left(  \frac{1 + \epsilon}{\mu_s} \cdot n \right)^{-\delta/2} - 5 n^{-(\frac{\delta}{2} + \frac{\delta \epsilon_2}{4})}.
\end{align*}
Since $\epsilon \in (0,1)$ was arbitrary it follows that
\begin{align*}
\lim_{n \to \infty} n^{\delta/2} \cdot P_{0,s}^V\Big(\sum_{k=0}^{\sigma_0^V} V_k > n \Big) = \mu_s^{\delta/2} \cdot c_{11}  = c_{13}.
\end{align*}
This concludes the main proof of the proposition, and it remains only now to
show Claim 1. To do this, though, we will first
need to establish some auxiliary claims that will be used in its proof. \\

\noindent \\
\emph{Claim 2:}  For all sufficiently large $n$,
\begin{align*}
P_{0,s}^V(E_i|V_{T_{i-1}} = x) \leq e^{-n^{\epsilon_1/2}}, \mbox{ for each } 0 \leq x \leq n^{\frac{1}{2}(1 + \epsilon_1)} \mbox{ and } i \geq 1.
\end{align*}

\noindent \emph{Proof.} By Lemma \ref{lem:ConcentrationEstimateUkVk}, we have
that for all sufficiently large $n$
\begin{align}
\label{eq:BoundOnFluctuationVTaujVTaujplus1}
P_{0,s}^V\Big( \max_{\tau_j \leq k \leq \tau_{j+1}} |V_{\tau_j} - V_k| > n^{\frac{1}{3}(1 + \epsilon_1)} \Big| V_{\tau_j} = x\Big)
\leq e^{-n^{\frac{1}{19}(1 + \epsilon_1)}},~0 \leq x \leq 2n^{\frac{1}{2}(1 + \epsilon_1)}.
\end{align}
For $i \geq 1$, define $T_0^{(i)} = T_{i-1}$ and $T_j^{(i)} = \inf\{k >
T_{j-1}^{(i)}:R_k = s\}$, $j \geq 1$. Then, let
\begin{align*}
A_j^{(i)} = \Big\{ \max_{ T_{j-1}^{(i)} \leq k \leq  T_j^{(i)} } |V_{T_{j-1}^{(i)}} - V_k| \leq n^{\frac{1}{3}(1 + \epsilon_1)} \Big\}.
\end{align*}
Observe that, for all sufficiently large $n$ and $1 \leq j \leq
2n^{\epsilon_1}$, if $V_{T_{i-1}} \leq n^{\frac{1}{2}(1 + \epsilon_1)}$ and
$A_1^{(i)}, \ldots, A_j^{(i)}$ all occur then $V_{T_j^{(i)}} \leq
n^{\frac{1}{2}(1 + \epsilon_1)} + j \cdot n^{\frac{1}{3}(1 + \epsilon_1)}
\leq 2n^{\frac{1}{2}(1 + \epsilon_1)}$. Thus, by
\eqref{eq:BoundOnFluctuationVTaujVTaujplus1}, for all sufficiently large $n$
and $0 \leq x \leq  n^{\frac{1}{2}(1 + \epsilon_1)}$ we have
\begin{align}
\label{eq:MaxVTiminus1MinusVkSmall}
P&_{0,s}^V\Big( \max_{k \in \Kt_i} |V_{T_{i-1}} - V_k| > 2n^{\epsilon_1} n^{\frac{1}{3}(1 + \epsilon_1)} \Big|V_{T_{i-1}} = x\Big) \nonumber \\
& \leq P_{0,s}^V\left( \exists 1 \leq j \leq 2n^{\epsilon_1} : \big(A_j^{(i)}\big)^c \mbox{ occurs } \Big| V_{T_{i-1}} = x \right) \nonumber \\
& \leq \sum_{j = 1}^{\floor{2n^{\epsilon_1}}} P_{0,s}^V\left( \big(A_j^{(i)}\big)^c \Big| A_1^{(i)}, \ldots, A_{j-1}^{(i)}, V_{T_{i-1}} = x\right)
\leq 2n^{\epsilon_1} \cdot e^{-n^{\frac{1}{19}(1 + \epsilon_1)}}.
\end{align}
Also, since $|\Kt_i| \geq \floor{2n^{\epsilon_1}}$ (deterministically) and
the event $\{K_i \not\subset \Kt_i\}$ is independent of the value of
$V_{T_{i-1}}$ it follows from \eqref{eq:tausStausRExponentialTail} that, for
any $x$,
\begin{align}
\label{eq:BoundKiNotInKti}
P_{0,s}^V(K_i \not\subset \Kt_i |V_{T_{i-1}} = x) = P_{0,s}^V(K_i \not\subset \Kt_i) \leq P_{0,s}^V(|K_i| > \floor{2n^{\epsilon_1}}) \leq c_3 e^{-c_4 \floor{n^{\epsilon_1}}}.
\end{align}
Combing the estimates \eqref{eq:MaxVTiminus1MinusVkSmall} and
\eqref{eq:BoundKiNotInKti} shows that, for all sufficiently large $n$ and $0
\leq x \leq n^{\frac{1}{2}(1+ \epsilon_1)}$,
\begin{align*}
P_{0,s}^V(E_i|V_{T_{i-1}} = x) \leq 2n^{\epsilon_1} e^{-n^{\frac{1}{19}(1 + \epsilon_1)}} + c_3 e^{-c_4 \floor{n^{\epsilon_1}}} \leq e^{-n^{\epsilon_1/2}}.
\end{align*}

\noindent \emph{Claim 3:} For all sufficiently large $n$,
\begin{align*}
P_{0,s}^V(F_i) \leq e^{-n^{\epsilon_1 \epsilon_2}}, \mbox{ for each } i \geq 1.
\end{align*}

\noindent \emph{Proof.} Define $j_0 \equiv j_1^{\max} = \inf \{j: \tau_j \geq
\floor{n^{\epsilon_1}}\} = \inf\{j : \sum_{\ell = 1}^j \Delta_{\tau,\ell}
\geq \floor{n^{\epsilon_1}} \}$. Since the $(R_k)$ process is Markovian,
\begin{align*}
P_{0,s}^V(F_i) = P_{0,s}^V(F_1) = P_{0,s}^V \Big( \Big| \sum_{j = 1}^{j_0} (\mu_s - \Delta_{\tau,j}) \Big| > n^{\frac{1}{2} \epsilon_1 + \epsilon_2} \Big) ~,~\mbox{ for all } i \geq 1.
\end{align*}
Let $N_n^+ = \lfloor{ (n^{\epsilon_1} + n^{\frac{1}{2}(\epsilon_1 +
\epsilon_2)})/\mu_s \rfloor}$ and $N_n^- = \lfloor{ (n^{\epsilon_1} -
n^{\frac{1}{2}(\epsilon_1 + \epsilon_2)})/\mu_s \rfloor}$. Define events
$B_1$ and $B_2$ by
\begin{align*}
B_1 = \Big\{ \Big| \sum_{j=1}^{N_n^+} \Delta_{\tau,j} - \mu_s N_n^+ \Big| \leq (N_n^+)^{\frac{1}{2} + \epsilon_2} \Big\} ~~\mbox{ and }~~
B_2 = \Big\{ \Big| \sum_{j=1}^{N_n^-} \Delta_{\tau,j} - \mu_s N_n^- \Big| \leq (N_n^-)^{\frac{1}{2} + \epsilon_2} \Big\}.
\end{align*}
Since the random variables $(\Delta_{\tau,j})_{j \geq 1}$ are i.i.d. with
mean $\mu_s$ and exponentials tails (due to
\eqref{eq:tausStausRExponentialTail}), it follows from Lemma
\ref{lem:GeneralLargeDeviationBoundSumsIID} that there exist some constants
$C_1, C_2 > 0$ such that
\begin{align*}
P_{0,s}^V \Big( \Big| \sum_{j=1}^{m} (\Delta_{\tau,j} - \mu_s) \Big| > \epsilon m \Big) \leq C_1 e^{-C_2 \epsilon^2 m} ~,~\mbox{ for all } 0 < \epsilon < 1 \mbox{ and } m \in \N.
\end{align*}
Using this with $m = N_n^+, N_n^-$ and $\epsilon = (N_n^+)^{-1/2 +
\epsilon_2}, (N_n^-)^{-1/2 + \epsilon_2}$, respectively, shows that
\begin{align*}
P_{0,s}^V(B_1^c) \leq C_1 e^{-C_2(N_n^+)^{2 \epsilon_2}} ~~\mbox{ and }~~ P_{0,s}^V(B_2^c) \leq C_1 e^{-C_2(N_n^-)^{2 \epsilon_2}}.
\end{align*}
Hence, for all sufficiently large $n$,
\begin{align*}
P_{0,s}^V\big( (B_1 \cap B_2)^c\big) \leq C_1 e^{-C_2(N_n^+)^{2 \epsilon_2}} + C_1 e^{-C_2(N_n^-)^{2 \epsilon_2}} \leq e^{-n^{\epsilon_1 \epsilon_2}}.
\end{align*}
So, it will suffice to show that, for all sufficiently large $n$,
\begin{align}
\label{eq:WantToShowSumSmallOnB1IntersectB2}
\Big| \sum_{j = 1}^{j_0} (\Delta_{\tau,j} - \mu_s) \Big| \leq n^{\frac{1}{2} \epsilon_1 + \epsilon_2} ~~\mbox{ on the event } B_1 \cap B_2.
\end{align}
Now, since $\epsilon_1 < 1/2$, $n^{\epsilon_1 \epsilon_2} < n^{\frac{1}{2}
\epsilon_2}$. Using this fact and a little bit of algebra it follows from the
definitions of $B_1$ and $B_2$ that, for all sufficiently large $n$, on the
event $B_1 \cap B_2$
\begin{align*}
n^{\epsilon_1} < \sum_{j=1}^{N_n^+} \Delta_{\tau,j} < n^{\epsilon_1} + 2n^{\frac{1}{2}(\epsilon_1 + \epsilon_2)} ~~~\mbox{ and }~~~
n^{\epsilon_1} - 2n^{\frac{1}{2}(\epsilon_1 + \epsilon_2)} <  \sum_{j=1}^{N_n^-} \Delta_{\tau,j} < n^{\epsilon_1}.
\end{align*}
Together these inequalities imply $N_n^- \leq j_0 \leq N_n^+$. So, by the
definitions of $N_n^+$ and $N_n^-$,
\begin{align*}
&\sum_{j=1}^{j_0} (\Delta_{\tau,j} - \mu_s) < [ n^{\epsilon_1} + 2n^{\frac{1}{2}(\epsilon_1 + \epsilon_2)} ] - N_n^- \mu_s \leq 4n^{\frac{1}{2}(\epsilon_1 + \epsilon_2)} \mbox{ and } \\
&\sum_{j=1}^{j_0} (\Delta_{\tau,j} - \mu_s) > [ n^{\epsilon_1} - 2n^{\frac{1}{2}(\epsilon_1 + \epsilon_2)} ] - N_n^+ \mu_s \geq - 4n^{\frac{1}{2}(\epsilon_1 + \epsilon_2)}
\end{align*}
on the event $B_1 \cap B_2$, for all sufficiently large $n$. Since
$4n^{\frac{1}{2}(\epsilon_1 + \epsilon_2)} \leq n^{\frac{1}{2} \epsilon_1 +
\epsilon_2}$ for all sufficiently large $n$,
this shows that \eqref{eq:WantToShowSumSmallOnB1IntersectB2} holds for all sufficiently large $n$. \\

\noindent \emph{Claim 4:} Denote $\Vh_{\max} = \max \{\Vh_j: 0 \leq j \leq
\tau_0^{\Vh} \}$. There exists some $c_{15} > 0$ such that
\begin{align}
\label{eq:VhmaxNotToBig}
P_{0,s}^V(\Vh_{\max} > t) \leq c_{15} t^{-\delta} ~,~ \mbox{ for all } t \in [0,\infty).
\end{align}

\noindent \emph{Proof.} As noted above in the proof of Propositions
\ref{prop:DecayOfHittingTimeTau0Vh} and
\ref{prop:DecayOfAccumulatedSumTillTimeTau0Vh} in Appendix
\ref{app:ProofOfPropsDecayAndAccumulatedSumTimeTau0Vh}, the Markov chain
$(\Vh_k)_{k \geq 0}$ satisfies all conditions (A)-(D) of Proposition
\ref{prop:GeneralizedMachineryFromLimitLawsERW} with $\alpha = \mu_s$ and
$\beta = 1 - \delta$. Thus, all the lemmas in Appendix
\ref{app:ProofOfPropsDecayAndAccumulatedSumTimeTau0Vh} which hold for a
Markov chain $(Z_k)$ satisfying these properties apply to $(\Vh_k)$. We will
use Lemma \ref{lem:ModifiedCorollary55FromLimitLawsERW}. By this lemma there
exists some constant $C = C(0)$ such that
\begin{align*}
P_{0,s}^V(\Vh_{\max} > n) = P_{0,s}^V(\tau_{(n+1)^+}^{\Vh} < \tau_0^{\Vh}) \leq C(n+1)^{-\delta} ~,~ \mbox{ for all } n \in \N.
\end{align*}
This implies \eqref{eq:VhmaxNotToBig}. \\

\noindent \emph{Proof of Claim 1.} We will show that $P_{0,s}^V(A_i) \leq
n^{-(\frac{\delta}{2} + \frac{\delta \epsilon_2}{4})}$, for each $i = 1,
\ldots, 5$.
The estimate for $A_1$ follows directly from Proposition \ref{prop:DecayOfHittingTimeSigma0V}. The bounds for the other events are given below. \\

\noindent
\underline{Bound for $A_2$}: \\
We decompose $P_{0,s}^V(V_{\max} > 2n^{\frac{1}{2}(1 + \epsilon_2)})$ as
\begin{align*}
P_{0,s}^V\left(V_{\max} > 2n^{\frac{1}{2}(1 + \epsilon_2)}\right) & = \begin{cases}
P_{0,s}^V( \Vh_{\max} > n^{\frac{1}{2}(1 + \epsilon_2)},  V_{\max} > 2n^{\frac{1}{2}(1 + \epsilon_2)}) \\
+ P_{0,s}^V( \Vh_{\max} \leq n^{\frac{1}{2}(1 + \epsilon_2)},  V_{\max} > 2n^{\frac{1}{2}(1 + \epsilon_2)}, \tau_0^{\Vh} > n) \\
+  P_{0,s}^V( \Vh_{\max} \leq n^{\frac{1}{2}(1 + \epsilon_2)},  V_{\max} > 2n^{\frac{1}{2}(1 + \epsilon_2)}, \tau_0^{\Vh} \leq n)
\end{cases} \\
& \equiv (I) + (II) + (III).
\end{align*}
By Claim 4, $(I) \leq P_{0,s}^V(\Vh_{\max} > n^{\frac{1}{2}(1 + \epsilon_2)})
\leq c_{15} (n^{\frac{1}{2}(1 + \epsilon_2)})^{-\delta}$. Also, by
Proposition \ref{prop:DecayOfHittingTimeTau0Vh}, $(II) \leq
P_{0,s}^V(\tau_0^{\Vh} > n) \leq 2 c_{10} n^{-\delta}$, for all sufficiently
large $n$. Term $(III)$ is estimated as follows:
\begin{align}
\label{eq:BoundTermIIIClaim2}
(III) &\leq P_{0,s}^V\left (\exists~ 0 \leq j \leq n-1: \Vh_j \equiv V_{\tau_j} \leq n^{\frac{1}{2}(1 + \epsilon_2)}  \mbox{ and }
\max_{\tau_j \leq k \leq \tau_{j+1}} |V_{\tau_j} - V_k| > n^{\frac{1}{2}(1+ \epsilon_2)} \right) \nonumber \\
&\leq \sum_{j=0}^{n-1}  P_{0,s}^V\left( V_{\tau_j} \leq n^{\frac{1}{2}(1 + \epsilon_2)}  \mbox{ and }
\max_{\tau_j \leq k \leq \tau_{j+1}} |V_{\tau_j} - V_k| > n^{\frac{1}{2}(1+ \epsilon_2)} \right) \nonumber \\
&\leq \sum_{j=0}^{n-1}  P_{0,s}^V\left( \max_{\tau_j \leq k \leq \tau_{j+1}} |V_{\tau_j} - V_k| > n^{\frac{1}{2}(1+ \epsilon_2)} \Big|
V_{\tau_j} \leq n^{\frac{1}{2}(1 + \epsilon_2)} \right) \nonumber \\
& \leq n \cdot \left[ \max_{0 \leq x \leq n^{\frac{1}{2}(1 + \epsilon_2)}}
P_{x,s}^V \left(\max_{0 \leq k \leq \tau_s^R} |V_k - x| > n^{\frac{1}{2}(1+ \epsilon_2)} \right) \right].
\end{align}
By \eqref{eq:VkLargeEpsBound} the right hand side of
\eqref{eq:BoundTermIIIClaim2} is at most $n e^{-n^{1/6}}$, for all
sufficiently large $n$ (note that although \eqref{eq:VkLargeEpsBound} is not
directly applicable when $x=0$, $P_{0,s}^V(\max_{0 \leq k \leq \tau_s^R} V_k
> t) \leq P_{1,s}^V(\max_{0 \leq k \leq \tau_s^R} V_k > t)$, for any $t > 0$,
so the bound still holds in this case as well). Combining these estimates on
terms $(I)$, $(II)$, and $(III)$ we find that, for all sufficiently large
$n$,
\begin{align*}
P_{0,s}^V(A_2) \equiv P_{0,s}^V\left(V_{\max} > 2n^{\frac{1}{2}(1 + \epsilon_2)}\right)
\leq c_{15} (n^{\frac{1}{2}(1 + \epsilon_2)})^{-\delta} + 2 c_{10} n^{-\delta} + n e^{-n^{1/6}}
\leq n^{-(\delta/2 + \delta \epsilon_2/4)}.
\end{align*} \\

\noindent
\underline{Bound for $A_3$}: \\
By construction, $i_{\max} \leq \tau_0^{\Vh}$. So, by Proposition
\ref{prop:DecayOfHittingTimeTau0Vh}, for all sufficiently large $n$
\begin{align}
\label{eq:BoundVimaxBiggern}
P_{0,s}^V(i_{\max} > n) \leq P_{0,s}^V(\tau_0^{\Vh} > n) \leq 2 c_{10} n^{-\delta}.
\end{align}
Combining this with Claim 2 and Claim 4 shows that, for all sufficiently
large $n$,
\begin{align*}
P_{0,s}^V(A_3)
& = \begin{cases} P_{0,s}^V(A_3, \Vh_{\max} > n^{\frac{1}{2}(1 + \epsilon_2)}) + P_{0,s}^V(A_3, \Vh_{\max} \leq n^{\frac{1}{2}(1 + \epsilon_2)}, i_{\max} > n) \\
+ P_{0,s}^V(A_3, \Vh_{\max} \leq n^{\frac{1}{2}(1 + \epsilon_2)}, i_{\max} \leq n) \end{cases} \\
& \leq \begin{cases} P_{0,s}^V(\Vh_{\max} > n^{\frac{1}{2}(1 + \epsilon_2)}) +  P_{0,s}^V(i_{\max} > n) \\
+ P_{0,s}^V(\exists 1 \leq i \leq n:V_{T_{i-1}} \leq n^{\frac{1}{2}(1+\epsilon_2)} \mbox{ and $E_i$ occurs}) \end{cases} \\
& \leq c_{15} (n^{\frac{1}{2}(1 + \epsilon_2)})^{-\delta} + 2 c_{10} n^{-\delta} + ne^{-n^{\epsilon_1/2}} \\
& \leq n^{-(\delta/2 + \delta \epsilon_2/4)}.
\end{align*}

\noindent
\underline{Bound for $A_4$}: \\
By \eqref{eq:BoundVimaxBiggern} and Claim 3 we have, for all sufficiently
large $n$,
\begin{align*}
P_{0,s}^V(A_4)
& = P_{0,s}^V(A_4, i_{\max} > n) + P_{0,s}^V(A_4, i_{\max} \leq n) \\
& \leq P_{0,s}^V(i_{\max} > n) + P_{0,s}^V(\exists 1 \leq i \leq n: F_i \mbox{ occurs}) \\
& \leq 2c_{10}n^{-\delta} + n e^{-n^{\epsilon_1 \epsilon_2}} \\
& \leq n^{-(\delta/2 + \delta \epsilon_2/4)}.
\end{align*}

\noindent
\underline{Bound for $A_5$}: \\
By \eqref{eq:BoundKiNotInKti}, $P_{0,s}^V(|K_i| > 2n^{\epsilon_1}) \leq c_3
e^{-c_4 \floor{n^{\epsilon_1}}}$, for each $i$. Using this along with
\eqref{eq:BoundVimaxBiggern} shows that, for all sufficiently large $n$,
\begin{align*}
P_{0,s}^V(A_5)
& = P_{0,s}^V(A_5, i_{\max} > n) + P_{0,s}^V(A_5, i_{\max} \leq n) \\
& \leq P_{0,s}^V(i_{\max} > n) + P_{0,s}^V(\exists 1 \leq i \leq n+1: |K_i| > 2n^{\epsilon_1}) \\
& \leq 2c_{10}n^{-\delta} + (n+1) \cdot c_3 e^{-c_4 \floor{n^{\epsilon_1}}} \\
& \leq n^{-(\delta/2 + \delta \epsilon_2/4)}.\qedhere
\end{align*}
\end{proof}



\begin{thebibliography}{99}
\bibitem{Amir2016} G.~Amir, N.~Berger, and T.~Orenshtein.
\newblock Zero-one law for directional transience of one dimensional excited
  random walks.
\newblock \emph{Annales de l'Institut Henri Poincar\'{e} - Probabilit\'{e}s
  Statistiques}, 52(1):47--57, 2016.

\bibitem{Basdevant2008} A.~Basdevant and A.~Singh.
\newblock On the speed of a cookie random walk.
\newblock \emph{Probab. Theory Related Fields}, 141:625--645, 2008.

\bibitem{Basdevant2008b} A.~Basdevant and A.~Singh.
\newblock Rate of growth of a transient cookie random walk.
\newblock \emph{Elec. J. Probab.}, 13:811--851, 2008.

\bibitem{Benjamini2003} I.~Benjamini and D.~B. Wilson.
\newblock Excited random walk.
\newblock \emph{Elec. Comm. Probab.}, 8:86--92, 2003.

\bibitem{Gnedenko1968} B.~V. Gnedenko and A.~N. Kolmogorov.
\newblock \emph{Limit Distributions for Sums of Independent Random Variables}.
\newblock Addison-Wesley, 1968.
\newblock Translated from the Russian by K. L. Chung.

\bibitem{Gut2009} A.~Gut.
\newblock \emph{Stopped Random Walks}.
\newblock Springer, second edition, 2009.

\bibitem{Kesten1975} H.~Kesten, M.~V. Kozlov, and F.~Spitzer.
\newblock A limit law for random walk in a random environment.
\newblock \emph{Composito Math.}, 30:145--168, 1975.

\bibitem{Kosygina2011} E.~Kosygina and T.~Mountford.
\newblock Limit laws of transient excited random walks on integers.
\newblock \emph{Annales de l'Institut Henri Poincar\'{e} - Probabilit\'{e}s
  Statistiques}, 47(2):575--600, 2011.

\bibitem{Kosygina2017} E.~Kosygina and J.~Peterson.
\newblock Excited random walks with {Markovian} cookie stacks.
\newblock \emph{Annales de l'Institut Henri Poincar\'{e} - Probabilit\'{e}s
  Statistiques}, 53(3):1458--1497, 2017.

\bibitem{Kosygina2008} E.~Kosygina and M.~Zerner.
\newblock Positively and negatively excited random walks on integers, with
  branching processes.
\newblock \emph{Elec. J. Probab.}, 13:1952--1979, 2008.

\bibitem{Kosygina2013} E.~Kosygina and M.~Zerner.
\newblock Excited random walks: results, methods, open problems.
\newblock \emph{Bull. Inst. Math. Acad. Sin. (N.S.)}, 8(1):105--157, 2013.

\bibitem{Kosygina2014} E.~Kosygina and M.~Zerner.
\newblock Excursions of excited random walks on integers.
\newblock \emph{Elec. J. Probab.}, 19(25):1--25, 2014.

\bibitem{Kozma2016} G.~Kozma, T.~Orenshtein, and I.~Shinkar.
\newblock Excited random walk with periodic cookies.
\newblock \emph{Annales de l'Institut Henri Poincar\'{e} - Probabilit\'{e}s
  Statistiques}, 52(3):1023--1049, 2016.

\bibitem{Levin2009} D.~Levin, Y.~Peres, and E.~Wilmer.
\newblock \emph{Markov Chains and Mixing Times}.
\newblock Providence, R.I. American Mathematical Society, 2009.

\bibitem{Mountford2006} T.~Mountford, L.~P.~R. Pimentel, and G.~Valle.
\newblock On the speed of the one-dimensional excited random walk in the
  transient regime.
\newblock \emph{Alea}, 2:279--296, 2006.

\bibitem{Peterson2012} J.~Peterson.
\newblock Large deviations and slowdown asymptotics for one-dimensional excited
  random walks.
\newblock \emph{Elec. J. Probab.}, 17(48):1--24, 2012.

\bibitem{Petrov1975} V.~V. Petrov.
\newblock \emph{Sums of Independent Random Variables}.
\newblock Spring-Verlag, 1975.
\newblock Translated from the Russian by A. A. Brown.

\bibitem{Pinsky2010} R.~G. Pinsky.
\newblock Transience/recurrence and the speed of a one-dimensional random walk
  in a ``have your cookie and eat it'' environment.
\newblock \emph{Annales de l'Institut Henri Poincar\'{e} - Probabilit\'{e}s
  Statistiques}, 46(4):949--964, 2010.

\bibitem{Ross1996} S.~Ross.
\newblock \emph{Stochastic Processes}.
\newblock Wiley, second edition, 1996.

\bibitem{Toth1995} B.~T\'{o}th.
\newblock The ``true'' self-avoiding walk with bond repulsion on \textbf{Z}:
  limit theorems.
\newblock \emph{Ann. Probab.}, 23(4):1523--1556, 1995.

\bibitem{Toth1996} B.~T\'{o}th.
\newblock Generalized {Ray-Knight} theory and limit theorems for
  self-interacting random walks on \textbf{Z}$^1$.
\newblock \emph{Ann. Probab.}, 24(3):1324--1367, 1996.

\bibitem{Zerner2005} M.~Zerner.
\newblock Multi-excited random walks on integers.
\newblock \emph{Probab. Theory Related Fields}, 133:98--122, 2005.

\end{thebibliography}



\ACKNO{I thank Jonathon Peterson for helpful discussion on excited random
walks, and in particular for making me aware of the diffusion approximation
method for analyzing the backward branching process from previous works. I
also thank an anonymous referee for comments that led to significant
improvements in the paper.}

\end{document}